\documentclass[11pt]{amsart}
\pdfoutput=1
\usepackage[english]{babel}
\usepackage{graphicx,amsmath}
\usepackage{enumitem}
\usepackage{MnSymbol}
\usepackage{subfig}
\usepackage{float}
\usepackage{floatflt,young,youngtab}
\usepackage{tikz}
\usepackage{relsize}
\usepackage[normalem]{ulem}
\usepackage{cancel} 
\numberwithin{equation}{section}

\newtheorem{theorem}{Theorem}[section]
\newtheorem{corollary}[theorem]{Corollary}
\newtheorem{lemma}[theorem]{Lemma}
\newtheorem{proposition}[theorem]{Proposition}
\newtheorem{remark}[theorem]{Remark}

\newtheorem{definition}[theorem]{Definition}

\newtheorem{example}[theorem]{Example}
\newtheorem{conjecture}[theorem]{Conjecture}

\def \R {{\mathbb R}}

\def \GTNN {{Gr^{\mbox{\tiny TNN}} (k,n)}}

\def \S {{\mathcal S}_{\mathcal M}^{\mbox{\tiny TNN}}}

\title[Internal edge vectors and Talaska formula]{Internal edge vectors on plabic networks in the disk and a generalization of Talaska formula}
\author{Simonetta Abenda}
\address{Dipartimento di Matematica and Alma Mater Research Center on Applied Mathematics, Universit\`a di Bologna, Italy\\ INFN, sez. di Bologna, Italy}
\email{simonetta.abenda@unibo.it}
\author{Petr G. Grinevich}
\address {Steklov Mathematical Institute of Russian Academy of Sciences, Moscow, Russia\\
L.D.Landau Institute for Theoretical Physics, Chernogolovka, Russia\\
Lomonosov Moscow State University, Faculty of Mechanics and Mathematics, Moscow, Russia}
\email{pgg@landau.ac.ru}

\thanks{
This research has been partially supported by GNFM-INDAM and RFO University of Bologna, by the Russian Foundation for Basic Research, grant 20-01-00157. Partially this research was fulfilled during the visit of the second author (P.G.) to IHES, Université Paris-Saclay, France in November 2017.}

\begin{document}

\begin{abstract}

  Following \cite{Pos}, positroid cells $\S$ in totally non-negative Grassmannians $\GTNN$ admit parametrizations by positive weights on planar bicolored directed perfect networks in the disk. An explicit formula for elements of matrices representing the points in  $\S$ was obtained in \cite{Tal2} in terms of flows on such networks. 

  The formulas from  \cite{Pos,Tal2} are defined on the boundary edge vectors. In this paper we propose an extension of these formulas for vectors on internal edges defined as summations over paths on the given directed network gauged by the choice of a ray direction. This gauge choice does not affect the boundary edge vectors, which generate the Postnikov boundary measurement map. The systems of internal edge vectors corresponding to different choices of gauge ray directions coincide up to sign, the sign rule admits a simple explicit description. 

  We prove that the components of these edge vectors are rational in the weights with subtraction--free denominators. Moreover, these components are expressed in terms of internal edge flows; these formulas extend the original Talaska ones to the internal edges. 
 
  These vectors also solve the system of geometric relations associated to the corresponding network. These relations are full rank and  respect the total non--negativity property on the full positroid cell. 

  We also provide explicit formulas both for the transformation rules of the edge vectors with respect to the orientation, and for their transformations due to moves and reductions of networks.

\medskip \noindent {\sc{2010 MSC.}} 14M15; 05C10, 05C22.

\noindent {\sc{Keywords.}} Totally non-negative Grassmannians, positroid cells, planar bicolored networks in the disk, moves and reductions, boundary measurement map, edge vectors.
\end{abstract}

\maketitle

\tableofcontents

\section{Introduction}
Totally non--negative Grassmannians $\GTNN$ historically first appeared as a special case of the generalization to reductive Lie groups by Lusztig  \cite{Lus1,Lus2} of the classical notion of total positivity \cite{GK,GK2,Sch,Kar}. As for classical total positivity, $\GTNN$ naturally arise in relevant problems in different areas of mathematics and physics. The combinatorial objects introduced by Postnikov \cite{Pos}, see also \cite{Rie}, to characterize $\GTNN$ have been linked to the theory of cluster algebras of Fomin-Zelevinsky  \cite{FZ1,FZ2} in \cite{Sc,OPS}. The topological characterization of $\GTNN$ is provided in \cite{GKL} (see also \cite{PSW,RW}).

In particular the planar bicolored (plabic) graphs introduced in \cite{Pos} have appeared in many contexts, such as the topological classification of real forms for isolated singularities of plane curves \cite{FPS}, they are on--shell diagrams (twistor diagrams) in scattering amplitudes in $N=4$ supersymmetric Yang--Mills theory \cite{AGP1,AGP2,ADM} and have a statistical mechanical interpretation as dimer models in the disk \cite{Lam1}. Totally non-negative Grassmannians naturally appear in many other areas, including the theory of Josephson junctions \cite{BG}, statistical mechanical models such as the asymmetric exclusion process \cite{CW} and in the theory of integrable systems \cite{CK,KW1,KW2,AG1,AG2,AG3,AGPR}. In particular, a relevant application of the totally non-negative Grassmannians is connected with the study of the asymptotic geometry of real regular multisoliton solutions of the Kadomtsev-Petviashvili II equation (KP II) and their tropicalizations \cite{CK,KW1,KW2}. It is well-known that such solutions can be constructed by degenerating the finite-gap solutions, but how to do this degeneration in the class of \textbf{real regular} finite-gap solutions was only recently understood using the combinatorics of the plabic graphs, see \cite{AG1,AG2,AG3}.

The motivation to the present research comes from problems of mathematical and theoretical physics where total positivity is connected to some measurable outcome at the boundary of the graph due to real local interactions occurring at its vertices. In particular, to construct the KP II divisor on the rational $\mathtt M$-curve corresponding to a given network, one has to know the wave function at all internal edges. In \cite{AG3} this problem was solved only for the canonically oriented Le--network. Using the results of the present paper, it is solved in \cite{AG6} for any plabic network. 

Moreover, the edge vectors introduced in the present paper fulfill a linear system of equations at the vertices. The latter system may be described in terms of the amalgamation of cluster varieties originally introduced by Fock and Goncharov in \cite{FG1}, which has relevant applications in cluster algebras and relativistic quantum field theory \cite{AGP1,AGP2,Kap,MS}. In connection to relevant open problems in theoretical physics, Lam (see \cite{Lam2}, Section 14) has proposed to use spaces of relations on planar bipartite graphs to represent amalgamation in totally non--negative Grassmannians and to characterize their maximal rank and total non-negativity properties in terms of admissible edge signatures on the final planar graph. Using the results of the present paper, in \cite{AG7} we prove that the geometric signature associated to the relations fulfilled by the edge vectors solves this problem of Lam. Finally, in \cite{A3}, geometric signatures are proven to fulfil  Speyer's variant \cite{Sp} of the classical Kasteleyn theorem \cite{Kas1} in the case of reduced bipartite plabic graphs.

\smallskip

\textbf{Main results} On a given plabic network $\mathcal N$ in the disk representing a point in $\S$, we introduce a system of edge vectors at all internal edges. The $j$--th edge vector component on $e$ is defined as a summation over all directed paths from $e$ to the boundary sink $b_j$. The absolute value of the contribution of one such path is the product of the edge weights counted with their multiplicities, whereas its sign depends on the sum of two indices: the generalized winding index of the path with respect to a chosen gauge direction $\mathfrak l$, and  the number of intersections of the path with the gauge rays starting at the boundary sources. We show that for boundary edges this sign coincides with the sum of the topological winding used in  \cite{Pos} plus the number of boundary sources passed by a directed path from boundary to boundary, therefore these edge components coincide with the boundary measurement matrix entries in  \cite{Pos}. The idea of fixing a ray direction to measure locally the winding first appeared in \cite{GSV}. 

Given a perfectly oriented graph with fixed gauge ray direction, we show that:
\begin{enumerate}
\item For any choice of positive edge weights, these edge vectors solve a full rank system of relations (Theorem \ref{theo:consist});
\item The vector components at internal edges are rational in the weights with subtraction--free denominators and are explicitly computed using conservative and edge flows, thus extending the results by Talaska in \cite{Tal2} to the interior of the graph (Theorem \ref{theo:null});
\item The solution of such system at the boundary sources provides the boundary measurement matrix associated to such network by Postnikov \cite{Pos} (see Corollary \ref{cor:bound_source}).
\end{enumerate}
Finally, we explicitly characterize how edge vectors change with respect to changes of orientation, of gauge ray direction (Section \ref{sec:vector_changes} and Appendix \ref{app:orient}), and with respect to Postnikov moves and reductions (Section \ref{sec:moves_reduc}).

In particular, if the graph is reduced in Postnikov sense, the vector components at all internal edges are subtraction--free rational expressions in the weights, therefore they satisfy the stronger condition settled for the boundary measurement map in \cite{Pos} (see Theorem \ref{thm:null_acyclic}). On the contrary, null edge vectors may appear in reducible networks even if there exist paths from the given edge to the boundary sinks (see Example \ref{example:null}). In such case, we conjecture that it is possible to obtain non--zero edge vectors using the extra gauge freedom of weights on reducible networks. 

\smallskip

\textbf{Final remarks and open problems}
In our construction we use $n$--row vectors and perfectly oriented trivalent plabic networks because this representation is suitable for the mathematical formulation of several problems connected to total non--negativity \cite{AG1, AG3, AG6,AGP1,AGP2, CK, KW1, KW2}.
We remark that the valency condition is by no means restrictive and that the formulation of the same problem in terms of $k$--column vectors is straightforward and amounts to exchange relations at white and black vertices. 

In \cite{GSV1} it is proven that the boundary measurement map possesses a natural Poisson-Lie structure, compatible with the natural cluster algebra structure on such Grassmannians. An interesting open question is how to use such Poisson--Lie structure in association with our geometric approach. Another open problem is to extend our construction to planar graphs on orientable surfaces with boundaries. Indeed, an extension of Talaska formula for such graphs was conjectured in \cite{FGPW} and proven in \cite{Mach}. Moreover, in such case it is necessary to modify the present construction using the procedure established in \cite{GSV2} to extend the boundary measurement map to planar networks in the annulus.

We plan to pursue such detailed construction in a different paper with the aim of generalizing the construction of KP-II divisors for other classes of soliton solutions and compare it with the so--called top-down approach for non-planar diagrams from gluing legs which plays a relevant role in the computation of scattering amplitudes of field theoretical models \cite{AGP1, AGP2, BFGW}. 

\section{Plabic networks and totally non--negative Grassmannians}\label{sec:plabic_graphs}

In this Section we recall some basic definitions on totally non--negative Grassmannians and define the class of graphs $\mathcal G$ representing a given positroid cell which we use throughout the text. 
We use the following notations throughout the paper:
\begin{enumerate}
\item $k$ and $n$ are positive integers such that $k<n$;
\item  For $s\in {\mathbb N}$  $[s] =\{ 1,2,\dots, s\}$; if $s,j \in {\mathbb N}$, $s<j$, then
$[s,j] =\{ s, s+1, s+2,\dots, j-1,j\}$;
\end{enumerate}

\begin{definition}\textbf{Totally non-negative Grassmannian \cite{Pos}.}
Let $Mat^{\mbox{\tiny TNN}}_{k,n}$ denote the set of real $k\times n$ matrices of maximal rank $k$ with non--negative maximal minors $\Delta_I (A)$. Let $GL_k^+$ be the group of $k\times k$ matrices with positive determinants. We define a totally non-negative Grassmannian as 
\[
\GTNN = GL_k^+ \backslash Mat^{\mbox{\tiny TNN}}_{k,n}.
\]
\end{definition}
In the theory of totally non-negative Grassmannians an important role is played by the positroid stratification. Each cell in this stratification is defined as the intersection of a Gelfand-Serganova stratum \cite{GS,GGMS} with the totally non-negative part of the Grassmannian. More precisely:
\begin{definition}\textbf{Positroid stratification \cite{Pos}.} Let $\mathcal M$ be a matroid i.e. a collection of $k$-element ordered subsets $I$ in $[n]$, satisfying the exchange axiom (see, for example \cite{GS,GGMS}). Then the positroid cell $\S$ is defined as
$$
\S=\{[A]\in \GTNN\ | \ \Delta_{I}(A) >0 \ \mbox{if}\ I\in{\mathcal M} \ \mbox{and} \  \Delta_{I}(A) = 0 \ \mbox{if} \ I\not\in{\mathcal M}  \}.
$$
A positroid cell is irreducible if, for any $j\in [n]$, there exist $I, J\in \mathcal M$ such that $j\in I$ and $j\not\in J$.
\end{definition}
The combinatorial classification of all non-empty positroid cells and their rational parametrizations were obtained in \cite{Pos}, \cite{Tal2}. In our construction we use the classification of positroid cells via directed planar networks in the disk in \cite{Pos}. More precisely, we use the following class of graphs ${\mathcal G}$ introduced by Postnikov \cite{Pos}:
\begin{definition}\label{def:graph} \textbf{Planar bicolored directed trivalent perfect graphs in the disk (plabic graphs).} A graph ${\mathcal G}$ is called plabic if:
\begin{enumerate}
\item  ${\mathcal G}$ is planar, directed and lies inside a disk. Moreover ${\mathcal G}$ is connected in the sense it does not possess components isolated from the boundary;
\item It has finitely many vertices and edges;
\item It has $n$ boundary vertices on the boundary of the disk labeled $b_1,\cdots,b_n$ clockwise. Each boundary vertex has degree 1. We call a boundary vertex $b_i$ a source (respectively sink) if its edge is outgoing (respectively incoming);
\item The remaining vertices are called internal and are located strictly inside the disk. They are either bivalent or trivalent;
\item ${\mathcal G}$ is a perfect graph, that is each internal vertex in  ${\mathcal G}$ is incident to exactly one incoming edge or to one outgoing edge. In the first case the vertex is colored white, in the second case black. Bivalent vertices are assigned either white or black color.
\end{enumerate}
A face of the graph is called \textbf{internal} if it does not contain boundary vertices, otherwise is called external. The external face containing the boundary vertices $b_n$, $b_1$ in clockwise order is called infinite, all other faces are called finite.

Moreover, to simplify the overall construction we further assume that the boundary vertices $b_j$, $j\in [n]$ lie on a common interval in the boundary of the disk.
\end{definition}

\begin{remark}
\begin{enumerate}
\item The trivalency assumption is not restrictive, since any perfect plabic graph can be transformed into a trivalent
  one.
\item The assumption that the boundary vertices $b_j$, $j\in [n]$ lie on a common interval in the boundary of the disk considerably simplifies the use of gauge ray directions to assign winding numbers to walks starting at internal edges and to count the number of boundary source points passed by such walks.  
\end{enumerate}
\end{remark}
In Figure \ref{fig:Rules0} we present an example of a plabic graph satisfying Definition~\ref{def:graph} and representing a 10-dimensional positroid cell in $Gr^{\mbox{\tiny TNN}}(4,9)$. 

The class of perfect orientations of the plabic graph ${\mathcal G}$ are those which are compatible with the coloring of the vertices. The graph is of type $(k,n)$ if it has $n$ boundary vertices and $k$ of them are boundary sources. Any choice of perfect orientation preserves the type of ${\mathcal G}$. To any perfect orientation $\mathcal O$ of ${\mathcal G}$ one assigns the base $I_{\mathcal O}\subset [n]$ of the $k$-element source set for $\mathcal O$. Following \cite{Pos} the matroid of ${\mathcal G}$ is the set of $k$-subsets  $I_{\mathcal O}$ for all perfect orientations:
$$
\mathcal M_{\mathcal G}:=\{I_{\mathcal O}|{\mathcal O}\ \mbox{is a perfect orientation of}\ \mathcal G \}.
$$
In \cite{Pos} it is proven that $\mathcal M_{\mathcal G}$ is a totally non-negative matroid $\mathcal S^{\mbox{\tiny TNN}}_{\mathcal M_{\mathcal G}}\subset \GTNN$. The following statements are straightforward adaptations of more general statements of \cite{Pos} to the case of  plabic graphs:
\begin{theorem}
A plabic graph $\mathcal G$ can be transformed into a plabic graph $\mathcal G'$ via a finite sequence of Postnikov moves and reductions if and only if $\mathcal M_{\mathcal G}=\mathcal M_{\mathcal G'}$.
\end{theorem}

A graph  $\mathcal G$ is reduced if there is no other graph in its move reduction equivalence class which can be obtained from $\mathcal G$ applying a sequence of transformations containing at least one reduction. Each positroid cell $\S$ is represented by at least one reduced graph, the so called Le--graph, associated to the Le--diagram representing $\S$ and it is possible to assign weights to such graphs in order to obtain a global parametrization of $\S$ \cite{Pos}.

\begin{remark}
If a positroid cell is irreducible, then the plabic graphs representing it do not possess isolated boundary vertices.
\end{remark}

\begin{proposition}\cite{Pos}
Each Le-graph is reduced and may be transformed into a reduced plabic Le-graph. 
If $\mathcal G$ is a reduced  plabic graph, then the dimension of  $\mathcal S^{\mbox{\tiny TNN}}_{\mathcal M_{\mathcal G}}$ is equal to the number of faces of $\mathcal G$ minus 1.
\end{proposition}
The plabic graph in Figure \ref{fig:Rules0} is a plabic Le-graph.

\begin{lemma}\textbf{Relations between vertices, edges, faces}
Let $t_W, t_B, d_W$ and $d_B$ respectively be the number of trivalent white, trivalent black, bivalent white and bivalent black internal vertices of ${\mathcal G}$. Let $n_I$ be the number of internal edges ({\sl i.e.} edges not connected to a boundary vertex) of ${\mathcal G}$. By Euler formula we have $g = n_I +n -(t_W + t_B+d_W+d_B)$. Moreover,
the following identities hold
$3(t_W+t_B)+2(d_W+d_B) = 2n_I +n$, $2t_B+ t_W+d_W+d_B =n_I+k$.
Therefore 
\begin{equation}\label{eq:vertex_type}
t_W = g-k, \qquad t_B = g-n+k, \qquad d_W+d_B= n_I+ 2n - 3g.
\end{equation}
\end{lemma}

\section{Systems of edge vectors on plabic networks}\label{sec:def_edge_vectors}
For any given $[A]\in \mathcal S^{\mbox{\tiny TNN}}_{\mathcal M_{\mathcal G}}$, there exists $\mathcal N$ a network representing $[A]$ with plabic graph $\mathcal G$ for some choice of positive edge weights $w_e$ \cite{Pos}. 

In \cite{Pos}, for any given oriented planar network in the disk it is defined the formal boundary measurement map 
\[
M_{ij} := \sum\limits_{P \, : \, b_i\mapsto b_j} (-1)^{\mbox{\scriptsize
 wind}(P)} w(P),
\] 
where the sum is over all directed walks from the source $b_i$ to the sink $b_j$, $w(P)$ is the product of the edge weights of $P$ and $\mbox{wind}(P)$ is its topological winding index. These formal power series sum up to subtraction--free rational expressions in the weights \cite{Pos} and explicit expressions in function of flows and conservative flows in the network are obtained in \cite{Tal2}. Let $I$ be the base inducing the orientation of $\mathcal N$ used in the computation of the boundary measurement map. Then the point $Meas(\mathcal N)\in Gr(k,n)$ is represented by the boundary measurement matrix $A$ such that:
\begin{itemize}
\item The submatrix $A_I$ in the column set $I$ is the identity matrix;
\item The remaining entries $A^r_j = (-1)^{\sigma(i_r,j)} M_{ij}$, $r\in [k]$, $j\in \bar I$, where $\sigma(i_r,j)$ is the number of elements of $I$ strictly between $i_r$ and $j$.
\end{itemize}
In the following we extend this measurement to the edges of plabic networks in such a way that, if $e_r$ is the edge at the boundary source $b_{i_r}$, then the vector $E_{e_r} = A[r]- E_{i_r}$, with $E_{j}$ the $j$--th vector of the canonical basis. At this aim we introduce gauge rays both to measure the local winding between consecutive edges in the path and to count the number of boundary sources passed by a path from an internal edge to a boundary sink vertex using the number of its intersections with gauge rays starting at the boundary sources. 
In \cite{GSV}, gauge rays were introduced to compute the winding number of a path joining boundary vertices. Here we use it also to generalize the index $\sigma(i_r,j)$ when the path starts at an internal edge $e$.

\begin{figure}
  \centering{\includegraphics[width=0.55\textwidth]{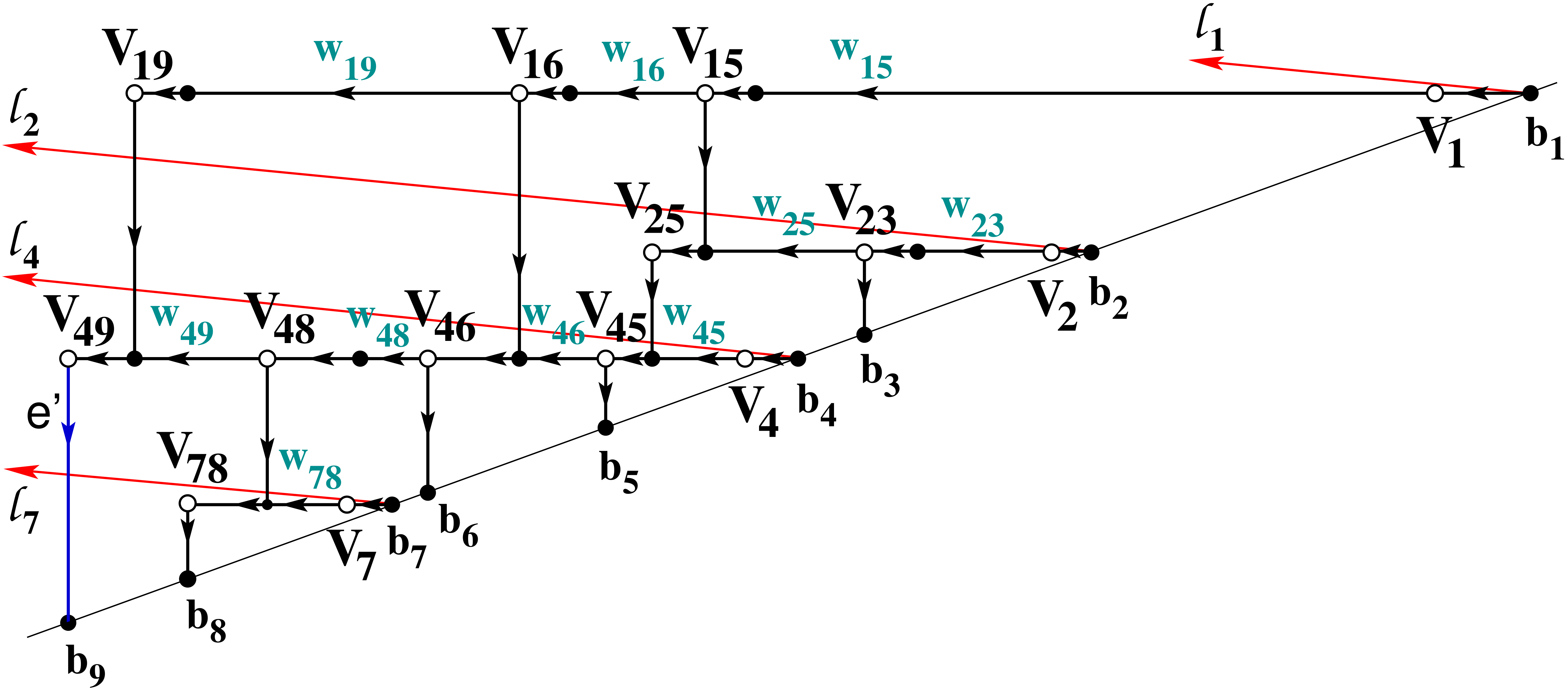}
      \caption{\small{\sl The rays starting at the boundary sources for a given orientation of the 
          network uniquely fix the edge vectors.
					}}\label{fig:Rules0}}
\end{figure}

\begin{definition}\label{def:gauge_ray}\textbf{The gauge ray direction $\mathfrak{l}$.}
A gauge ray direction is an oriented direction ${\mathfrak l}$ with the following properties:
\begin{enumerate}
\item The ray ${\mathfrak l}$ starting at a boundary vertex points inside the disk; 
\item No internal edge is parallel to this direction;
\item All rays starting at boundary vertices do not contain internal vertices of the network.
\end{enumerate}
\end{definition}
We remark that the first property may always be satisfied since all boundary vertices lie at a common straight interval in the boundary of $\mathcal N$. We then define the local winding number between a pair of consecutive edges $e_k,e_{k+1}$ as follows.

\begin{figure}
\centering{\includegraphics[width=0.9\textwidth]{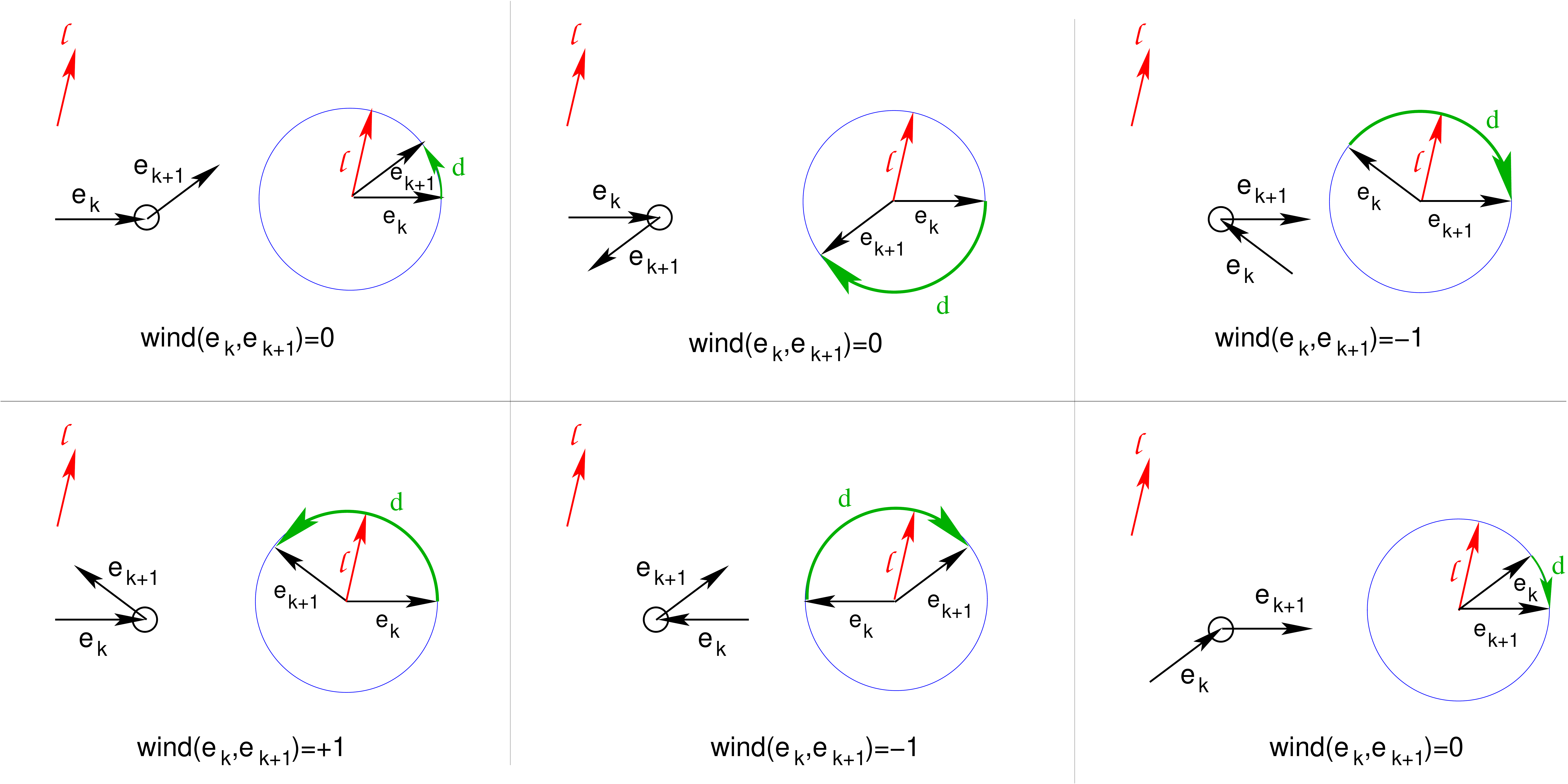}}
\caption{\small{\sl The local rule to compute the winding number.}}\label{fig:winding}
\end{figure}
\begin{definition}\label{def:winding_pair}\textbf{The local winding number at an ordered pair of oriented edges}
Let  $(e_k,e_{k+1})$ be an ordered pair of oriented edges. If they are not antiparallel, let us define
\begin{equation}\label{eq:def_s}
s(e_k,e_{k+1}) = \left\{
\begin{array}{ll}
+1 & \mbox{ if the ordered pair is positively oriented }  \\
0  & \mbox{ if } e_k \mbox{ and } e_{k+1} \mbox{ are parallel }\\
-1 & \mbox{ if the ordered pair is negatively oriented }
\end{array}
\right.
\end{equation}
Then the winding number of the ordered pair $(e_k,e_{k+1})$ with respect to the gauge ray direction $\mathfrak{l}$ is
\begin{equation}\label{eq:def_wind}
\mbox{wind}(e_k,e_{k+1}) = \left\{
\begin{array}{ll}
+1 & \mbox{ if } s(e_k,e_{k+1}) = s(e_k,\mathfrak{l}) = s(\mathfrak{l},e_{k+1}) = 1\\
-1 & \mbox{ if } s(e_k,e_{k+1}) = s(e_k,\mathfrak{l}) = s(\mathfrak{l},e_{k+1}) = -1\\
0  & \mbox{otherwise}.
\end{array}
\right.
\end{equation}
We illustrate the rule in Figure \ref{fig:winding}.

In the non generic case of ordered antiparallel edges, we slightly rotate the pair $(e_k, e_{k+1})$ to $(e_k^{\prime},
e_{k+1}^{\prime})$ as in Figure \ref{fig:antipar} and define
\begin{equation}\label{eq:s_antipar}
\mbox{wind}(e_k,e_{k+1}) = \lim_{\epsilon\to 0^+} \mbox{wind} (e_k^{\prime},e_{k+1}^{\prime}).
\end{equation}

\end{definition}
\begin{figure}
  \centering{\includegraphics[width=0.49\textwidth]{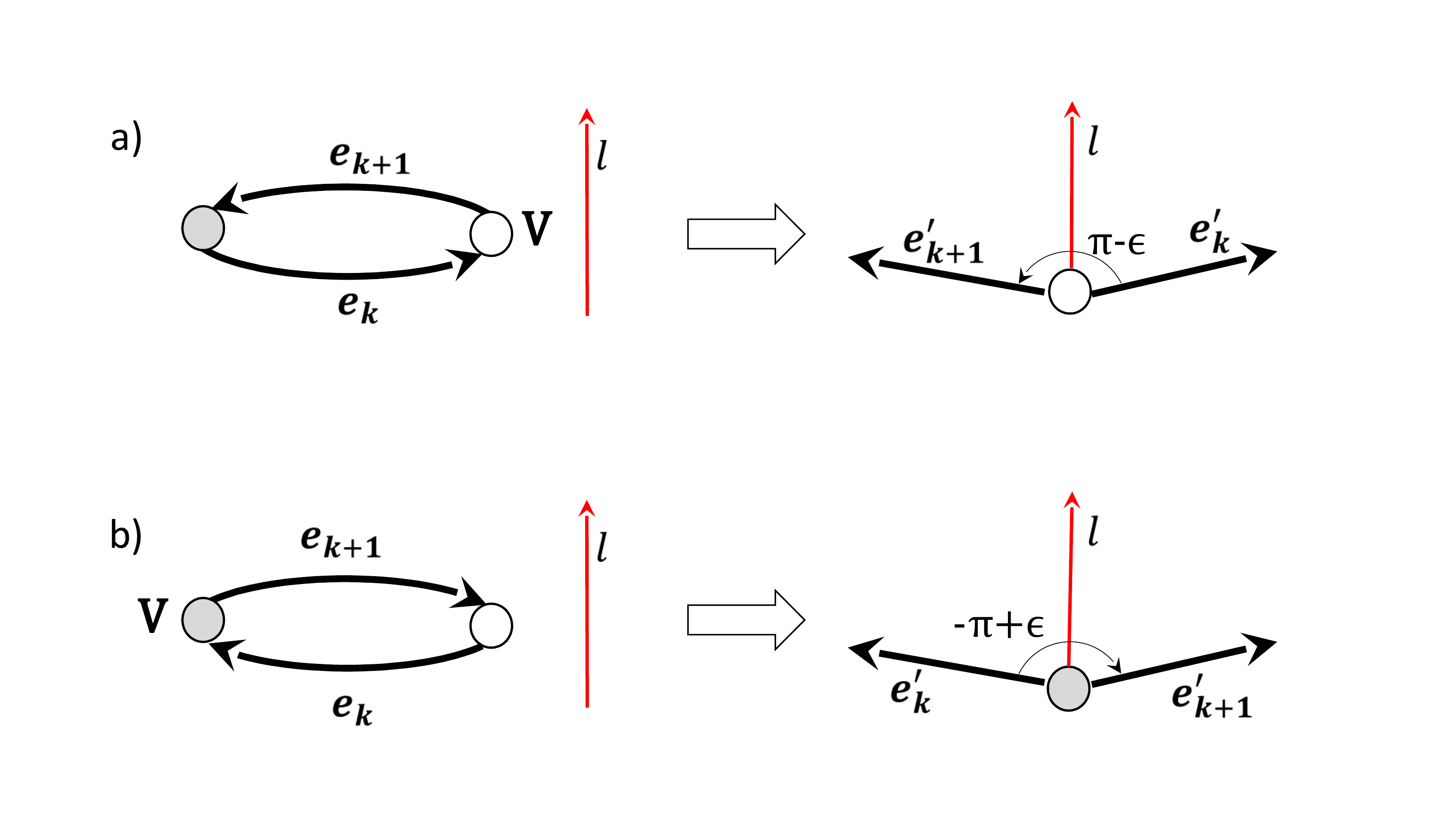}
	\includegraphics[width=0.49\textwidth]{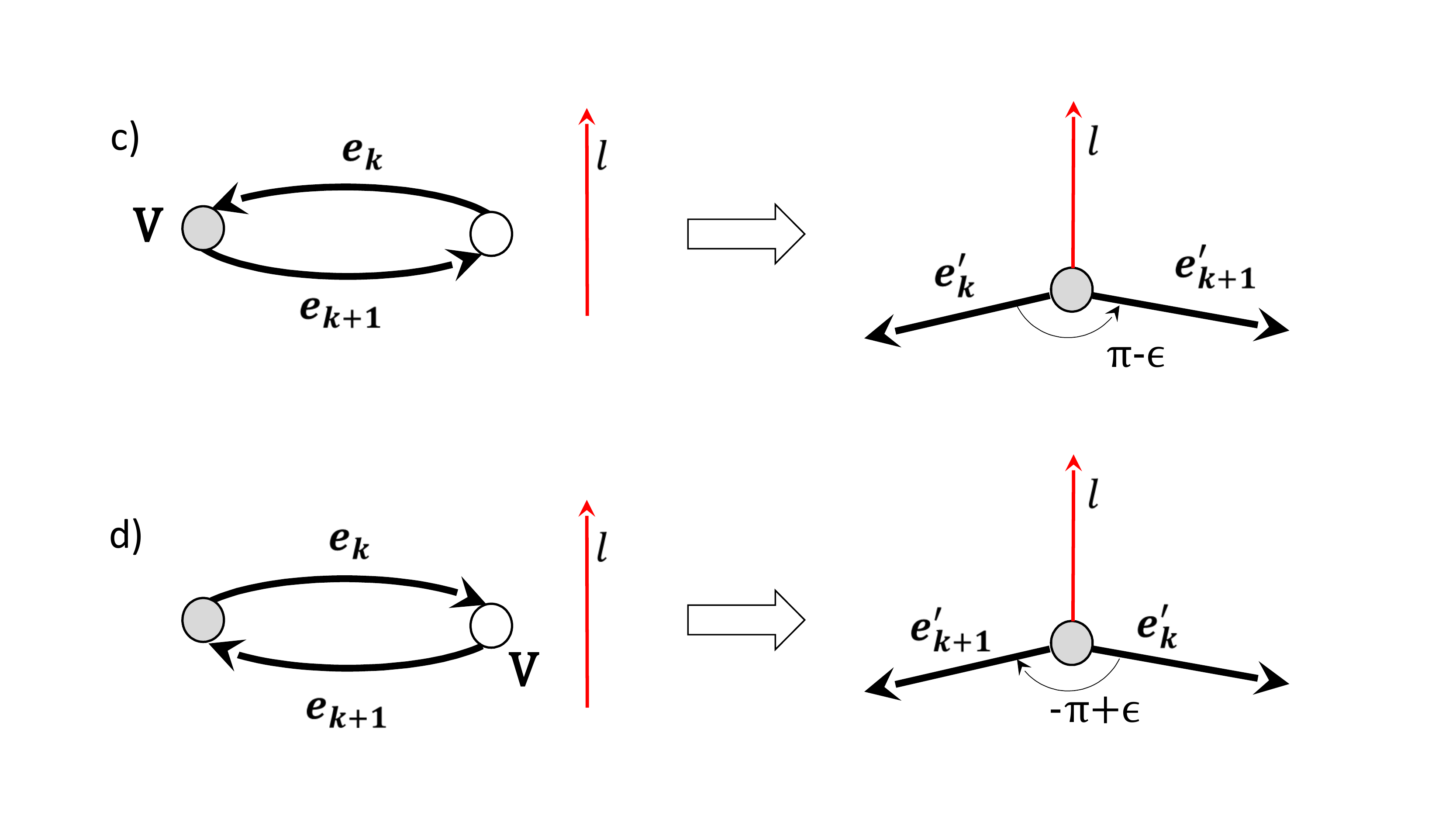}
      \caption{\small{\sl If the ordered pair $(e_k, e_{k+1})$ is antiparallel at $V$, we slightly rotate the two edge vectors at $V$ to compute $\mbox{wind}(e_k,e_{k+1})$. Using (\ref{eq:s_antipar}) and (\ref{eq:def_wind}), we get $a)$: $\mbox{wind}(e_k,e_{k+1})=1$; $b)$: $\mbox{wind}(e_k,e_{k+1})=-1$; $c)$: $\mbox{wind}(e_k,e_{k+1})=0$; $d)$: $\mbox{wind}(e_k,e_{k+1})=0$. }}\label{fig:antipar}}
\end{figure}

The local winding defined above has the following properties:

\begin{lemma}
\label{lem:rotation}  
\begin{enumerate}   
\item If we keep $e_k$, $e_{k+1}$ fixed and rotate the gauge direction $\mathfrak l$, $\mbox{wind}(e_k,e_{k+1})$ changes by $\pm 1$ each time $\mathfrak l$ passes $e_k$ or $e_{k+1}$;
\item If we keep $e_k$, $\mathfrak l$ fixed and  rotate $e_{k+1}$,  $\mbox{wind}(e_k,e_{k+1})$ changes by $\pm 1$ each time $e_{k+1}$  passes  $\mathfrak l$ or  $-e_k$.
\end{enumerate}
\end{lemma}  
The proof is straightforward.

Let $b_{i_r}$, $r\in[k]$, $b_{j_l}$, $l\in[n-k]$, respectively  be the set of boundary sources and boundary sinks 
associated to the given orientation. Then draw the rays ${\mathfrak l}_{i_r}$, $r\in[k]$, starting at $b_{i_r}$ associated with the pivot columns of the given orientation. In Figure \ref{fig:Rules0} we show an example.

Let us now consider a directed path ${\mathcal P}=\{e=e_1,e_2,\cdots, e_m\}$ starting at a vertex $V_1$ (either a boundary source or internal vertex) and ending at a 
boundary sink $b_j$, where $e_1=(V_1,V_2)$, $e_2=(V_2,V_3)$, \ldots, $e_m=(V_m,b_j)$. At each edge the orientation of the path coincides with the orientation of this edge in the graph.
We assign three numbers to ${\mathcal P}$:
\begin{enumerate}
\item The \textbf{weight $w({\mathcal P})$} is simply the product of the weights $w_l$ of all edges $e_l$ in ${\mathcal P}$, $w({\mathcal P})=\prod_{l=1}^m w_l$. If we pass the 
same edge $e$ of weight $w_e$ $r$ times, the weight is counted as $w_e^r$;
\item The \textbf{generalized winding number} $\mbox{wind}({\mathcal P})$ is the sum of the local winding numbers at each ordered pair of its edges
$\mbox{wind}(\mathcal P) = \sum_{k=1}^{m-1} \mbox{wind}(e_k,e_{k+1}),$ 
with $\mbox{wind}(e_k,e_{k+1})$ as in Definition \ref{def:winding_pair}; 
\item $\mbox{int}(\mathcal P)$ is the \textbf{number  of intersections} between the path and the rays ${\mathfrak l}_{i_r}$, $r\in[k]$: $\mbox{int}(\mathcal P) = \sum\limits_{s=1}^m \mbox{int}(e_s)$, where $\mbox{int}(e_s)$ is the number of intersections of gauge rays ${\mathfrak l}_{i_r}$ with $e_s$.
\end{enumerate}
The generalized winding of the path $\mathcal P$ depends on the gauge ray direction $\mathfrak{l}$ since it
counts how many times the tangent vector to the 
path is parallel and has the same orientation as ${\mathfrak l}$; also the number of intersections $\mbox{int}(\mathcal P)$ depends on $\mathfrak{l}$. 

\begin{definition}\label{def:edge_vector}\textbf{The edge vector $E_e$.}
For any edge $e$, let us consider all possible directed paths ${\mathcal P}:e\rightarrow b_{j}$, 
in $({\mathcal N},{\mathcal O},{\mathfrak l})$ such that the first edge is $e$ and the end point is the boundary vertex $b_{j}$, $j\in[n]$.
Then the $j$-th component of $E_{e}$ is defined as:
\begin{equation}\label{eq:sum}
\left(E_{e}\right)_{j} = \sum\limits_{{\mathcal P}\, :\, e\rightarrow b_{j}} (-1)^{\mbox{wind}({\mathcal P})+ \mbox{int}({\mathcal P})} 
w({\mathcal P}).
\end{equation}
If there is no path from $e$ to $b_{j}$, the $j$--th component of $E_e$ is assigned to be zero. By definition, at the edge $e$ at the boundary sink $b_j$, the edge vector $E_{e}$ is
\begin{equation}\label{eq:vec_bou_sink}
\left(E_{e}\right)_{k} =  (-1)^{\mbox{int}(e)} w(e) \delta_{jk}.
\end{equation}
\end{definition}

In particular, if $b_j$ is a boundary source, then for any $e$, the $j$-th component of $E_e$ is equal to zero. 
If $e$ is an edge belonging to the connected component of an isolated boundary sink $b_j$, 
then $E_e$ is proportional to the $j$--th vector of the canonical basis, whereas $E_e$ is the null vector if $e$ is an edge belonging to the connected component of an isolated boundary source.

If the number of paths starting at $e$ and ending at $b_j$ is finite for a given edge $e$ and destination $b_j$, the component $\left(E_{e}\right)_{j}$  
in (\ref{eq:sum}) is a polynomial in the edge weights.

If the number of paths starting at $e$ and ending at $b_j$ is infinite and the weights are sufficiently small, it is easy to check that the right hand side in (\ref{eq:sum}) converges. In Section 
\ref{sec:rational} we adapt the summation procedures of \cite{Pos} and \cite{Tal2} to prove that the edge vector components are rational expressions with subtraction-free denominators and provide explicit expressions in Theorem \ref{theo:null}. 

\subsection{Edge--loop erased walks, conservative and edge flows}\label{sec:flows}

Next we study the structure of the expressions representing the components of the edge vectors.

First, following \cite{Fom,Law}, we adapt the notion of loop-erased walk to our situation, since our walks start at an edge, not at a vertex. 
\begin{definition}
\label{def:loop-erased-walk}
\textbf{Edge loop-erased walks.}  Let ${\mathcal P}$ be a walk (directed path) given by
$$
V_e \stackrel{e}{\rightarrow} V_1 \stackrel{e_1}{\rightarrow} V_2 \rightarrow \ldots \rightarrow b_j,
$$
where $V_e$ is the initial vertex of the edge $e$. The edge loop-erased part of  ${\mathcal P}$, denoted $LE({\mathcal P})$, is defined recursively as 
follows. If ${\mathcal P}$ does not pass any edge twice (i.e. all edges $e_i$ are distinct), then $LE({\mathcal P})={\mathcal P}$.
Otherwise, set $LE({\mathcal P})=LE({\mathcal P}_0)$, where ${\mathcal P}_0$ is obtained from ${\mathcal P}$ removing the first
edge loop it makes; more precisely, given all pairs $l,s$ with $s>l$ and $e_l = e_s$, one chooses the one with the smallest value of $s$ and removes the cycle
$$
V_l \stackrel{e_l}{\rightarrow} V_{l+1} \stackrel{e_{l+1}}\rightarrow V_{l+2} \rightarrow \ldots \stackrel{e_{s-1}}\rightarrow V_{s} ,
$$
from ${\mathcal P}$.
\end{definition}

\begin{remark}\label{rem:loop}
An edge loop-erased walk can pass twice through the first vertex $V_e$, but it cannot pass twice any other vertex due to  perfectness.
For example, the directed path $1,2,3,4,5,6,7,8,12$ at Figure~\ref{fig:loop_erased} is edge loop-erased but it passes twice through the starting 
vertex $V_1$.
In general, the edge loop-erased walk does not coincide with the loop-erased walk defined in \cite{Fom, Law}. For instance, the directed path
$1,2,3,4,5,6,7,8,9,4,11$ has edge loop-erased walk $1,2,3,4,11$ and the loop-erased walk $7,8,9,4,11$. 

The two definitions coincide if $e$ starts at a boundary source.  

In our text we never use loop-erased walks in the sense of \cite{Fom} and we use the notation $LE({\mathcal P})$ in the sense of 
Definition~\ref{def:loop-erased-walk}. 
\end{remark}

\begin{figure}
  \centering{
	\includegraphics[width=0.27\textwidth]{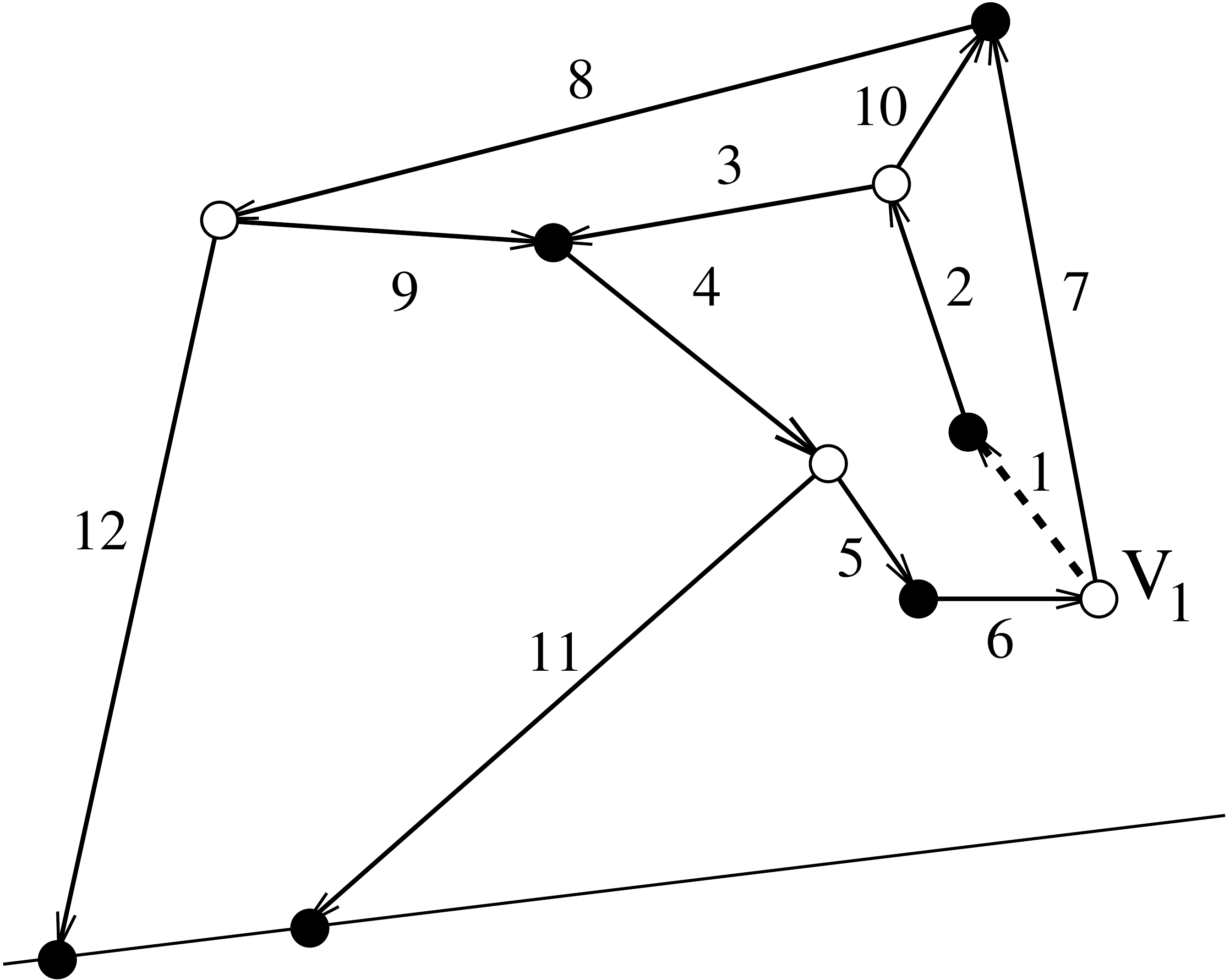}}
  \caption{\small{\sl The graph of Remark \ref{rem:loop}.}\label{fig:loop_erased}}
\end{figure}

With this procedure, to each path starting at $e$ and ending at $b_j$ we associate a unique edge loop-erased walk $LE({\mathcal P})$, where the latter path is either acyclic or possesses one simple cycle passing through the initial vertex.
Then we formally reshuffle the summation over 
infinitely many paths starting at $e$ and ending at $b_j$ to a summation over a finite number of equivalent classes  $[LE({\mathcal P}_s)]$, each one consisting of all paths sharing the 
same edge loop-erased walk, $LE({\mathcal P}_s)$, $s=1,\ldots, S$. Let us remark that 
$ \mbox{int}({\mathcal P})-\mbox{int}(LE({\mathcal P}_s))=0 \,\,(\!\!\!\!\mod 2) $ for any ${\mathcal P}\in [LE({\mathcal P}_s)]$, and
$ \mbox{wind}({\mathcal P})-\mbox{wind}(LE({\mathcal P}_s))$ has the same parity as the number of simple cycles of ${\mathcal P}$ minus the number of 
simple cycles of $LE({\mathcal P}_s)$.
With this in mind, we rexpress (\ref{eq:sum}) as follows
\begin{equation}\label{eq:sum2}
\left(E_{e}\right)_{j} = \sum\limits_{s=1}^S (-1)^{\mbox{wind}(LE({\mathcal P}_s))+ \mbox{int}(LE({\mathcal P}_s))}\left[ 
\mathop{\sum\limits_{{\mathcal P}:e\rightarrow b_{j}}}_{{\mathcal P}\in [LE({\mathcal P}_s)] } (-1)^{\mbox{wind}({\mathcal P})-\mbox{wind}(LE({\mathcal P}_s))  } 
w({\mathcal P}) \right].
\end{equation}

We remark that the winding number along each simple closed loop introduces a $-$ sign in agreement with \cite{Pos}. Therefore the summation over paths may be interpreted as a discretization of path integration in some spinor theory. In typical spinor theories the change of phase during the rotation of the spinor corresponds to standard measure on the group $U(1)$ and requires the use of complex numbers. The introduction of the gauge direction forces the use of $\delta$--type measures instead of the standard measure on $U(1)$, and it permits to work with real numbers only.

Next we adapt the definitions of flows and conservative flows in \cite{Tal2} to our case.

\begin{definition}\label{def:edge_flow}\textbf{Edge flow at $e$}. A collection $F_e$ of distinct edges in a plabic graph $\mathcal G$ is called edge flow starting at the edge $e$ if: 
\begin{enumerate}
\item It contains the edge $e$;
\item For each interior vertex $V_d$ in $\mathcal G$ except the starting vertex of $e$  the number of edges of $F$ that arrive at $V_d$ is equal to the number of edges of $F$ that leave from $V_d$;
\item If $V$ is the starting vertex of $e$, the number of edges of $F$ that arrive at $V$ is equal to the number of edges of $F$ that leave from $V$ minus 1;
\item It contains no boundary edges at sources, except possibly $e$ itself.  
\end{enumerate}
We  denote by   ${\mathcal F}_{ej}(\mathcal G)$ the collection of edge flows at edge $e$ containing the boundary sink $b_j$.  
\end{definition}

\begin{definition}\label{def:cons_flow}\textbf{Conservative flow \cite{Tal2}}. A collection $C$ of distinct edges in a plabic graph $\mathcal G$ is called a conservative flow if 
\begin{enumerate}
\item For each interior vertex $V_d$ in $\mathcal G$ the number of edges of $C$ that arrive at $V_d$ is equal to the number of edges of $C$ that leave from $V_d$;
\item $C$ does not contain edges incident to the boundary.
\end{enumerate}
We denote the set of all conservative flows $C$ in $\mathcal G$  by ${\mathcal C}(\mathcal G)$. In particular, ${\mathcal C}(\mathcal G)$ contains the trivial flow with no edges to which we assign weight 1.
\end{definition}

The conservative flows are collections of non-intersecting simple loops in the directed graph $\mathcal G$.

In our setting an edge flow  $F_{e,b_j}$ in ${\mathcal F}_{e,b_j}(\mathcal G)$ is either an edge loop-erased walk $P_{e,b_j}$ starting at the edge $e$ and ending at the boundary sink $b_j$ or the union of $P_{e,b_j}$ with a conservative flow with no common edges with $P_{e,b_j}$. In particular, our definition of edge flow coincides with the definition of flow in \cite{Tal2} if $e$ starts at a boundary source.

Next we assign weight, winding and intersection numbers to edge flows, and weight to conservative flows. We remark that in \cite{Tal2} there is no winding nor intersection number assigned to flows from boundary to boundary.
\begin{definition}\label{def:numb_flows}
\begin{enumerate}
\item We assign one number to each $C\in {\mathcal C}(\mathcal G)$: the \textbf{weight $w(C)$} is the product of the weights of all edges in $C$.
\item Let  $F_{e,b_j}\in{\mathcal F}_{e,b_j}(\mathcal G)$ be  the union of the edge loop-erased walk $P_{e,b_j}$ with a conservative flow with no common edges with $P_{e,b_j}$ (this conservative flow may be the trivial one). We assign three numbers to $F_{e,b_j}$:
\begin{enumerate} 
\item The \textbf{weight $w(F_e)$} is the product of the weights of all edges in $F_e$.
\item The \textbf{winding number} $\mbox{wind}(F_{e,b_j})$:
\begin{equation}
\label{eq:wind_flow}
\mbox{wind}(F_{e,b_j}) = \mbox{wind}(P_{e,b_j});
\end{equation}
\item The \textbf{intersection number} $\mbox{int}(F_{e,b_j})$:
\begin{equation}
\label{eq:int_flow}
\mbox{int}(F_{e,b_j}) = \mbox{int}(P_{e,b_j}).
\end{equation}
\end{enumerate}
\end{enumerate}
\end{definition}

\subsection{The linear system on $({\mathcal N},\mathcal O,\mathfrak l)$}\label{sec:linear}

The edge vectors satisfy linear relations at the vertices of ${\mathcal N}$. In Theorem \ref{theo:consist} we prove that this set of linear relations provides a unique system of edge vectors on $({\mathcal N},\mathcal O,\mathfrak l)$ for any chosen set of independent vectors at the boundary sinks. Therefore the components of the edge vectors in Definition \ref{def:edge_vector} have a unique rational representation. In the next Section, we provide their explicit representation
in Theorem \ref{theo:null}.

\begin{lemma}
\label{lem:relations}
The edge vectors $E_e$ on $({\mathcal N},\mathcal O,\mathfrak l)$ satisfy the following linear equation at each vertex:  
\begin{enumerate}
\item  At each bivalent vertex with incoming edge $e$ and outgoing edge $f$:
\begin{equation}\label{eq:lineq_biv}
E_e =  (-1)^{\mbox{int}(e)+\mbox{wind}(e,f)} w_e E_f;
\end{equation}
\item At each trivalent black vertex with incoming edges $e_2$, $e_3$ and outgoing edge $e_1$ we have two relations:
\begin{equation}\label{eq:lineq_black}
E_2 =  (-1)^{\mbox{int}(e_2)+\mbox{wind}(e_2, e_1)}\ w_2 E_1,\quad\quad
E_3 =  (-1)^{\mbox{int}(e_3)+\mbox{wind}(e_3, e_1)}\  w_3 E_1;
\end{equation}
\item At each trivalent white vertex with incoming edge $e_3$ and outgoing edges $e_1$, $e_2$:
\begin{equation}\label{eq:lineq_white}
E_3 =  (-1)^{\mbox{int}(e_3)+\mbox{wind}(e_3, e_1)}\ w_3 E_1 + (-1)^{\mbox{int}(e_3)+\mbox{wind}(e_3, e_2)}\ w_3 E_2,
\end{equation}
\end{enumerate}
where $E_k$ denotes the vector associated to the edge $e_k$.
\end{lemma}
This statement follows directly from the definition of edge vector components as summations over all paths starting from this edge. Let us remark that the last two formulas can be naturally generalized for perfect graphs and vertices of valency greater then 3.



Next we show that, for any given boundary condition at the boundary sink vertices, the linear system in Lemma \ref{lem:relations} defined by equations (\ref{eq:lineq_black}), (\ref{eq:lineq_white}) and (\ref{eq:lineq_biv}) at the internal vertices of $(\mathcal N, \mathcal O, \mathfrak{l})$ possesses a unique solution.

\begin{theorem}\textbf{Full rank of the geometric system of equations for edge vectors on $(\mathcal N, \mathcal O, \mathfrak{l})$.}\label{theo:consist}
Let $(\mathcal N, \mathcal O, \mathfrak{l})$ be a given plabic network with orientation ${\mathcal O}={\mathcal O}(I) $ and gauge ray direction $\mathfrak{l}$.

Given a set $\{B_j\,|\,j\in\bar I\}$ of $n-k$ linearly independent vectors assigned to the boundary sinks $b_j$, let the corresponding edge vectors be defined by: $E_{e_j} = (-1)^{\mbox{int}(e_j)} w_{e_j} B_j$, $j\in\bar I$.  Then  the linear system of equations (\ref{eq:lineq_biv})--(\ref{eq:lineq_white}) at all the internal vertices of $(\mathcal N, \mathcal O, \mathfrak{l})$ has full rank and the number of equations coincides with the number of unknowns, therefore it is consistent and provides a unique system of edge vectors on 
$(\mathcal N, \mathcal O, \mathfrak{l})$. 

Moreover, if we properly order variables and equations, the determinant of the matrix $M$ for this linear system is the sum of the weights of all conservative flows in $(\mathcal N,\mathcal O)$:
\begin{equation}
\label{eq:tal_den}
\det M = \sum\limits_{C\in {\mathcal C}(\mathcal G)}  w(C).
\end{equation}
\end{theorem}

\begin{proof} Let $g+1$ be the number of faces of ${\mathcal N}$ and let $t_W, t_B, d_W$ and $d_B$ respectively be the number of trivalent white, of trivalent black, of bivalent white, and of bivalent black internal vertices of ${\mathcal N}$ as in (\ref{eq:vertex_type}), where $n_I$ is the number of internal edges ({\sl i.e.} edges not connected to a boundary vertex) of ${\mathcal N}$. The total number of equations is
  $$
  n_L=2t_B+ t_W+d_W+d_B =n_I+k,
  $$
  whereas the total number of variables is equal to the total number of edges $n_I+n$. Therefore the number of free boundary conditions is $n-k$ and equals the number of boundary sinks. 

Let us consider the inhomogeneous linear system obtained from equations (\ref{eq:lineq_biv})--(\ref{eq:lineq_white}) in the $
n_L$ unknowns given by the edge vectors not ending at the boundary sinks. Let us denote $M$ the $n_L\times n_L$ representative matrix of such linear system in which we enumerate edges so that each $r$-th row corresponds to the equation in (\ref{eq:lineq_biv}), (\ref{eq:lineq_black}) and (\ref{eq:lineq_white})  in which the edge $e_r$ ending at the given vertex is in the l.h.s.. Then $M$ has unit diagonal by construction. 

If the orientation $\mathcal O$ is acyclic, then it is possible to enumerate the edges of $\mathcal N$ so that their indices in the right hand sight of each equation are bigger than that of the index on the left hand side. Therefore $M$ is upper triangular with unit diagonal, $\det M=1$ and the system of linear relations at the vertices has full rank.

Suppose now that the orientation is not acyclic. The standard formula expresses the determinant of $M$ as:
\begin{equation}
\label{eq:detM}
\det M = \sum\limits_{\sigma\in {S_{n_L}}} \mbox{sign}(\sigma)\prod\limits_{i=1}^{n_L}m_{i,\sigma(i)},
\end{equation}
where $S_{n_L}$ is the permutation group and $\mbox{sign}$ denotes the parity of the permutation $\sigma$. 

Any permutation can be uniquely decomposed as the product of disjoint cycles:

$\sigma=(i_1,i_2,\ldots,i_{u_1+1})(j_1,j_2,\ldots,j_{u_2+1})\ldots(l_1,l_2,\ldots,l_{u_s+1}),$ 
and
$\mbox{sign}(\sigma)=(-1)^{u_1+u_2+\ldots+u_s}.$
On the other side, for $i\ne j$ $m_{i,j}\ne 0$ if and only if the ending vertex of the edge $i$ is the starting vertex of the edge $j$. Therefore $\prod\limits_{i=1}^{n_L} m_{i,\sigma(i)} \ne 0$ if and only if each cycle with $u_k>0$ in $\sigma$ coincides with a simple cycle in the graph, i.e. $\sigma$ encodes a conservative flow in the network. Therefore (\ref{eq:detM}) can be equivalently expressed as:
\begin{equation}
\label{eq:detM1}
\det M =\sum\limits_{C\in {\mathcal C}(\mathcal G)} \mbox{sign}(\sigma(C))\prod\limits_{i=1}^{n_L}m_{i,\sigma(i)}
\end{equation}
where
$\sigma(C)$ denotes the permutation corresponding to the conservative flow $C=C_1\cup C_2\cup\ldots\cup C_s$. Therefore 
\[
\mbox{sign}(\sigma(C))\prod\limits_{i=1}^{n_L}m_{i,\sigma(i)}=\prod\limits_{r=1}^s \left[(-1)^{u_r}\prod\limits_{t=1}^{u_r+1}(-1)^{1+\mbox{wind}(e_{i_t},e_{i_{t+1}})+\mbox{int}(e_{i_t})} w_{i_t}\right]= \prod\limits_{r=1}^s w(C_i)=  w(C),
\]
since the total winding of each simple cycle is $1\,\,(\!\!\!\mod 2)$, the total intersection number for each simple cycle is $0\,\,(\!\!\!\mod 2)$, and $w(C)=w(C_1)\cdots w(C_s)$. 
\end{proof}

\section{Extension of Talaska formula to internal edge vectors}\label{sec:rational}

A deep result of \cite{Pos}, see also \cite{Tal2}, is that each infinite summation in the square bracket of (\ref{eq:sum2}) is a subtraction-free rational expression when $e$ is the edge at a boundary source.
In this Section we adapt Theorem~3.2 in \cite{Tal2} to our purposes. The edge vectors $E_e$ defined in (\ref{eq:sum}) are linear combinations of the edge vectors at the boundary sinks, and the coefficients are rational expressions in the weights with subtraction-free denominator. We express them explicitly as functions of the edge flows and conservative flows. We remark that, contrary to the case in which the initial edge starts at a boundary source, if $e$ is an internal edge, the $j$--th component of $E_e$ may be null even if there exist paths starting at $e$ and ending at $b_j$ (see Section \ref{sec:null_vectors} and Figure \ref{fig:zero-vector}). 

\begin{theorem}\label{theo:null}\textbf{Rational representation for the components of vectors $E_e$}
Let $({\mathcal N},\mathcal O, \mathfrak l)$ be a plabic network representing a point $[A]\in \S \subset \GTNN$ with orientation $\mathcal O$ associated to the base $I =\{ 1\le i_1< i_2 < \cdots < i_k\le n\}$ in the matroid $\mathcal M$ and gauge ray direction $\mathfrak{l}$. Let us assign the vectors $B_j$ to the boundary sinks $b_j$, $j\in \bar I$.
Then edge vector $E_{e}$ at the edge $e$ defined in (\ref{eq:sum2}), is a rational expression in the weights on the network with subtraction-free denominator: 
\begin{equation}
\label{eq:tal_formula}
E_{e} =\mathlarger{\mathlarger{\sum\limits}_{j\in\bar I}}\ \ \left[ \frac{\displaystyle\sum\limits_{F\in {\mathcal F}_{e,b_j}(\mathcal G)} \big(-1\big)^{\mbox{wind}(F)+\mbox{int}(F)}\ w(F)}{\sum\limits_{C\in {\mathcal C}(\mathcal G)} \ w(C)}\right]\  B_j,
\end{equation}
where notations are as in Definitions~\ref{def:edge_flow},~\ref{def:cons_flow} and~\ref{def:numb_flows}.  
\end{theorem}
\begin{proof}
The proof is a straightforward adaptation of the proof in \cite{Tal2} for the computation of the Pl\"ucker coordinates. 
If the graph is acyclic, the proof of (\ref{eq:tal_formula}) is elementary since the denominator is one and the edge flows $F_{e,b_j}$ are in one-to-one correspondence with directed paths connecting $e$ to $b_j$. Therefore  (\ref{eq:tal_formula}) and (\ref{eq:sum}) coincide when $B_j$ is the $j$-th vector of the canonical basis, $j\in\bar I$.

Otherwise, in view of  (\ref{eq:sum2}), we have to prove the following identity:
\begin{equation}
\label{eq:tal_formula_2}
 \sum\limits_{{\mathcal P}:e\rightarrow b_{j}}  \sum\limits_{C\in {\mathcal C}(\mathcal G)}  (-1)^{\mbox{wind}({\mathcal P})+ \mbox{int}({\mathcal P})}  w({\mathcal P})  w(C)  = \sum\limits_{F\in {\mathcal F}_{e,b_j}(\mathcal G)} \big(-1\big)^{\mbox{wind}(F)+\mbox{int}(F)}\ w(F),
\end{equation}
where in the left-hand side the first sum is over all directed paths from $e$ to $b_j$. In the left-hand side we have two types of terms:
\begin{enumerate}
\item $\mathcal P$ is an edge loop-erased walk and $C$ is a conservative flow with no common edges with $\mathcal P$. By (\ref{eq:wind_flow}), the summation over this group coincides with the right-hand side of (\ref{eq:tal_formula_2});
\item $\mathcal P$ is not edge loop-erased or it is loop-erased, but has a common edge with  $C$.
\end{enumerate}
Following \cite{Tal2}, we prove that the summation over the second group gives zero by introducing a sign-reversing involution $\varphi$ on the set of pairs $(C,P)$. We first assign two numbers to each pair $(C,P)$ as follows:
\begin{enumerate}
\item Let $P=(e_1,\ldots,e_m)$. If $P$ is edge loop-erased, set $\bar s=+\infty$; otherwise, let $L_1=(e_l,e_{l+1},\ldots,e_{s})$ be the first loop erased according to Definition~\ref{def:loop-erased-walk} and set $\bar s=s$;
\item If $C$ does not intersect $P$, set $\bar t=+\infty$. Otherwise, set $\bar t$ the smallest $t$ such that $e_t\in P$ and $e_t\in C$. Denote the component of $C$ containing $e_{\bar t}$ by $L_2=(l_1,\ldots,l_p)$ with $l_1=e_{\bar t}$. 
\end{enumerate}
A pair $(C,P)$ belongs to the second group if and only if at least one of the numbers $\bar s$, $\bar t$ is finite. Moreover, in this case, $\bar s\ne \bar t$, because if  $e_{\bar s}=e_{\bar t}$, then $\bar t < \bar s$ by the labeling rules. We then define $(C^*,P^*)=\varphi(C,P)$ as follows:
\begin{enumerate}
\item  If $\bar s < \bar t$, then $P$ completes its first cycle $L_1$ before intersecting any cycle in $C$. In this case $L_1\cap C=\emptyset$, and we remove $L_1$ from $P$ and add it to $C$. Then $P^*=(e_1,\ldots,e_{l-1},e_s,\ldots,e_m)$ and $C^*=C\cup L_1$;
\item If $\bar t < \bar s$, then $P$ intersects  $L_2$ before completing its first cycle. Then we remove $L_2$ from $C$ and add it to $P$: $C^*=C\backslash L_2$, $P^*=(e_1,\ldots,e_{\bar t-1},e_{l_1}=e_{\bar t},e_{l_2},\ldots,e_{l_p},e_{\bar t+1},\ldots,e_m)$. 
\end{enumerate}
From the construction of $\varphi$ it follows immediately that $(C^*,P^*)$ belongs to the second group, $\varphi^2=\mbox{id}$, and $\varphi$ is sign-reversing since 
$w(C^*) w(P^*) = w(C) w(P)$, $\mbox{wind}(P) +\mbox{wind}(P^*) =1 \,\,
(\!\!\!\!\mod 2)$ and $\mbox{int}(P) +\mbox{int}(P^*) =0 \,\,
(\!\!\!\!\mod 2)$. 
\end{proof}

\begin{corollary}
\label{cor:bound_source}\textbf{The connection between the edge vectors at the boundary sources and Talaska formula for the boundary measurement matrix.} Under the hypotheses of Theorem \ref{theo:null}, let $e$ be the edge starting at the boundary source $b_{i_r}$. Then the number $\mbox{wind}(F)+\mbox{int}(F)$ has the same parity for all edge flows $F$ from $b_{i_r}$ to $b_j$ and it is equal to the number $N_{rj}$ of boundary sources between $i_r$ and $j$ in the orientation $\mathcal O$,
\begin{equation}\label{eq:index_source}
N_{rj} = \# \left\{ i_s \in I\ , \ i_s \in \big] \min \{i_r, j\}, \max \{ i_r , j \} \big[ \ \right\}.
\end{equation}
Therefore, for such edges and the choice $B_j=E_j$, where $E_k$, $k\in[n]$ are the canonical basic vectors in $\R^n$, (\ref{eq:tal_formula}) simplifies to
\begin{equation}
\label{eq:tal_formula_source}
\left(E_{e}\right)_{j}= \big(-1\big)^{N_{rj}}\ \frac{\sum_{F\in {\mathcal F}_{e,b_j}(\mathcal G)} \ w(F)}{\sum_{C\in {\mathcal C}(\mathcal G)} \ w(C)} =A^r_{j},
\end{equation}
where $A^r_j$ is the entry of the reduced row echelon matrix $A$ with respect to the base $I=\{1\le i_1 < i_2 < \cdots < i_k \le n\}$.
\end{corollary}

\begin{proof}
First of all, in this case, each edge flow $F$ from $i_r$ to $j$ is either an acyclic edge loop--erased walk $\mathcal P$ or the union of $\mathcal P$ with a conservative flow $C$ with no common edges with $\mathcal P$. Therefore to prove that the number $\mbox{wind}(F)+\mbox{int}(F)$ has the same parity for all $F$ is equivalent to prove that $\mbox{wind}(\mathcal P)+\mbox{int}(\mathcal P)$ has the same parity for all edge loop--erased walks from $b_{i_r}$ to $b_j$ (again notations are as in Definitions~\ref{def:edge_flow},~\ref{def:cons_flow} and~\ref{def:numb_flows}).   
Any two such loop erased walks, ${\mathcal P}$ and ${\tilde {\mathcal P}}$, share at least the initial and the final edges and are both acyclic.
\begin{figure}
  \centering
	{\includegraphics[width=0.3\textwidth]{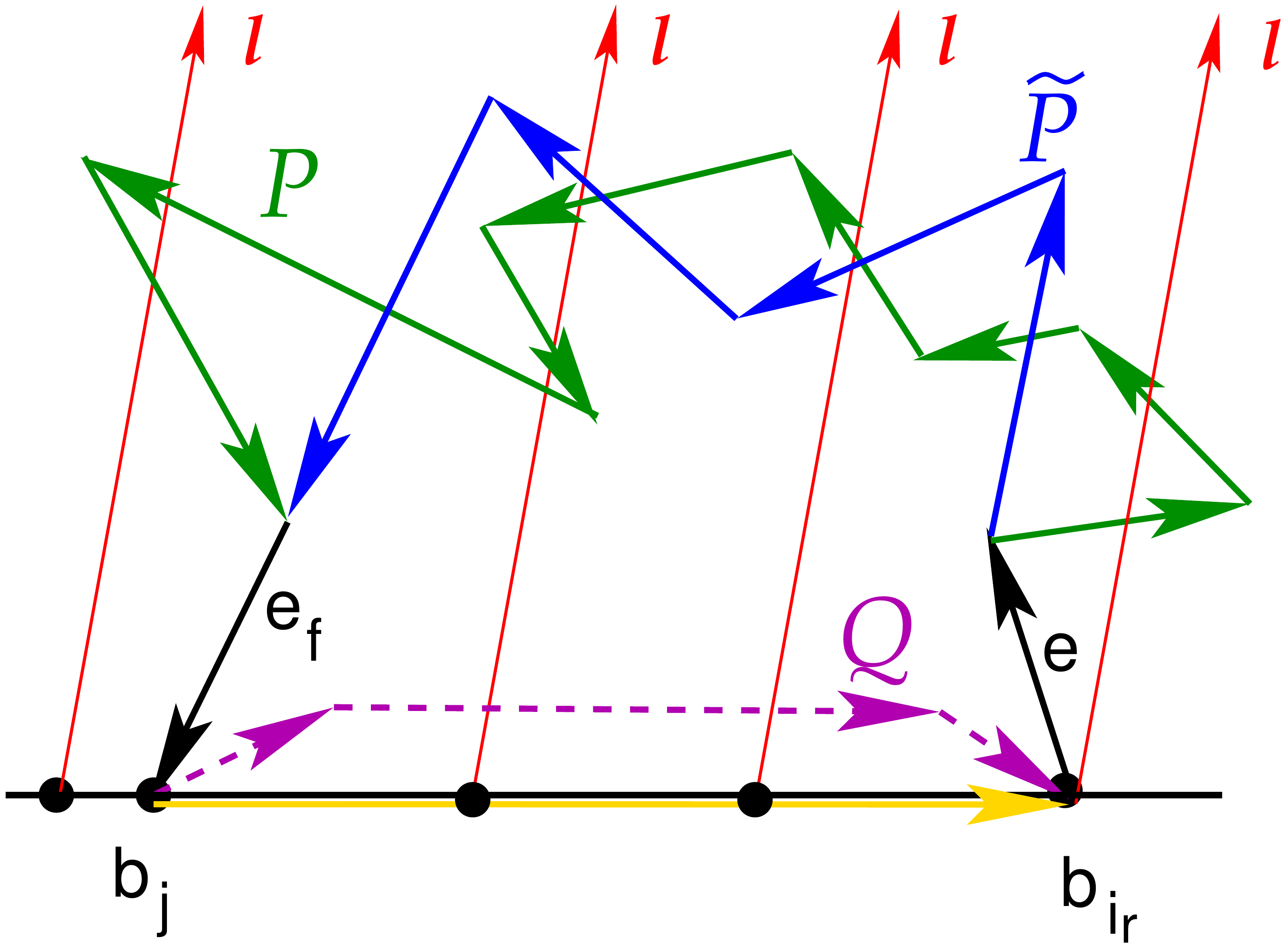} \hspace{1cm} \includegraphics[width=0.3\textwidth]{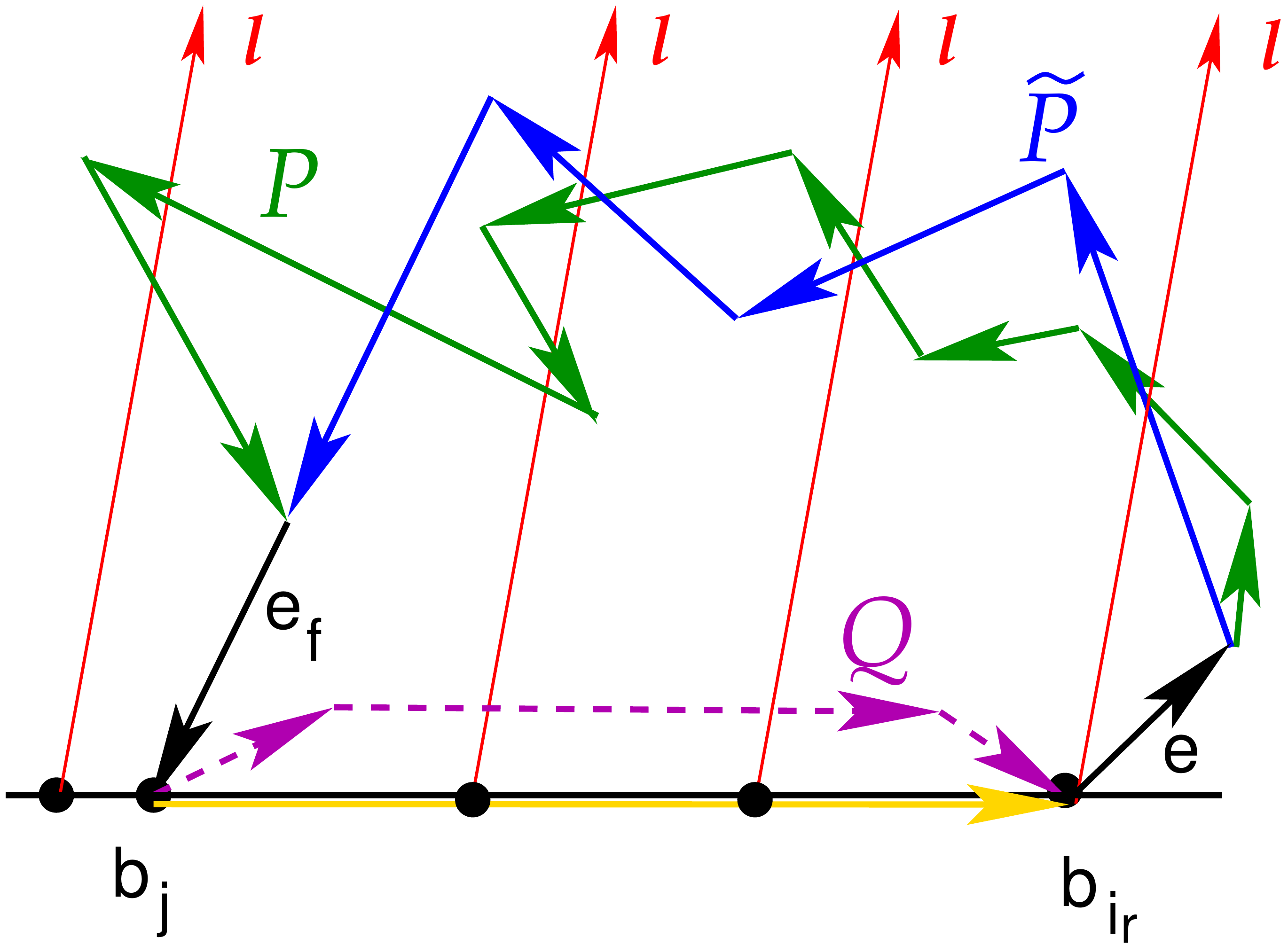}}
  \caption{\small{\sl We illustrate the proof of Corollary~\ref{cor:bound_source}. The path $\mathcal P$ is the union of black and green edges, the path $\tilde{\mathcal P}$ is the union of black and blue edges. Additional path $Q$ is drawn magenta, and the edge  $e_{j,i_r}$  is drawn gold. On the left (respectively right) the gauge ray starting at  $b_{i_r}$ lies outside (respectively inside) the angle $\widehat{-e_{j,i_r},e}$.}\label{fig:winding_intersection}}
\end{figure}
If we add an edge $e_{j,i_r}$ from $b_j$ to $b_{i_r}$ (see Fig.~\ref{fig:winding_intersection}), we obtain a pair of simple cycles with the same orientation and of total winding equal to $1$ modulo $2$. Therefore 
\begin{equation}
\label{eq:2wind}  
\mbox{wind} (P) = \mbox{wind} (\tilde P)  = 1 - \mbox{wind}(e_{f},e_{j,i_r}) - \mbox{wind}(e_{j,i_r},e) \quad (\!\!\!\!\!\!\mod 2). 
\end{equation}
Moreover, $\mbox{wind}(e_{f},e_{j,i_r})=0$ and  
$$
\mbox{wind}(e_{j,i_r},e) =\left\{\begin{array}{ll} 1 & \mbox{if the gauge ray}\ \ {\mathfrak l}_{i_r} \ \ \mbox{lies outside the angle} \ \ \widehat{-e_{j,i_r},e}   \\
     0 & \mbox{if the gauge ray}\ \ {\mathfrak l}_{i_r} \ \ \mbox{lies inside the angle} \ \ \widehat{-e_{j,i_r},e}                                 
    \end{array}\right.
$$
Therefore, in the first case $\mbox{wind} (P) = \mbox{wind} (\tilde P)  = 0 \quad (\!\!\!\!\mod 2)$ and in the second case
$\mbox{wind} (P) = \mbox{wind} (\tilde P)  = 1 \quad (\!\!\!\!\mod 2)$.

Next, let us add a directed path $Q$ from $b_j$ to  $b_{i_r}$ very close to the boundary to the graph (see Fig~\ref{fig:winding_intersection}). Then the total intersection number of the simple cycles $Q\cup P$, $Q\cup \tilde P$ are both zero $(\!\!\!\!\mod 2)$ and we easily conclude that $\mbox{int} ({\mathcal P}) =\mbox{int} ({\tilde {\mathcal P}}) \, (\!\!\!\!\mod 2)$.

Without loss of generality, we may assume that $i_r<j$. Since $\mathcal P$ is acyclic, all pivot rays ${\mathfrak l}_{i_l}$, $i_l\in [i_r -1] \cup [j, n]$ intersect $\mathcal P$ an even number of times, whereas all pivot rays ${\mathfrak l}_{i_l}$, $i_l\in [i_r +1, j]$ intersect $\mathcal P$ an odd number of times. Finally, the ray ${\mathfrak l}_{i_r}$ intersects $\mathcal P$  an even (odd) number of times if the gauge ray ${\mathfrak l}_{i_r}$ lies outside (inside) the angle $\widehat{-e_{j,i_r},e}$. Therefore,  $\mbox{wind}(F)+\mbox{int}(F)$ is equal to the number $N_{rj}$ of sources in the interval $]i_r,j[$  (\ref{eq:index_source}), and the sum in the right hand side in (\ref{eq:tal_formula_source}) coincides  with the formula in \cite{Tal2}.
\end{proof}

\begin{figure}
  \centering
	{\includegraphics[width=0.4\textwidth]{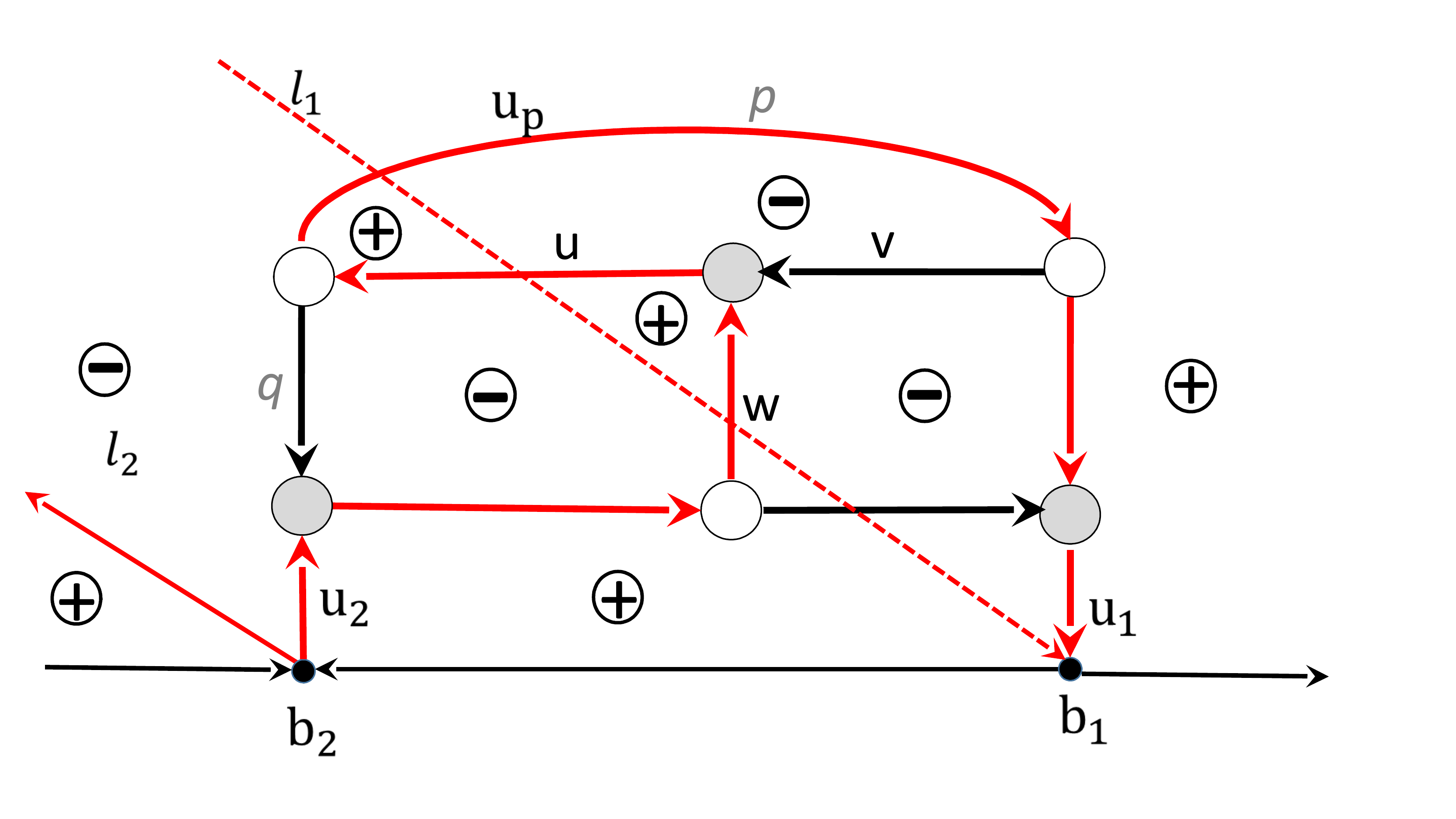}}
	\hspace{1. truecm}
	{\includegraphics[width=0.4\textwidth]{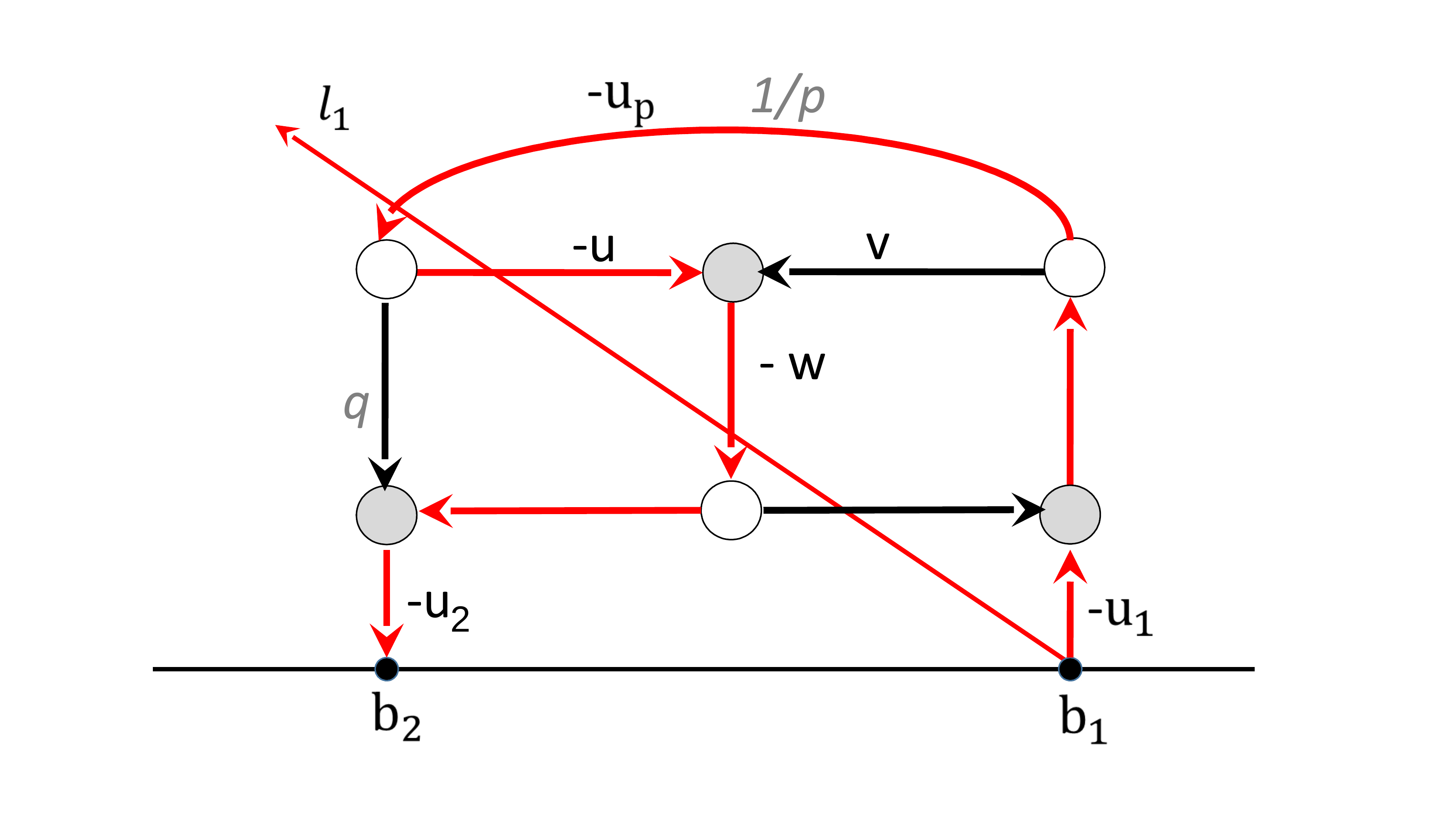}}
  \caption{\small{\sl The computation of edge vectors using Theorem \ref{theo:null} and Lemma \ref{lemma:path}. The path along which we change orientation is colored red in both Figures; we mark regions to compute the indices in (\ref{eq:eps_not_path}) and  (\ref{eq:eps_on_path})  [left].}\label{fig:talaska_orientation}}
\end{figure}

\begin{example}
For the orientation and gauge ray direction as in Figure \ref{fig:Rules0}, the vectors $E_e$ on the Le--network coincide with those introduced
in the direct algebraic construction in \cite{AG3}.
\end{example}

\begin{example}\label{example:null}
We illustrate both the Theorem and the Corollary on the example in Figure \ref{fig:talaska_orientation} [left]. The network represents the point $[ 2p/(1+p+q),1] \in Gr^{\mbox{\tiny TP}}(1,2)$: all weights are equal to 1 except for the two edges carrying the positive weights $p$ and $q$. In the given orientation, the networks possesses two conservative flows of weight $p$ and $q$. Therefore $\sum\limits_{C\in {\mathcal C}(\mathcal G)} \ w(C)=1+p+q$. There are two possible loop erased edge walks starting at $u$, which coincide with the edge flows from $u$, so that $\sum\limits_{F\in {\mathcal F}_{u,b_1}(\mathcal G)} \big(-1\big)^{\mbox{wind}(F)+\mbox{int}(F)}\ w(F) = q-p$.
Therefore on the edges $u,v,w$, using (\ref{eq:tal_formula}) we get
$$
E_u=E_v=-E_w=\left( \frac{q-p}{1+p+q}, 0\right).
$$
We remark that $E_u=E_v=E_w=(0,0)$, when $p=q$, that is null edge vectors are possible even if there exist paths starting at the edge and ending at some boundary sink. We shall return on the problem of null edge vectors in Section \ref{sec:null_vectors}.
It is easy to check that all other edge vectors associated to such network have non--zero first component for any choice of $p,q>0$. In particular the edge vector at the boundary source is equal to
$E_{u_2} =(\frac{1+2p}{1+p+q},0)$, since there are two loop erased walks starting from $u_2$ and three edge flows so that $\sum\limits_{F\in {\mathcal F}_{u_2,b_1}(\mathcal G)} \big(-1\big)^{\mbox{wind}(F)+\mbox{int}(F)}\ w(F) = 1(1+p)+p$.
The edge vector $E_{u_p} = (\frac{p+2pq}{1+p+q},0)$ since there are two loop erased walks and three edge flows starting at $u_p$.
Similarly $E_{u_1} =(1,0)$ since there is only one loop erased walk and three edge flows from $u_1$.
Finally the representative matrix associated to this system of vectors is $A[1] = (\frac{1+2p}{1+p+q}, 1)$.
\end{example}

\section{Dependence of edge vectors on orientation and network gauge freedoms}\label{sec:vector_changes}

In this section we discuss the dependence of edge vectors on the various gauge freedoms of the network.

\subsection{Dependence of edge vectors on gauge ray direction $\mathfrak l$}\label{sec:gauge_ray}
We show that the effect of a change of direction in the gauge ray $\mathfrak l$ on the vectors $E_e$ is the following: the new vectors  $E^{\prime}_e$ coincide with the old ones  $E_e$ up to a sign, and the boundary measurement matrix is preserved.

\begin{figure}
  \centering
	{\includegraphics[width=0.5\textwidth]{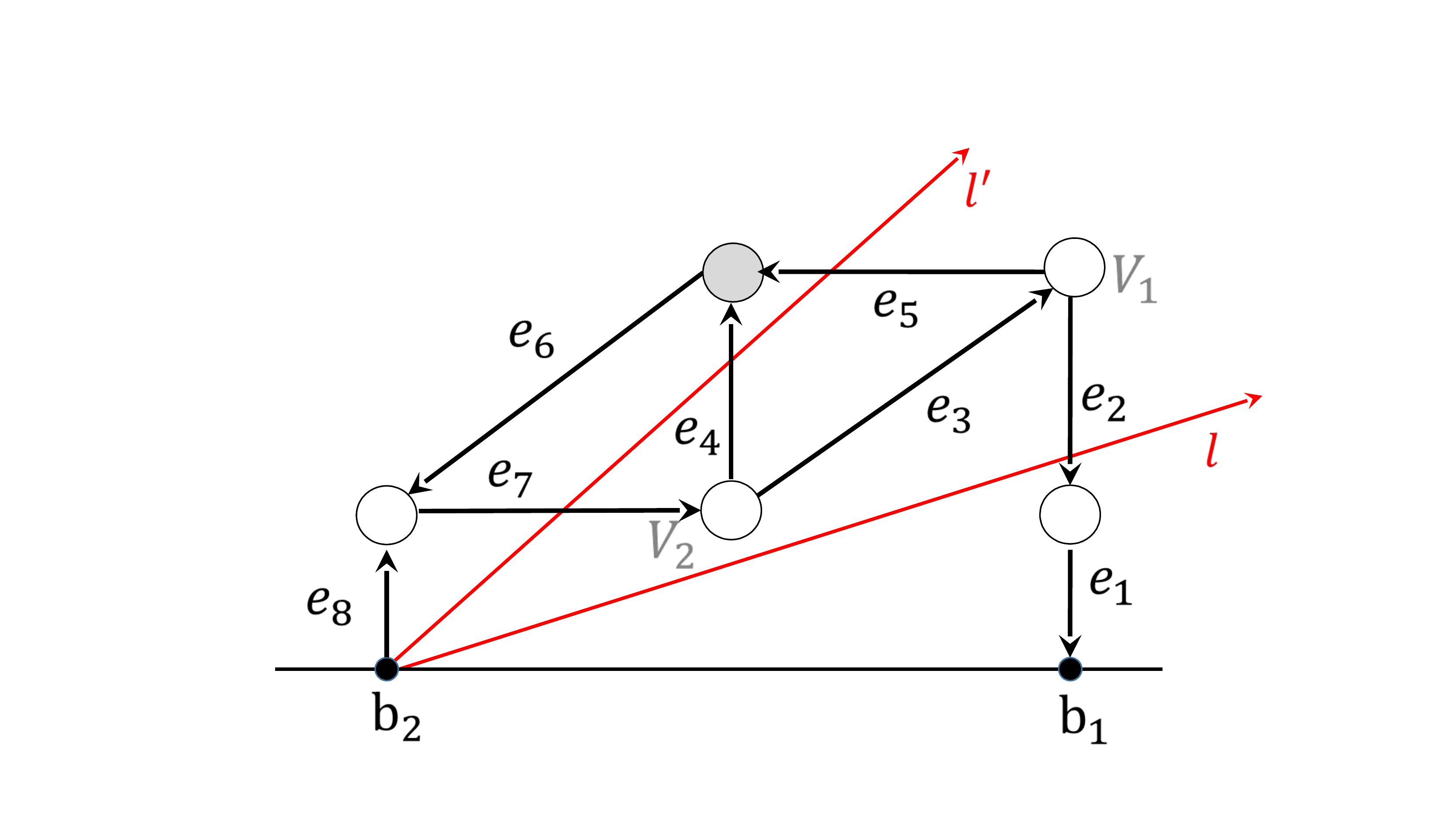}}
	\vspace{-.7 truecm}
  \caption{\small{\sl We illustrate Proposition \ref{prop:rays}.}}\label{fig:pivot}
\end{figure}

\begin{proposition}\label{prop:rays}\textbf{The dependence of the system of vectors on the ray direction $\mathfrak l$}
Let $({\mathcal N},\mathcal O)$ be an oriented network and consider two gauge directions $\mathfrak l$ and 
$\mathfrak{l}^{\prime}$.
\begin{enumerate}
\item For any boundary source edge $e_{i_r}$ the vector $E_{e_{i_r}}$ does not depend on the gauge direction $\mathfrak l$ 
and it coincides with the $r$-th row of the generalized RREF of $[A]$, associated to the pivot set $I$,  minus  the $i_r$--th vector of the canonical basis, which we denote $E_{i_r}$,
\begin{equation}
\label{eq:Ei}
E_{e_{i_r}} = A[r] -E_{i_r}.
\end{equation}
\item For any other edge $e$ we have
\begin{equation}\label{eq:vector_ray}
E^{\prime}_e = (-1)^{\mbox{cr}(V_e)+\mbox{par(e)}} E_e,
\end{equation} 
where $E_e$ and  $E^{\prime}_e$ respectively are the edge vectors for $e$ for the gauge direction ${\mathfrak l}$
and ${\mathfrak l}^{\prime}$, $\mbox{par(e)}$ is 1 if $e$ belongs to the angle $\widehat{{\mathfrak l},{\mathfrak l}^{\prime}}$ , and 0 otherwise,
whereas $\mbox{cr}(V_e)$ denotes the number of gauge rays passing the initial vertex $V_e$ of $e$ during the rotation from  ${\mathfrak l}$ to ${\mathfrak l}^{\prime}$ inside the disk.
\end{enumerate}
\end{proposition}

\begin{proof}
Formula (\ref{eq:Ei}) follows from Corollary~\ref{cor:bound_source}: indeed (\ref{eq:index_source}) implies that the components of $E_{e_{\bar i}}$ are invariant with 
respect to changes of the gauge direction. Finally, since there is no path to the boundary source $b_{\bar i}$, 
the corresponding component of the edge vector is zero. 

To prove the second statement, we show that, for a given initial edge $e$, the sign contribution of each edge loop--erased walk $\mathcal P$ starting at $e$ is either the same before and after the gauge ray rotation or changes in the same way for every walk independently of the destination $b_j$.

Indeed let us consider a monotone continuous change of the gauge direction from initial $\mathfrak l = \mathfrak l(0)$ to final ${\mathfrak l}'=\mathfrak l(1)$. If for some $t\in(0,1)$ the vector $\mathfrak l(t)$ forms a zero angle with an edge of $\mathcal P$ distinct from the initial one, the parity of the  winding number of $\mathcal P$ remains unchanged. On the contrary, if $\mathfrak l(t)$ forms a zero angle with the initial edge $e$, the winding number of $\mathcal P$ changes its parity, and, in such case we settle $\mbox{par(e)}=1$. We remark that $\mathfrak l(t)$ can never form a zero angle with the edge at the boundary sink in $\mathcal P$. 

Similarly, if one of the gauge lines passes through a vertex in $\mathcal P$ distinct from the initial vertex, then the parity of the intersection number of $\mathcal P$ remains unchanged. It changes  $1\  (\!\!\!\!\mod 2)$ only if one of the gauge rays passes through the initial vertex of $\mathcal P$ (again it can never pass through the final vertex).    

Since the first edge $e$ and its initial vertex are common to all paths starting at $e$, all components of the 
vector $E_e$ either remain invariant, or are simultaneously multiplied by $-1$.
\end{proof}

\begin{example}
We illustrate Proposition \ref{prop:rays} in Figure \ref{fig:pivot}. In the rotation from $\mathfrak{l}$ to $\mathfrak{l}^{\prime}$ inside the disk, the gauge ray starting at $b_2$ passes the vertices $V_1$ and $V_2$ and the direction $e_3$. Therefore
$E^{\prime}_{e_i} = - E_{e_i}$, for $i=2,4,5$, whereas $E^{\prime}_{e_i} = E_{e_i}$ for all other edges.
\end{example}

\subsection{Dependence of edge vectors on orientation of the graph}\label{sec:orient}
We now explain how the system of vectors changes when we change the orientation of the graph. Following \cite{Pos}, a change of orientation can be represented as a finite composition of elementary changes of orientation, each one consisting in
a change of orientation either along a simple cycle ${\mathcal Q}_0$ or
along a non-self-intersecting oriented path ${\mathcal P}$ from a boundary source $i_0$ to a boundary sink $j_0$.
Here we use the standard rule that we do not change the edge weight if the edge does not change orientation, otherwise we replace the original weight by its reciprocal. 

\begin{theorem}\label{theo:orient}\textbf{The dependence of the system of vectors on the orientation of the network.}
Let ${\mathcal N}$ be a plabic network representing a given point $[A]\in \S \subset\GTNN$ and $\mathfrak l$ be a gauge ray direction. Let $\mathcal O$, ${\hat {\mathcal O}}$ be
two perfect orientations of ${\mathcal N}$ for the bases $I,I^{\prime}\in {\mathcal M}$. Let $A[r]$, $r\in [k]$, denote 
the $r$-th row of a chosen representative matrix of $[A]$.
Let $E_e$ be the system of 
vectors associated to $({\mathcal N},\mathcal O, \mathfrak l)$ and satisfying the boundary conditions $E[j]$ at $b_j$, $j \in \bar I$, whereas $\hat E_e$ are those associated to $({\mathcal N},{\hat {\mathcal O}}, \mathfrak l)$ and satisfying the boundary conditions $E[l]$ at $b_l$, $l \in \bar I^{\prime}$.  
Then for any $e\in {\mathcal N}$, there exist real constants $\alpha_e\ne 0$, $c^r_e$, $r\in [k]$ such that
\begin{equation}\label{eq:orient}
\hat E_e = \alpha_e E_e + \sum_{r=1}^k c^r_e A[r],
\end{equation}
where for elementary transformations $\alpha_e$ are as in (\ref{eq:hat_E_P}), (\ref{eq:hat_E_Q}).
\end{theorem}

\begin{proof}
    It is sufficient to prove this statement in the case of elementary changes of orientation (see Lemmas~\ref{lemma:path} and \ref{lemma:cycle} below). Indeed, a generic change of orientation is represented by the composition of a finite set of such elementary transformations, and the resulting system of vectors does not depend on the sequence of transformations.
\end{proof}

Both in the case of an elementary change of orientation along a non-self-intersecting directed path $\mathcal P_0$ from a boundary source to a boundary sink or along a simple cycle ${\mathcal Q}_0$, we provide the explicit relation between the edge vectors in the two orientations.

Given an elementary change of orientation, we assign an index ${\gamma}(e)$ to each edge of the network in its \textbf{initial orientation}.

First, we mark all regions of the disk by either $+$ or $-$ using the following rule.
\begin{enumerate}
\item If $\mathcal P_0$ is a non-self-intersecting oriented path from a boundary source $i_0$ to a boundary sink $j_0$ in the initial orientation of ${\mathcal N}$, we divide the interior of the disk into a finite number of regions bounded by the gauge ray ${\mathfrak l}_{i_0}$ oriented 
upwards, the gauge ray ${\mathfrak l}_{j_0}$ oriented downwards, the path $\mathcal P_0$ oriented as in $({\mathcal N},\mathcal O,\mathfrak l)$ and the boundary of the disk divided into two arcs, each oriented from $j_0$ to $i_0$. Then we mark a region with a $+$ if its boundary is
oriented, otherwise mark it with $-$ (see Figure~\ref{fig:inv_symb}).
\item If $\mathcal Q_0$ is a closed oriented simple path, it  divides the interior of the disk into two regions: we mark the region external to $\mathcal Q_0$ with a $+$ and the internal region with $-$. 
\end{enumerate}

\begin{figure}
  \centering
	{\includegraphics[width=0.4\textwidth]{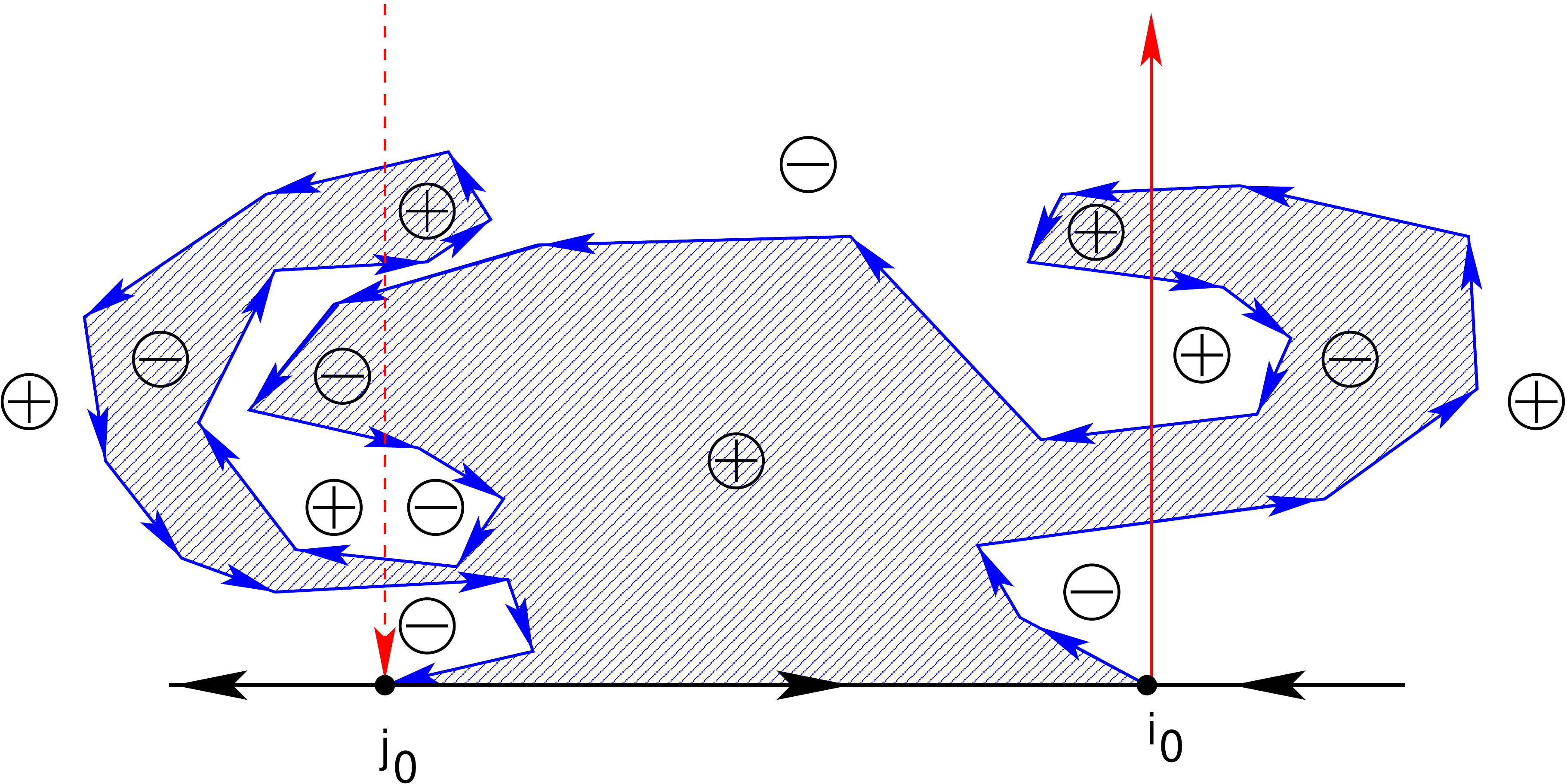}}
	\hfill
	{\includegraphics[width=0.4\textwidth]{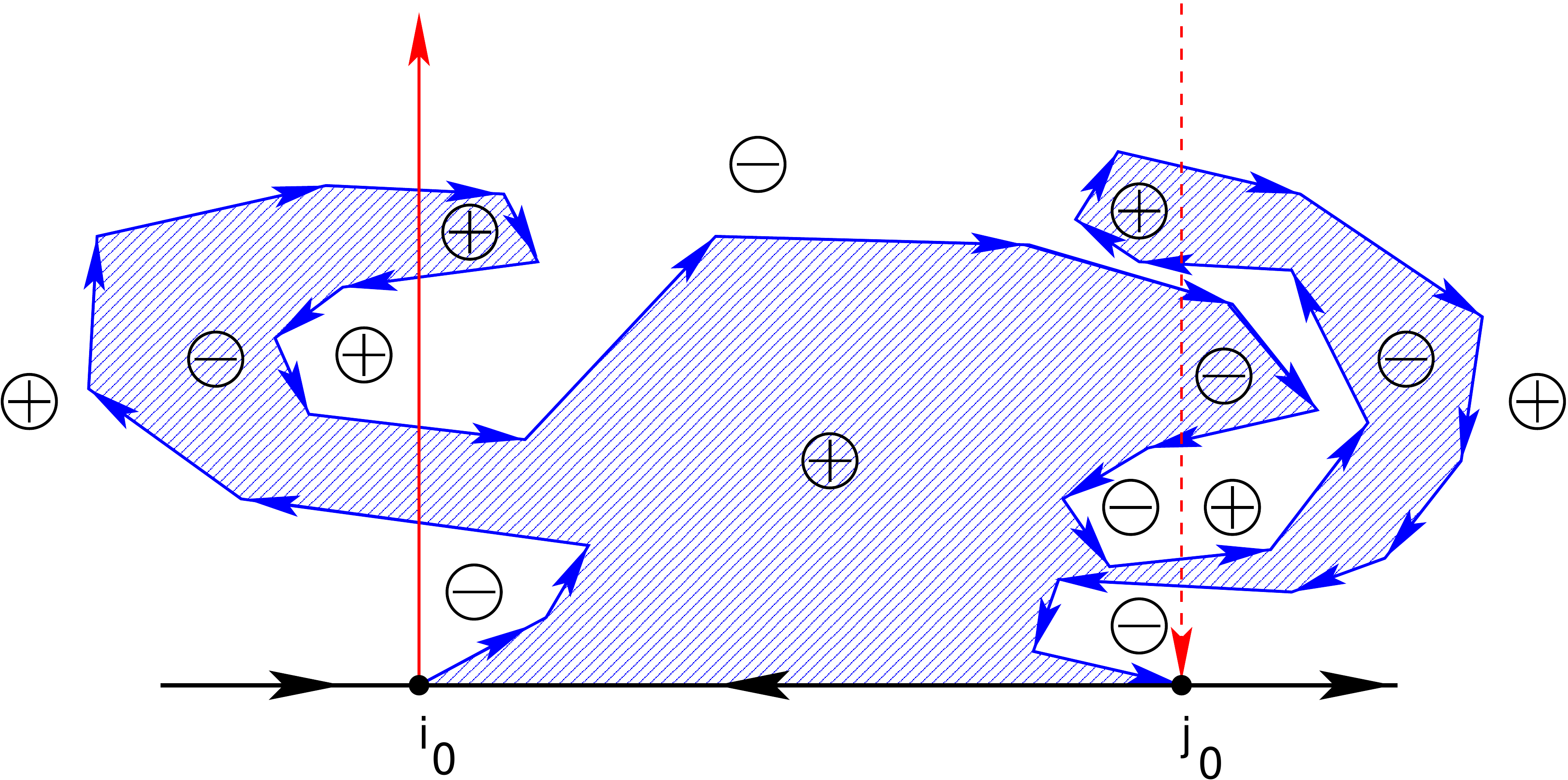}}
  \caption{\small{\sl We illustrate the marking of the regions.}\label{fig:inv_symb}}
\end{figure}

Let us remark that this marking remains invariant after the change of orientation.

If the edge $e\not \in \mathcal P_0$ (respectively $e\not \in \mathcal Q_0$), then we assign it an index $\gamma(e)$ as follows
\begin{equation}\label{eq:eps_not_path}
\gamma (e) =
\left\{ 
\begin{array}{ll} 
0 & \mbox{ if the starting vertex of } e \mbox{ belongs to a } + \mbox{ region, }\\
1 & \mbox{ if the starting vertex of } e \mbox{ belongs to a } - \mbox{ region, }\\
\end{array}
\right.
\end{equation}
where, in case the initial vertex of $e$ belongs to $\mathcal P_0$ or $\mathcal Q_0$, we make an infinitesimal shift of the starting vertex in the direction of $e$ before assigning the edge to a region. 

If the edge $e\in \mathcal P_0$ (respectively $e\in \mathcal Q_0$), we assign it the index 
\begin{equation}\label{eq:eps_on_path}
\gamma (e) = \gamma_1 (e) + \gamma_2 (e) +\mbox{int}(e),
\end{equation}
using the initial orientation $\mathcal O$ as follows:
\begin{enumerate}
\item We look at the region to the left and near the ending point of $e$, and assign index 
\[
\gamma_1 (e) =
\left\{ 
\begin{array}{ll} 
0 & \mbox{ if the region is marked with } + ,\\
1 & \mbox{ if the region is marked with } - ;\\
\end{array}
\right.
\]
\item We consider the ordered pair $(e, \mathfrak{l})$ and assign index 
\[
\gamma_2 (e) = \frac{1-s(e,\mathfrak{l})}{2}
\]
with $s(\cdot,\cdot)$ as in (\ref{eq:def_s}).
\end{enumerate}

It is easy to check that $\gamma(e)$ does not change after the change of orientation.

\medskip

If the change of orientation is ruled by $\mathcal P_0$, we use a two-steps proof:
\begin{enumerate}
\item We conveniently change the boundary conditions at the boundary sinks in the initial orientation of the
network $({\mathcal N},\mathcal O,\mathfrak l)$, we compute the system of vectors ${\tilde E}_e$ satisfying these new boundary conditions and we give explicit relations between the two systems of vectors $E_e$ and  ${\tilde E}_e$ on $({\mathcal N},\mathcal O,\mathfrak l)$;
\item Then, we show that the system of vectors ${\hat E}_e$ in (\ref{eq:hat_E_P}), defined on $({\mathcal N},{\hat {\mathcal O}},\mathfrak l)$ coincides with the system ${\tilde E}_e$ up to non-zero multiplicative factors.
\end{enumerate}

\begin{lemma}\label{lemma:path}\textbf{The effect of a change of orientation along a non self--intersecting path from a boundary source to a boundary sink.}
Let $I = \{ 1\le i_1 < i_2 < \cdots < i_k\le n\}$ and $\bar I = \{ 1\le j_1 < j_2 < \cdots < j_{n-k}\le n\}$ respectively be the pivot and non--pivot indices in the representative RREF matrix $A$ associated to $({\mathcal N},\mathcal O,\mathfrak l)$. Assume that all the edges at the boundary vertices have unit weight and that no gauge ray intersects such edges in the initial orientation. 
Assume that we change the orientation  along a non-self-intersecting oriented path $\mathcal P_0$ from a boundary source $i_0$ to a 
boundary sink $j_0$. Let $E_e$ and $\tilde E_e$ be the systems of vectors on $({\mathcal N},\mathcal O,\mathfrak l)$ corresponding to the following choices of boundary conditions at edges $e_j$ ending at the boundary sinks $b_j$, $j\in \bar I$:
\begin{equation}
\label{eq:orient1}
E_{e_{j}}=E_{j}, \quad\quad\quad\quad
\tilde E_{e_{j}}= \left\{ \begin{array}{ll} E_{j} & \mbox{ if } j\not = j_0;\\
E_{j_0}-\frac{1}{A^{r_0}_{j_0}} A[r_0], &\mbox{ if } j = j_0,
\end{array}
\right.
\end{equation}
where $E_{j}$ is the $j$--th vector of the canonical basis, whereas $A[r_0]$ is the row of the matrix $A$ associated to the source $i_0$. Then the following system of vectors $\hat E_e$,  $e\in {\mathcal N}$, 
\begin{equation}\label{eq:hat_E_P}
{\hat E}_e = \left\{ \begin{array}{ll}
 (-1)^{\gamma(e)} {\tilde E}_e, & \mbox{ if } e\not \in \mathcal P_0, \mbox{ with } \gamma(e) \mbox{ as in (\ref{eq:eps_not_path})},\\
\displaystyle \frac{(-1)^{\gamma(e)}}{w_e} {\tilde E}_e, & \mbox{ if } e\in \mathcal P_0, \mbox{ with } \gamma(e) \mbox{ as in (\ref{eq:eps_on_path})},
\end{array}\right.
\end{equation}
is the system of vectors on the network $({\mathcal N},{\hat {\mathcal O}},\mathfrak l)$ satisfying the boundary conditions
\begin{equation}
\label{eq:orient2}
\hat E_{e_{j}}= \left\{ \begin{array}{ll} (-1)^{\mbox{int} (e_j)^{\prime}} E_{j} & \mbox{ if } j\in \bar I \backslash \{ j_0\}\\
E_{i_0}, &\mbox{ if } j = i_0,
\end{array}
\right.
\end{equation}
where $\mbox{int}(e)^{\prime}$ is the number of intersections of the gauge ray $\mathfrak{l}_{j_0}$  with $e$.
\end{lemma}

\begin{remark}
To simplify the proof of Lemma \ref{lemma:path}, we assume without loss of generality that the edges at boundary vertices have unit weight and that no gauge ray intersects them in the initial orientation. This hypothesis may be always fulfilled modifying the initial network using the weight gauge freedom (Remark~\ref{rem:gauge_weight}) and adding, if necessary, bivalent vertices next to the boundary vertices using move (M3). 
In Sections~\ref{sec:different_gauge} and \ref{sec:moves_reduc}, we show that the effect of these transformations amounts to a well--defined non zero multiplicative constant for the edge vectors. Therefore, the statement in Lemma \ref{lemma:path} holds in the general case with obvious minor modifications in the boundary conditions for the three systems of vectors $E_e$, $\tilde E_e$ and $\hat E_e$.
\end{remark}

\begin{proof}
The system of vectors  $\tilde E_e - E_e$ is the solution to the system of linear relations on 
$({\mathcal N},\mathcal O,\mathfrak l)$ for the following boundary conditions: 
\[
\tilde E_{e_{j}}-E_{e_{j}} = \left\{ \begin{array}{ll} 0 & \mbox{ if } j\not = j_0;\\
-\frac{1}{A^{r_0}_{j_0}} A[r_0], &\mbox{ if } j = j_0.
\end{array}
\right.
\]
Then at all edges $e\in {\mathcal N}$ the difference $\tilde E_{e}-E_{e}$ is proportional to $A[r_0]$. In particular $\tilde E_{e_{i_0}} = -E_{i_0}$, since, by construction, $E_{e_{i_0}} = A[r_0] - E_{i_0}$. Therefore, each vector $\hat E_e$ in (\ref{eq:hat_E_P}) is a linear combination of the vector $E_e$ and $A[r_0]$.

Next  we check that the system of edge vectors ${\hat E}_e$ defined by  (\ref{eq:hat_E_P})  satisfies the boundary conditions for the transformed network (\ref{eq:orient2}). First of all, any given boundary sink edge $e_j$, $j\not = j_0, i_0$, ends in a $+$ region, whereas it starts in a $-$ region only if it intersects $\mathfrak{l}_{j_0}$. The latter is exactly the unique case in which $\prec E_{e_j}, \hat E_{e_j} \succ =\prec \tilde E_{e_j}, \hat E_{e_j} \succ =-1$. 

The edge $e_{i_0}$ belongs to the path $\mathcal P_0$ and it does not intersect any gauge ray in both orientations of the network. $e_{i_0}$ has a $+$ region to the left and the pair $(e,\mathfrak{l})$ is negatively oriented or it has a $-$ region to the left and the pair $(e,\mathfrak{l})$ is positively oriented (see Figure \ref{fig:edgei_0}). Therefore
\[
\gamma (e_{i_0}) = \gamma_1 (e_{i_0}) + \gamma_2 (e_{i_0}) =1.
\]
Finally $\hat E_{e_{i_0}} = (-1)^{\gamma (e_{i_0})} \tilde E_{e_{i_0}} = E_{i_0}$ since $\tilde E_{e_{i_0}} = - E_{i_0}$. 

To complete the proof we have to check that the system $\hat E_e$ defined by (\ref{eq:hat_E_P})  solves the linear system on $({\mathcal N},{\hat {\mathcal O}},\mathfrak l)$ at each internal vertex of the network. We prove it in Appendix \ref{app:orient}.

\end{proof}

\begin{example}\label{example:null_orient}
We illustrate Lemma \ref{lemma:path} for the Example \ref{example:null_orient} in Figure \ref{fig:talaska_orientation}. Let us compute the vectors $\tilde E$ using the orientation in Figure \ref{fig:talaska_orientation} [left]
and boundary condition ${\tilde E}_{u_1} = E_{u_1} - \frac{1+p+q}{2p+1} A[1] = (0, -\frac{1+p+q}{2p+1})$. Then, we immediately get
\[
\resizebox{\textwidth}{!}{$ 
{\tilde E}_{u} ={\tilde E}_{v} = \frac{q-p}{1+p+q } {\tilde E}_{u_1} = \left( 0   , \frac{p-q}{2p+1 }\right),
\quad {\tilde E}_{u_p} = \frac{p(1+2q)}{1+p+q } {\tilde E}_{u_1} = \left( 0   , -\frac{p(1+2q)}{2p+1 }\right), \quad {\tilde E}_{u_2} =  \frac{1+2p}{1+p+q } {\tilde E}_{u_1} = \left( 0   , -1\right).
$}
\]
Applying (\ref{eq:hat_E_P}), we get ${\hat E}_{-u_2} = - {\tilde E}_{u_2} =\left( 0   ,1 \right)$ and
\[
\resizebox{\textwidth}{!}{$ 
{\hat E}_{-u} = {\tilde E}_u = \left( 0   , \frac{p-q}{2p+1 }\right), \quad {\hat E}_{v} = -{\tilde E}_v = \left( 0   , \frac{q-p}{2p+1 }\right), \quad {\hat E}_{-u_p} = -	\frac{1}{p}{\tilde E}_{u_p} = \left( 0   , \frac{1+2q}{2p+1 }\right), \quad
{\hat E}_{-u_1} = {\tilde E}_{u_1},
$}
\]
since  $\gamma(u)=\gamma(u_1)=1$ and  $\gamma(v)=\gamma(u_p)=\gamma(u_2)=-1$. The latter vectors coincide with those computed directly using Theorem \ref{theo:null} in the new orientation (Figure \ref{fig:talaska_orientation}[right]): there are two untrivial conservative flows of weight $1$ and $p^{-1}$ so that
\[
\resizebox{\textwidth}{!}{$ 
{\hat E}_{-u} = \left( 0   , \frac{p}{2p+1} \Big( 1- \frac{q}{p}\Big)\right), \quad {\hat E}_{v} = \left( 0   , \frac{p}{2p+1} \Big( \frac{q}{p}-1\Big)\right), \quad {\hat E}_{-u_p} =  \left( 0   , \frac{p}{2p+1} \Big( \frac{2q+1}{p}\Big)\right), \quad
{\hat E}_{-u_1} =\left( 0   , \frac{p}{2p+1} \Big( 1+\frac{q+1}{p}\Big)\right).
$}
\]
\end{example}

\begin{figure}
  \centering{\includegraphics[width=0.3\textwidth]{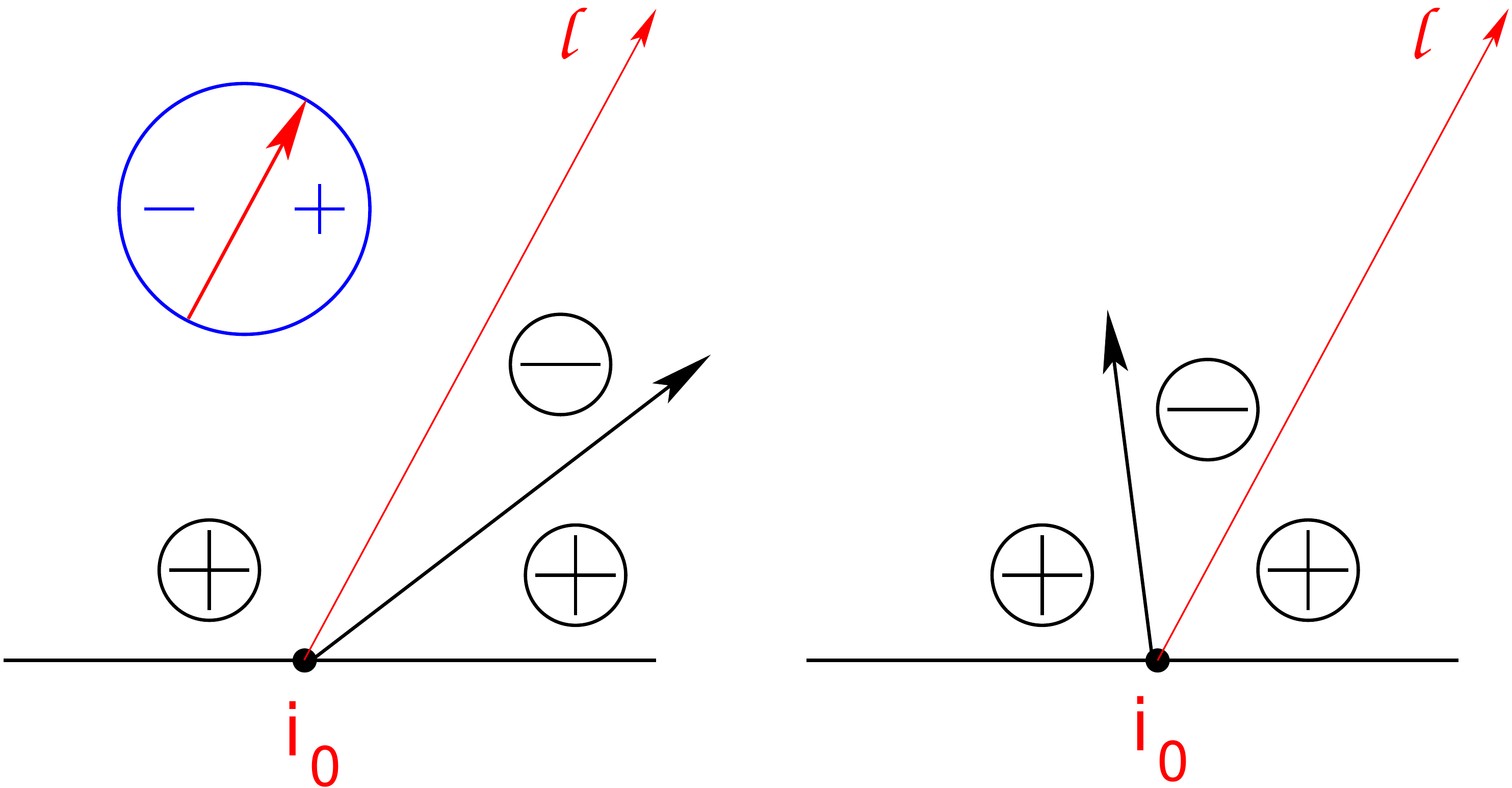}}
  \caption{\small{\sl The rule of the sign at $e_{i_0}$.}}\label{fig:edgei_0}
\end{figure}

The effect of a change of orientation along a closed simple path $\mathcal Q_0$ on the system of edge vectors follows along similar lines as above. 

 \begin{lemma}\label{lemma:cycle}\textbf{The effect of a change of orientation along a simple closed cycle.}
Let $I = \{ 1\le i_1 < i_2 < \cdots < i_k\le n\}$ be the pivot indices for $({\mathcal N},\mathcal O,\mathfrak l)$.
Let $E_e$ be the system of vectors on $({\mathcal N},\mathcal O,\mathfrak l)$ satisfying the boundary conditions $E_j$ at the boundary sinks $j\in \bar I$. Assume that we change the orientation along a simple closed cycle $\mathcal Q_0$ and let $({\mathcal N},{\hat {\mathcal O}},\mathfrak l)$ be the newly oriented network. 
Then, the system of edge vectors 
\begin{equation}\label{eq:hat_E_Q}
{\hat E}_e = \left\{ \begin{array}{ll}
 (-1)^{\gamma(e)} E_e, & \mbox{ if } e\not \in \mathcal Q_0, \mbox{ with } \gamma(e) \mbox{ as in (\ref{eq:eps_not_path})},\\
\displaystyle \frac{(-1)^{\gamma(e)}}{w_e} E_e, & \mbox{ if } e\in \mathcal Q_0, \mbox{ with } \gamma(e) \mbox{ as in (\ref{eq:eps_on_path})}.
\end{array}\right.
\end{equation}
is the system of vectors on the network  $({\mathcal N},{\hat {\mathcal O}},\mathfrak l)$ satisfying the same boundary conditions $E_j$ at the boundary sinks $j\in \bar I$.
\end{lemma}
The proof that the system ${\hat E}_e$ satisfy the linear relations for the transformed network is presented in Appendix~\ref{app:orient}.

\subsection{Dependence of edge vectors on weight and vertex gauge freedoms}\label{sec:different_gauge}

\begin{figure}
  \centering{\includegraphics[width=0.48\textwidth]{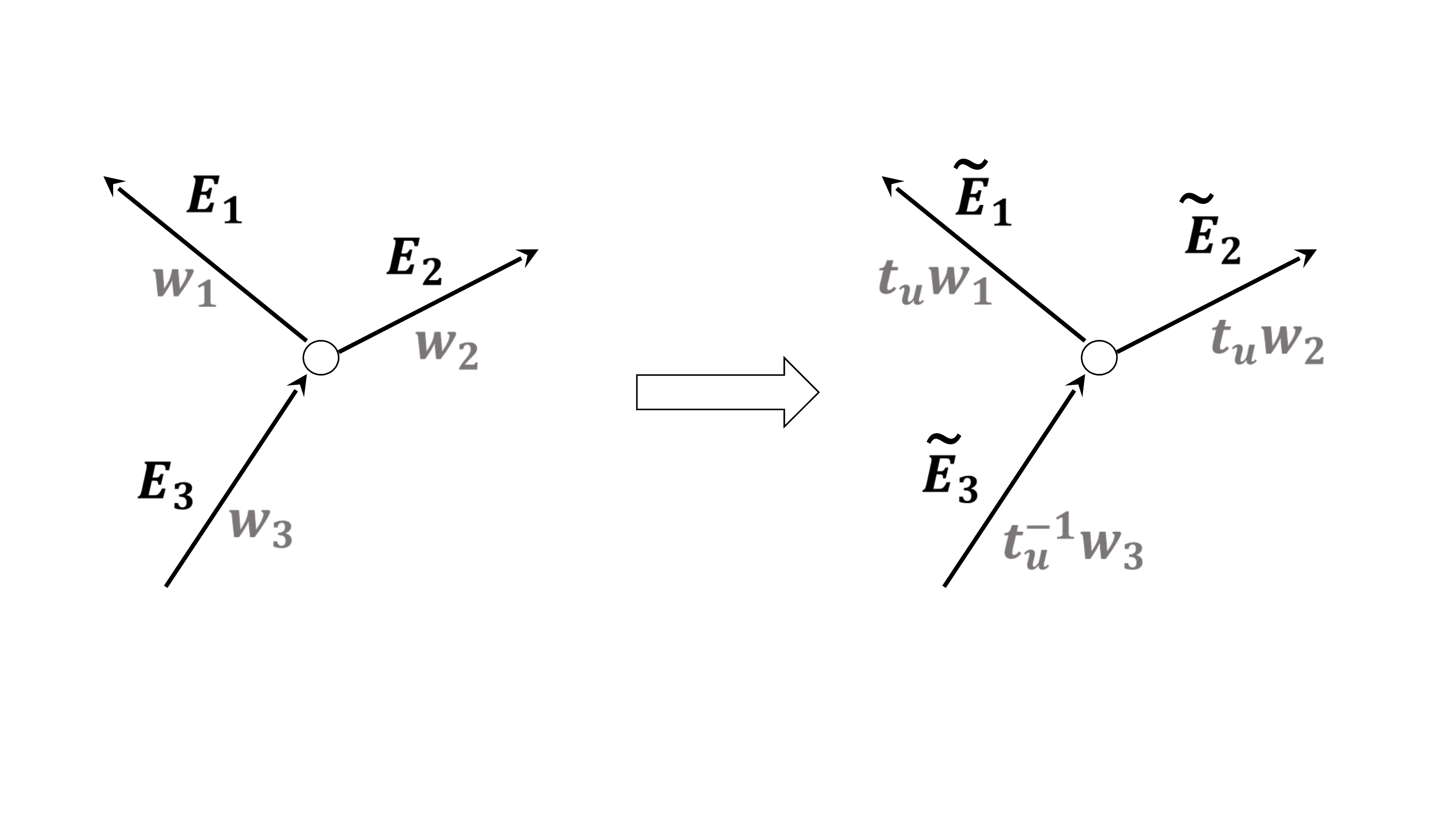}
	\hfill
	\includegraphics[width=0.48\textwidth]{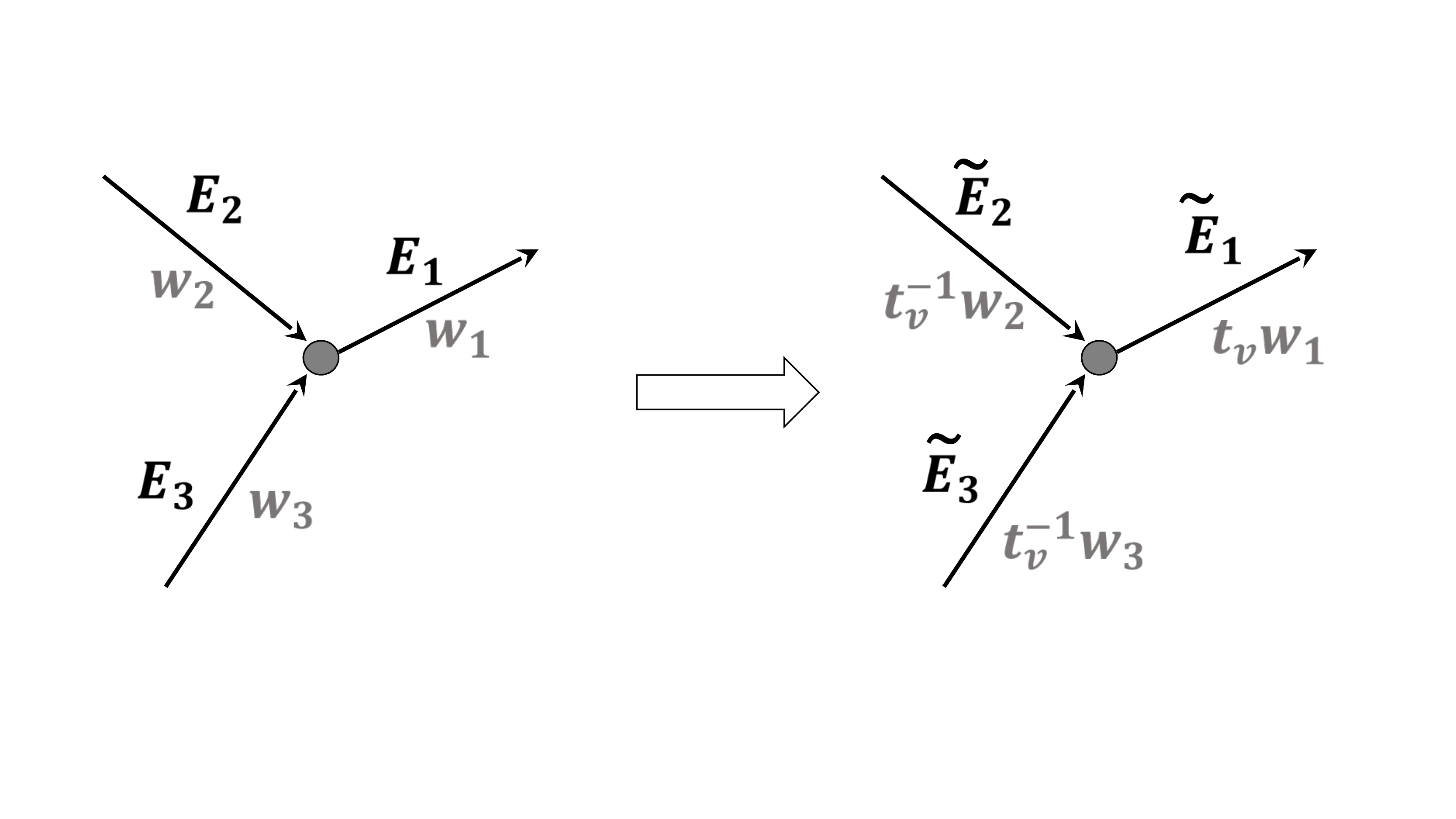}}
	\vspace{-1.5 truecm}
  \caption{\small{\sl The effect of the weight gauge transformation at a white [left] and at a black [right] vertex on the edge vectors.}\label{fig:weight_gauge_vectors}}
\end{figure}

Next we discuss the effect of the weight gauge and vertex gauge feeedom on the system of edge vectors: in both cases it is just local.

For any given $[A]\in \S$, $\mathcal N$ denotes a network representing $[A]$ with graph $\mathcal G$ and edge weights $w_e$. There is a fundamental difference in the gauge freedom of assigning weights depending on whether or not the graph $\mathcal G$ is reduced \cite{Pos}.

\begin{remark}\label{rem:gauge_weight}\textbf{The weight gauge freedom \cite{Pos}.} Given a point $[A]\in\S$
and a planar directed graph ${\mathcal G}$ in the disk representing $\S$, then $[A]$
 is represented by infinitely many gauge equivalent systems of weights $w_e$ on the edges $e$ of ${\mathcal G}$. Indeed, if a positive number $t_V$ is assigned to each internal vertex $V$, whereas $t_{b_i}=1$ for each boundary vertex $b_i$, then the transformation on each directed edge $e=(U,V)$
\begin{equation}
\label{eq:gauge}
w_e\rightarrow w_e t_U \left(t_V\right)^{-1},
\end{equation}
transforms the given directed network into an equivalent one representing $[A]$. 
\end{remark}

\begin{remark}\label{rem:gauge_freedom}\textbf{The unreduced graph gauge freedom.} As it was pointed out in \cite{Pos}, for unreduced directed graphs there is no one-to-one correspondence between the orbits of the gauge weight action (\ref{eq:gauge}) and the points in the corresponding positroid cell. Since we do not consider graphs with components isolated from the boundary, this extra gauge freedom arises if we apply the creation of parallel edges and leafs (see Section~\ref{sec:moves_reduc}). In Section \ref{sec:null_vectors} we show that in contrast with gauge transformations of the weights (\ref{eq:gauge}), the unreduced graph gauge freedom affects the system of edge vectors untrivially.
\end{remark}

\begin{lemma}\label{lem:weight_gauge}\textbf{Dependence of edge vectors on the weight gauge}
Let $E_e$ be the edge vectors on the network $({\mathcal N}, {\mathcal O}, \mathfrak{l})$.
\begin{enumerate}
\item Let ${\tilde E}_e$ be the system of edge vectors on $(\tilde {\mathcal N}, {\mathcal O}, \mathfrak{l})$, where $\tilde {\mathcal N}$ is obtained from ${\mathcal N}$ applying the weight gauge transformation at a white trivalent vertex as in Figure \ref{fig:weight_gauge_vectors} [left]. Then 
\begin{equation}\label{eq:wg_vector_white}
{\tilde E}_e = \left\{ \begin{array}{ll} E_e, &\quad  \forall e\in \mathcal N, \;\; e\not = e_1,e_2,\\
t_u E_e &\quad \mbox{ if } e = e_1,e_2.
\end{array}
\right.
\end{equation}
\item Let ${\tilde E}_e$ be the system of edge vectors on $(\tilde {\mathcal N}, {\mathcal O}, \mathfrak{l})$, where $\tilde {\mathcal N}$ is obtained from ${\mathcal N}$ applying the weight gauge transformation at a black trivalent vertex as in Figure \ref{fig:weight_gauge_vectors} [right]. Then
\begin{equation}\label{eq:wg_vector_black}
{\tilde E}_e = \left\{ \begin{array}{ll} E_e, &\quad \forall e\in \mathcal N, \;\;  e\not = e_1,\\
t_v E_e &\quad \mbox{ if } e = e_1.
\end{array}
\right.
\end{equation}
\end{enumerate}
\end{lemma}

The proof is straightforward and is omitted.
\begin{figure}
  \centering{\includegraphics[width=0.6\textwidth]{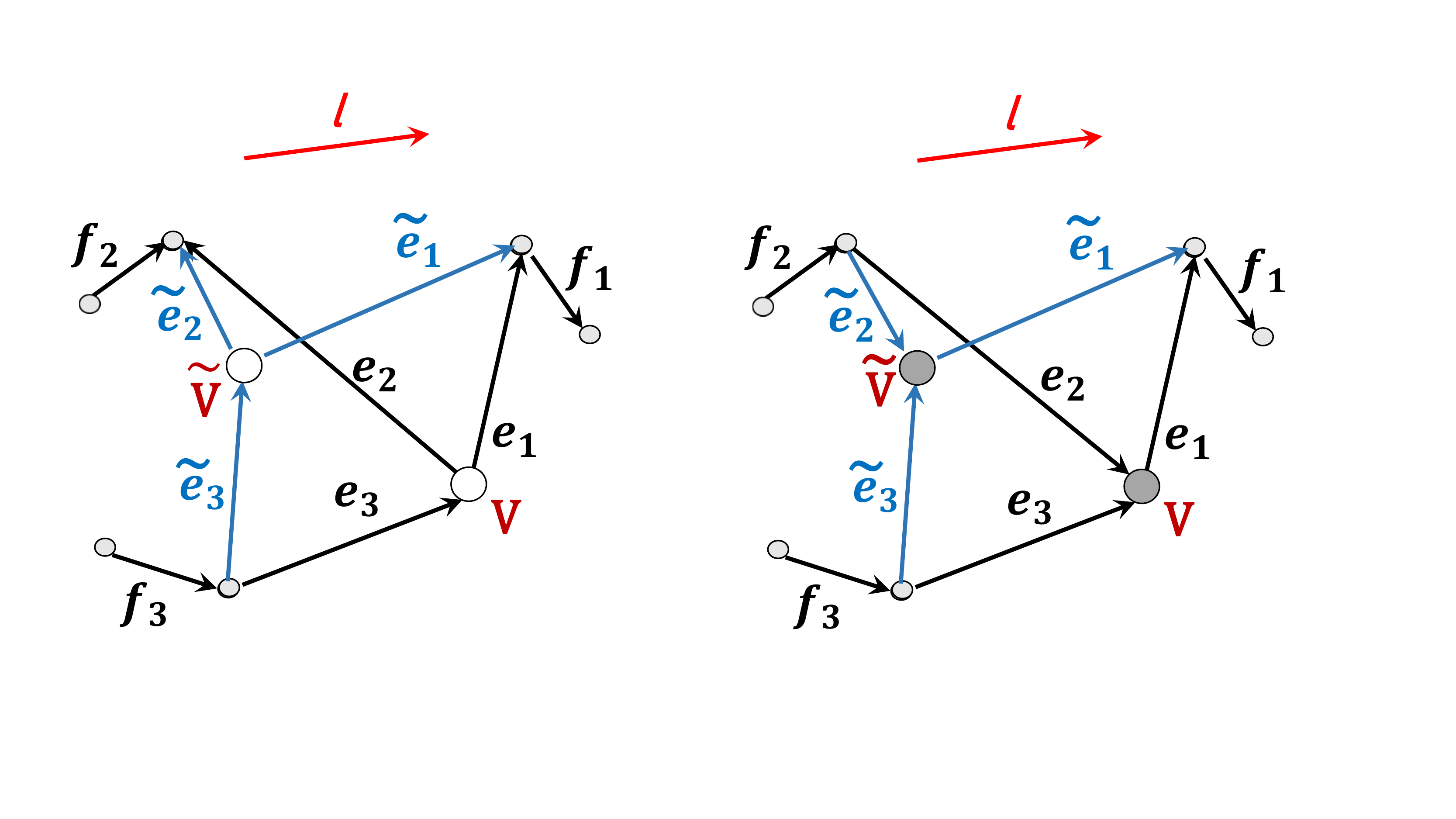}}
	\vspace{-.9 truecm} 
  \caption{\small{\sl The vertex gauge transformation at a white [left] and at a black [right] vertex consists in moving an internal vertex from position $V$ to $\tilde V$.}\label{fig:vertex_gauge_vectors}}
\end{figure}

\begin{remark}\label{rem:gauge_vertices}\textbf{Vertex gauge freedom of the graph} The boundary map is the same if we move vertices in $\mathcal G$ without changing their relative positions in the graph. Such transformation acts on edges via rotations, translations and contractions/dilations of their lenghts. Any such transformation may be decomposed in a sequence of elementary transformations in which a single vertex is moved whereas all other vertices remain fixed (see also Figure \ref{fig:vertex_gauge_vectors}).
\end{remark}

This transformation effects only the three edge vectors incident at the moving vertex and the latter may only change of sign. 

\begin{lemma}\label{lem:vertex_gauge}\textbf{Dependence of edge vectors on the vertex gauge}
\begin{enumerate}
\item Let $E_e$ and ${\tilde E}_e$ respectively be the system of edge vectors on $({\mathcal N}, {\mathcal O}, \mathfrak{l})$ and on $({\tilde {\mathcal N}}, {\mathcal O}, \mathfrak{l})$, where ${\tilde {\mathcal N}}$ is obtained from ${\mathcal N}$ moving one internal white vertex where notations are as in Figure \ref{fig:vertex_gauge_vectors}[left]. Then ${\tilde E}_e =E_e$, for all $e\not = e_1,e_2,e_3$ and
\begin{equation}\label{eq:white_vertex_gauge}
\resizebox{\textwidth}{!}{$ 
{\tilde E}_{e_i} = (-1)^{\mbox{wind}({\tilde e}_i,f_i)- \mbox{wind}(e_i,f_i)+\mbox{int}({\tilde e}_i)-\mbox{int}(e_i)} E_{e_i}, \quad i=1,2,
\quad\quad
{\tilde E}_{e_3} = (-1)^{\mbox{wind}(f_3,{\tilde e}_3)-\mbox{wind}(f_3,e_3)  } E_{e_3};
$}
\end{equation}
\item Let $E_e$ and ${\tilde E}_e$ respectively be the system of edge vectors on $({\mathcal N}, {\mathcal O}, \mathfrak{l})$ and on $({\tilde {\mathcal N}}, {\mathcal O}, \mathfrak{l})$, where ${\tilde {\mathcal N}}$ is obtained from ${\mathcal N}$ moving one internal black vertex as in Figure \ref{fig:vertex_gauge_vectors}[right]. Then ${\tilde E}_e =E_e$, for all $e\not = e_1,e_2,e_3$ and
\[
\resizebox{\textwidth}{!}{$ 
{\tilde E}_{e_1} = (-1)^{\mbox{wind}({\tilde e}_1,f_1)- \mbox{wind}(e_1,f_1)+\mbox{int}({\tilde e}_1)-\mbox{int}(e_1)} E_{e_1},
\quad\quad
{\tilde E}_{e_i} = (-1)^{\mbox{wind}(f_i,{\tilde e}_i)-\mbox{wind}(f_i,e_i)  } E_{e_i}, \quad i=2,3;
$}
\]
\end{enumerate}
\end{lemma}

\begin{proof}
The statement follows from the linear relations at vertices and the following identities at trivalent white vertices
\begin{equation}\label{eq:int_vertex_gauge}
\begin{array}{c}
\mbox{int} (e_i) +\mbox{int} (e_3) =\mbox{int} ({\tilde e}_i) +\mbox{int} ({\tilde e}_3), \quad (\!\!\!\!\!\!\mod 2),\\
\mbox{wind}(f_3,e_3) +  \mbox{wind}(e_3,e_i) +  \mbox{wind}(e_i,f_i) = \mbox{wind}(f_3,{\tilde e}_3) +  \mbox{wind}({\tilde e}_3,{\tilde e}_i) +  \mbox{wind}({\tilde e}_i,f_i) \quad (\!\!\!\!\!\!\mod 2),
\end{array}
\end{equation}
and the corresponding identities at black vertices and the next Lemma.
\end{proof}
\begin{lemma}
 Denote the vertex we move by $V$.  
\begin{enumerate}
\item If $e_1$ is an outgoing vector for $V$, and the vertex $V_1$ at the end of $e_1$ is white trivalent with the outgoing vectors $f_1$, $f_2$, then
  $$
  \mbox{wind}(e_1,f_1) - \mbox{wind}(\tilde e_1,f_1) = \mbox{wind}(e_1,f_2) - \mbox{wind}(\tilde e_1,f_2)
  \quad (\!\!\!\!\!\!\mod 2);
  $$
\item If $e_3$ is an incoming vector for $V$, and the vertex $V_3$ at the beginning of $e_3$ is black trivalent with the incoming vectors $g_1$, $g_2$, then
  $$
  \mbox{wind}(g_1,e_3) - \mbox{wind}(g_1,\tilde e_3) = \mbox{wind}(g_2,e_3) - \mbox{wind}(g_2,\tilde e_3). 
  \quad (\!\!\!\!\!\!\mod 2);
  $$
\begin{proof}
  Let us proof the first statement. When we move $V$, the vector $e_1$ continuously changes the direction. By Lemma~\ref{lem:rotation}, the winding numbers $\mbox{wind}(e_1,f_1)$  ($\mbox{wind}(e_1,f_2)$ respectively) changes if $e_1$ intersects $\mathfrak l$ or $-f_1$ ($\mathfrak l$ or $-f_2$ respectively). We keep the topology of the graph fixed, therefore $e_1$ cannot intersect either $-f_1$ or $-f_2$, therefore both windings may only change simultaneously.
  
  The proof of the second statement uses the same arguments.
\end{proof} 
\end{enumerate}  
\end{lemma}

\section{Effect of moves and reductions on edge vectors}\label{sec:moves_reduc}

In \cite{Pos} it is introduced a set of local transformations - moves and reductions - on planar bicolored networks in the disk which leave invariant the boundary measurement map. Two networks in the disk connected by a sequence of such moves and reductions represent the same point in $\GTNN$. 
There are three moves,
(M1) the square move (Figure \ref{fig:squaremove}),
(M2) the unicolored edge contraction/uncontraction (Figure \ref{fig:flipmove}),
(M3) the middle vertex insertion/removal (Figure \ref{fig:middle}),
and three reductions
(R1) the parallel edge reduction (Figure \ref{fig:parall_red_poles}),
(R2) the dipole reduction (Figure \ref{fig:dip_leaf}[left]),
(R3) the leaf reduction (Figure \ref{fig:dip_leaf}[right]).

In our construction each such transformation induces a well defined change in the system of edge vectors.
In the following, we restrict ourselves to plabic networks and, without loss of generality, we fix both the orientation and the gauge ray direction since their effect on the system of vectors is completely under control in view of the results of Section \ref{sec:vector_changes}.
We denote $({\mathcal N}, \mathcal O, \mathfrak l)$ the initial oriented 
network and $({\tilde {\mathcal N}}, {\tilde{\mathcal O}}, \mathfrak l)$ the oriented network obtained from it by applying one move (M1)--(M3) or one reduction 
(R1)--(R3). We assume that the orientation ${\tilde {\mathcal O}}$ coincides with $\mathcal O$ at all edges except at those involved in the move or reduction where we use Postnikov rules to assign the orientation. We denote with the same symbol and a tilde any quantity referring to the transformed network. 

\smallskip

{\bf (M1) The square move}
If a network has a square formed by four trivalent vertices
whose colors alternate as one goes around the square, then one can switch the colors of these
four vertices and transform the weights of adjacent faces as shown in Figure \ref{fig:squaremove}.
The relation between the face weights before and after the square move is \cite{Pos}
${\tilde f}_5 = (f_5)^{-1}$, ${\tilde f}_1 = f_1/(1+1/f_5)$, ${\tilde f}_2 = f_2 (1+f_5)$, ${\tilde f}_3 = f_3 (1+f_5)$, ${\tilde f}_4 = f_4/(1+1/f_5)$,
so that the relation between the edge weights with the orientation in Figure \ref{fig:squaremove} is
${\tilde \alpha}_1 = \frac{\alpha_3\alpha_4}{{\tilde\alpha}_2}$, 
${\tilde \alpha}_2 = \alpha_2 + \alpha_1\alpha_3\alpha_4$, ${\tilde \alpha}_3 = \alpha_2\alpha_3/{\tilde \alpha}_2$,  ${\tilde \alpha}_4 = \alpha_1\alpha_3/{\tilde \alpha}_2$.

The system of equations on the edges outside the square is the same before and after the move and also the boundary
conditions remain unchanged. The uniqueness of the solution implies that all vectors outside the square 
including $E_1$, $E_2$, $E_3$, $E_4$ remain the same. In the following Lemma, $e_j, h_k$ respectively are the edges carrying the vectors $E_j,F_k$ in the initial configuration. For instance, $\mbox{int}(h_1)$ is the number of intersections of gauge rays with the edge carrying the vector $F_1$ in the initial configuration, whereas $\mbox{wind}(-h_1,h_2)$ is the winding number of the pair of edges carrying the vectors $\tilde F_1, \tilde F_2$ after the move because the edge $h_1$ has changed of versus.

\begin{figure}
\centering{\includegraphics[width=0.45\textwidth]{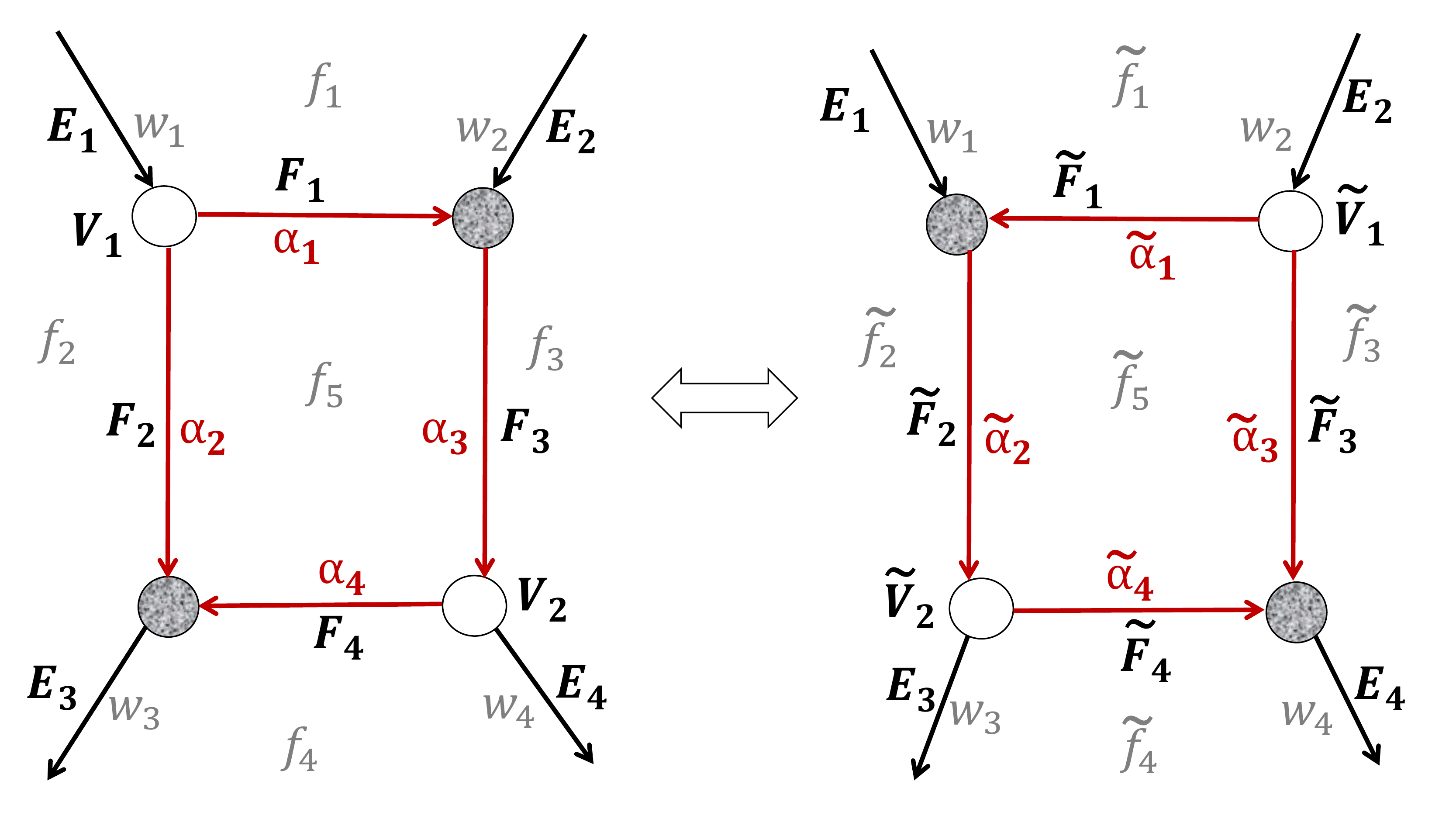}}
\caption{\small{\sl The effect of the square move.}\label{fig:squaremove}}
\end{figure}

In the next Lemma we provide the transformation rules between the vectors $F_j$,${\tilde F}_j$ and $E_j$, $j\in [4]$, assuming the orientation as in Figure \ref{fig:squaremove}.
\begin{lemma}
  The vectors $\tilde F$ after the square move are given by:
$$
{\tilde F}_4 = \tilde\alpha_4 \alpha_3^{-1} (-1)^{\mbox{int}(h_3)+\mbox{int}(h_4)+\gamma_2(h_4) +\mbox{wind}(h_3,h_4)+1 } F_3 + \tilde\alpha_4 (-1)^{\mbox{int}(h_4)+\gamma_2(h_4) } F_4,
$$
$$
{\tilde F}_2 = \alpha_1 (-1)^{\mbox{int}(h_1) +\gamma_2(h_4) + \mbox{wind}(h_2,-h_4) +\mbox{wind}(h_3,h_4)+1}  F_3 +\alpha_2 \alpha_4^{-1} (-1)^{1+\mbox{int}(h_2)+\mbox{int}(h_4)+\gamma_2(h_4)+  \mbox{wind}(h_2,-h_4)} F_4.
$$
$$
{\tilde F}_1 = \tilde\alpha_1 (-1)^{\mbox{int}(h_1)+\mbox{wind}(-h_1,h_2)} {\tilde F}_2,
$$
$$
{\tilde F}_3 =  (-1)^{\mbox{int}(h_3)+ \mbox{int}(h_4) + \gamma_2(h_4) +\mbox{wind}(h_3,h_4)+1 } \alpha_2 \alpha_1^{-1}  {\tilde F}_4,
$$

\end{lemma}  
\begin{proof}
Using the linear relations
\begin{equation}\label{eq:move1}
\begin{array}{l}
\displaystyle F_1 = (-1)^{\mbox{int}(h_1)+\mbox{wind}(h_1,h_3)} \alpha_1 F_3, \quad F_2 = (-1)^{\mbox{int}(h_2)+\mbox{wind}(h_2,e_3)} \alpha_2E_3,\quad F_4 = (-1)^{\mbox{int}(h_4)+\mbox{wind}(h_4,e_3)} \alpha_4 E_3,
\\
F_3 = (-1)^{\mbox{int}(h_3)} \alpha_3 \left( (-1)^{\mbox{int}(h_4)+\mbox{wind}(h_3,h_4)+\mbox{wind}(h_4,e_3)}\alpha_4 E_3 + (-1)^{\mbox{wind}(h_3,e_4)} E_4  \right),
\end{array}
\end{equation}
\begin{equation}\label{eq:move2}
\begin{array}{l}
{\tilde F}_3 =  (-1)^{\mbox{int}(h_3)+\mbox{wind}(h_3,e_4)} {\tilde \alpha}_3 E_4, \quad {\tilde F}_4 =  (-1)^{\mbox{int}(h_4)+\mbox{wind}(-h_4,e_4)} {\tilde \alpha}_4 E_4, \\
{\tilde F}_2 =  (-1)^{\mbox{int}(h_2)} {\tilde \alpha}_2 \left( \,  (-1)^{\mbox{wind}(h_2,e_3)} E_3 + (-1)^{\mbox{int}(h_4)+\mbox{wind}(h_2,-h_4)+\mbox{wind}(-h_4,e_4)} {\tilde \alpha}_4  E_4  \right),
\end{array}
\end{equation}
we immediately get
$$
{\tilde F}_4 = \tilde\alpha_4 \alpha_3^{-1} (-1)^{\mbox{int}(h_3)+\mbox{int}(h_4)+ \mbox{wind}(h_3,e_4) +\mbox{wind}(-h_4,e_4)} F_3 + \tilde\alpha_4 (-1)^{1+\mbox{int}(h_4) + \mbox{wind}(h_3,h_4)+ \mbox{wind}(h_3,e_4) +\mbox{wind}(-h_4,e_4) } F_4.
$$
Since the triple $h_4,e_4,-h_3$ is oriented counterclockwise, the cyclic order $[h_4,e_4,-h_3]=0$ (see Definition \ref{def:cyclic_order}), and the statement follows from (\ref{eq:white2}).

Analogously,
$$
{\tilde F}_2 = \tilde\alpha_2 \alpha_4^{-1} (-1)^{\mbox{int}(h_2)+\mbox{int}(h_4)+ \mbox{wind}(h_2,e_3) +\mbox{wind}(h_4,e_3)} F_4 + \tilde\alpha_2 (-1)^{\mbox{int}(h_2) + \mbox{wind}(h_2,-h_4)}{\tilde F}_4.
$$
Again, the triple $-h_2,e_3,-h_4$ is oriented counterclockwise, the cyclic order $[-h_2,e_3,-h_4]=0$ and the statement follows from (\ref{eq:white2}).
\end{proof}

\smallskip

\begin{figure}
  \centering{\includegraphics[width=0.49\textwidth]{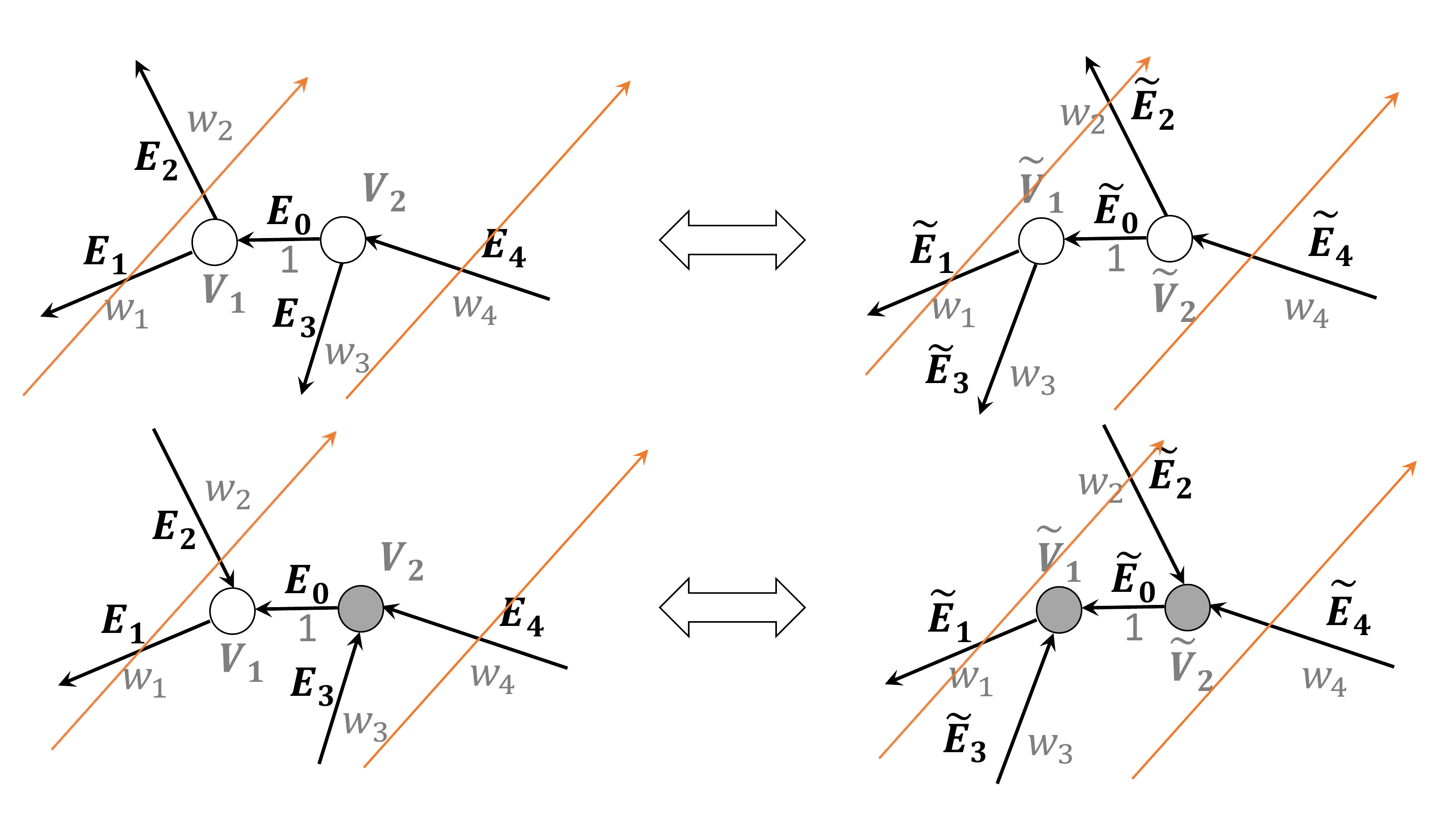}}
  \caption{\small{\sl The insertion/removal of an unicolored internal vertex is equivalent to a flip move of the unicolored vertices.}\label{fig:flipmove}}
\end{figure}

{\bf (M2) The unicolored edge contraction/uncontraction}
The unicolored edge contraction/un\-con\-traction consists in the elimination/addition of an internal vertex of equal color and of a unit edge, and it leaves invariant the face weights and the boundary measurement map \cite{Pos}. 

The contraction/uncontraction of an unicolored internal edge combined with the trivalency condition is equivalent to a flip of the unicolored vertices involved in the move (see Figure \ref{fig:flipmove}). We consider only pure flip moves, i.e.
all vertices keep the same positions before and after the move. Moreover, we assume that the edge $e_0$ connecting this pair of vertices has unit weight and sufficiently small length so that no gauge ray crosses it, all other edges
preserve their intersection numbers, and the winding at the vertices not involved in the move remain invariant.
Therefore, additivity of the winding numbers holds in this special case $\mbox{wind} (e_i,e_0)+ \mbox{wind} (e_0,e_j) = \mbox{wind} (e_i,e_j),$ with $e_i$ -- any incoming vector, $e_j$ -- any outgoing vector involved in the move.

Finally,
$$
{\tilde F}_i = F_i
$$
in all cases, 
$$
{\tilde E}_0 =
\left\{
\begin{array}{ll} E_0, & \mbox{if the vertices are black}, \\
  E_0 - (-1)^{\mbox{wind}(e_0, e_3)} F_3 +(-1)^{\mbox{wind}(e_0, e_2)} F_2, &
       \mbox{if the vertices are white.}
\end{array}
\right.
$$
We remark that the flip move may create/eliminate null edge vectors. For instance suppose that $(-1)^{\mbox{wind}(e_0, e_1)} F_1 +(-1)^{\mbox{wind}(e_0, e_3)} F_3=0$ and $(-1)^{\mbox{wind}(e_0, e_1)} F_1 +(-1)^{\mbox{wind}(e_0, e_2)} F_2\not =0$. Then in the initial configuration all edge vectors are different from zero whereas in the final $\tilde E_0=0$.

\smallskip

{\bf (M3) The middle edge insertion/removal}

\begin{figure}
  \centering
	{\includegraphics[width=0.49\textwidth]{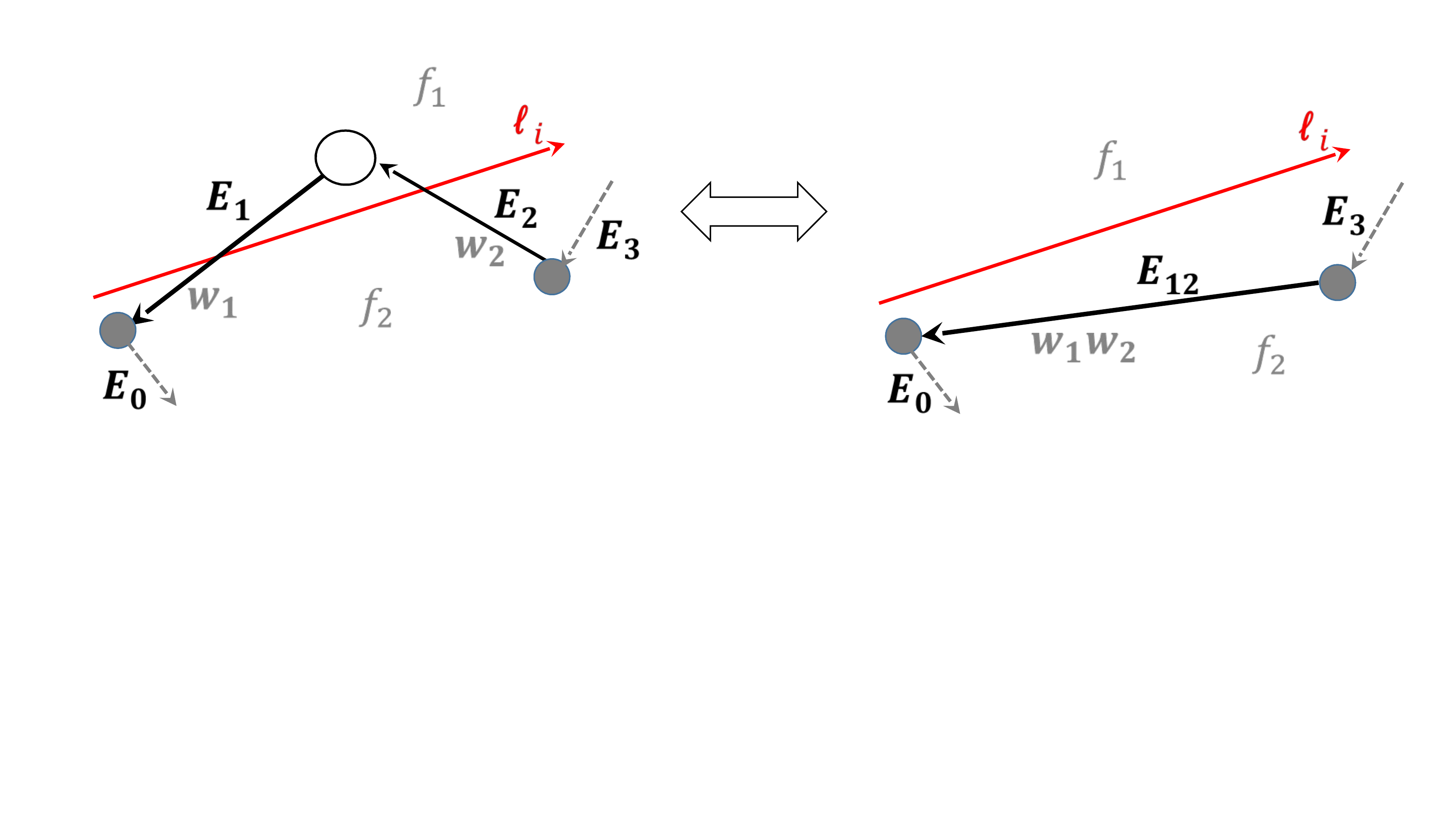}}
	\hfill
	{\includegraphics[width=0.49\textwidth]{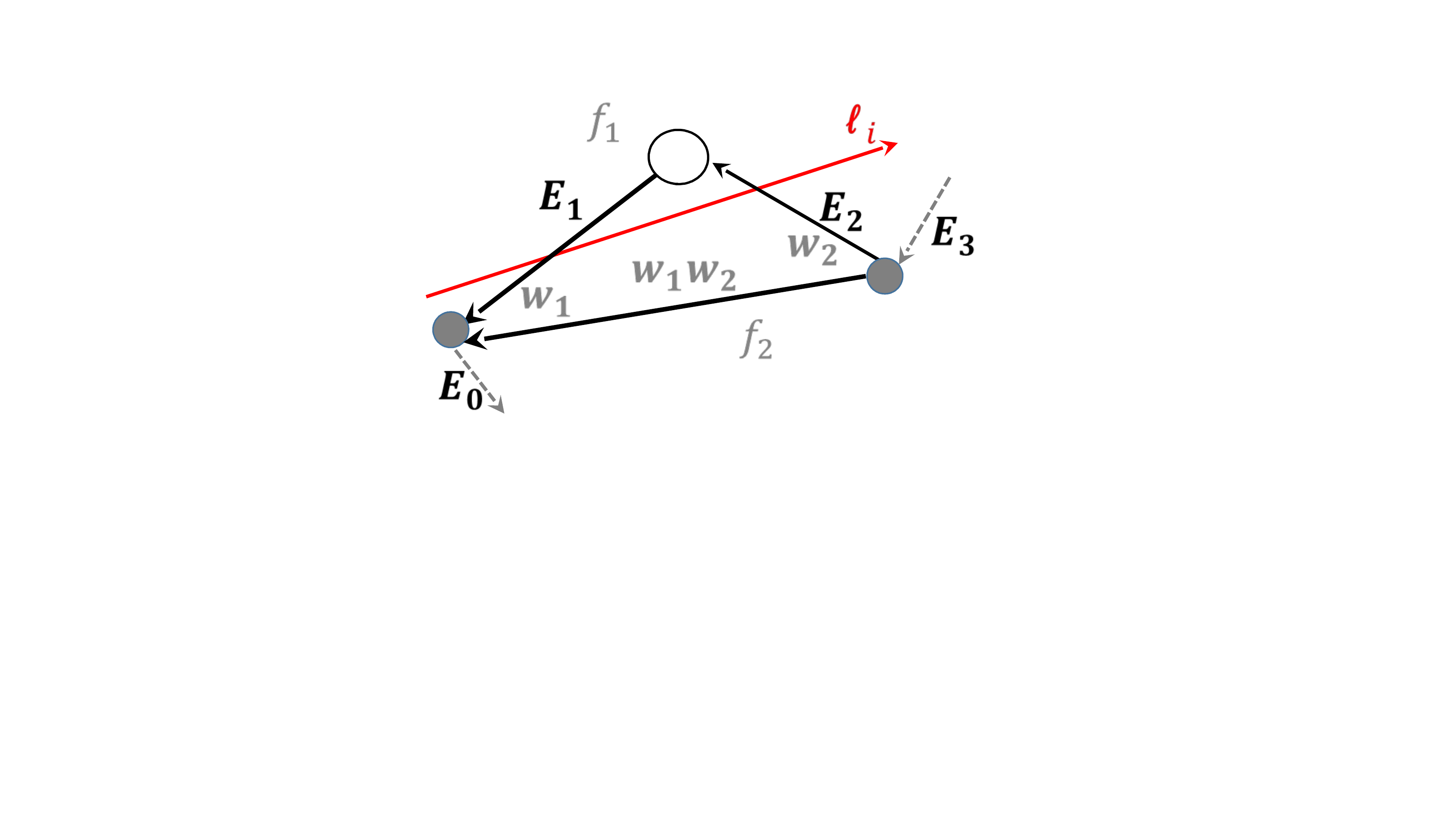}}
	\vspace{-2.1 truecm}
  \caption{\small{\sl The middle edge insertion/removal.}\label{fig:middle}}
\end{figure}
The middle edge insertion/removal consists in the addition/elimination of bivalent vertices (see Figure \ref{fig:middle}) without changing the face configuration. i.e. the triangle formed by the edges $e_1$, $e_2$, $e_{12}$ does not contain other edges of the network. Then the action of such move is trivial, since 
$\mbox{int} (e_1) +\mbox{int} (e_2) = \mbox{int} (e_{12}), \quad (\!\!\!\!\mod 2)$, 
$\mbox{wind}(e_3, e_2) +\mbox{wind}(e_2, e_1)+\mbox{wind}(e_1, e_0) =\mbox{wind}(e_3, e_{12}) +\mbox{wind}(e_{12}, e_0), \quad (\!\!\!\!\mod 2)$ so that the relation between the vectors $E_2$ and $E_{12}$ is simply, 
$$
E_{12} =(-1)^{\mbox{wind}(e_3, e_2) -\mbox{wind}(e_3, e_{12})} E_2.
$$

\smallskip

{\bf (R1) The parallel edge reduction}
The parallel edge reduction consists of the removal of two trivalent vertices of different color connected by a pair of parallel edges (see Figure \ref{fig:parall_red_poles}[top]). If the parallel edge separates two distinct faces, the relation of the face weights before and after the reduction is 
${\tilde f}_1 = \frac{f_1}{1+(f_0)^{-1}}$, ${\tilde f}_2 = f_2 (1+f_0)$,
otherwise ${\tilde f}_1 = {\tilde f}_2 = f_1 f_0$ \cite{Pos}. In both cases, for the choice of orientation in Figure \ref{fig:parall_red_poles}, the relations between the edge weights and the edge vectors respectively are 
${\tilde w}_1 = w_1(w_2+w_3)w_4$,
$ {\tilde E}_{1} = E_1 = (-1)^{\mbox{int}(e_1)+\mbox{int}(e_2)}w_1(w_2+w_3) E_4$, $
E_2 = (-1)^{\mbox{int}(e_2)}w_2 E_4$, $E_3 = (-1)^{\mbox{int}(e_2)}w_3 E_4$,
since $\mbox{wind} (e_1,e_2) =\mbox{wind} (e_1,e_3) =\mbox{wind} (e_2,e_4) =\mbox{wind} (e_3,e_4) =0$, $\mbox{int}(e_1)+\mbox{int}(e_2)+\mbox{int}(e_4) =\mbox{int}({\tilde e}_1)$ and $\mbox{int}(e_2) =\mbox{int}(e_3)$.

\begin{figure}
  \centering{\includegraphics[width=0.55\textwidth]{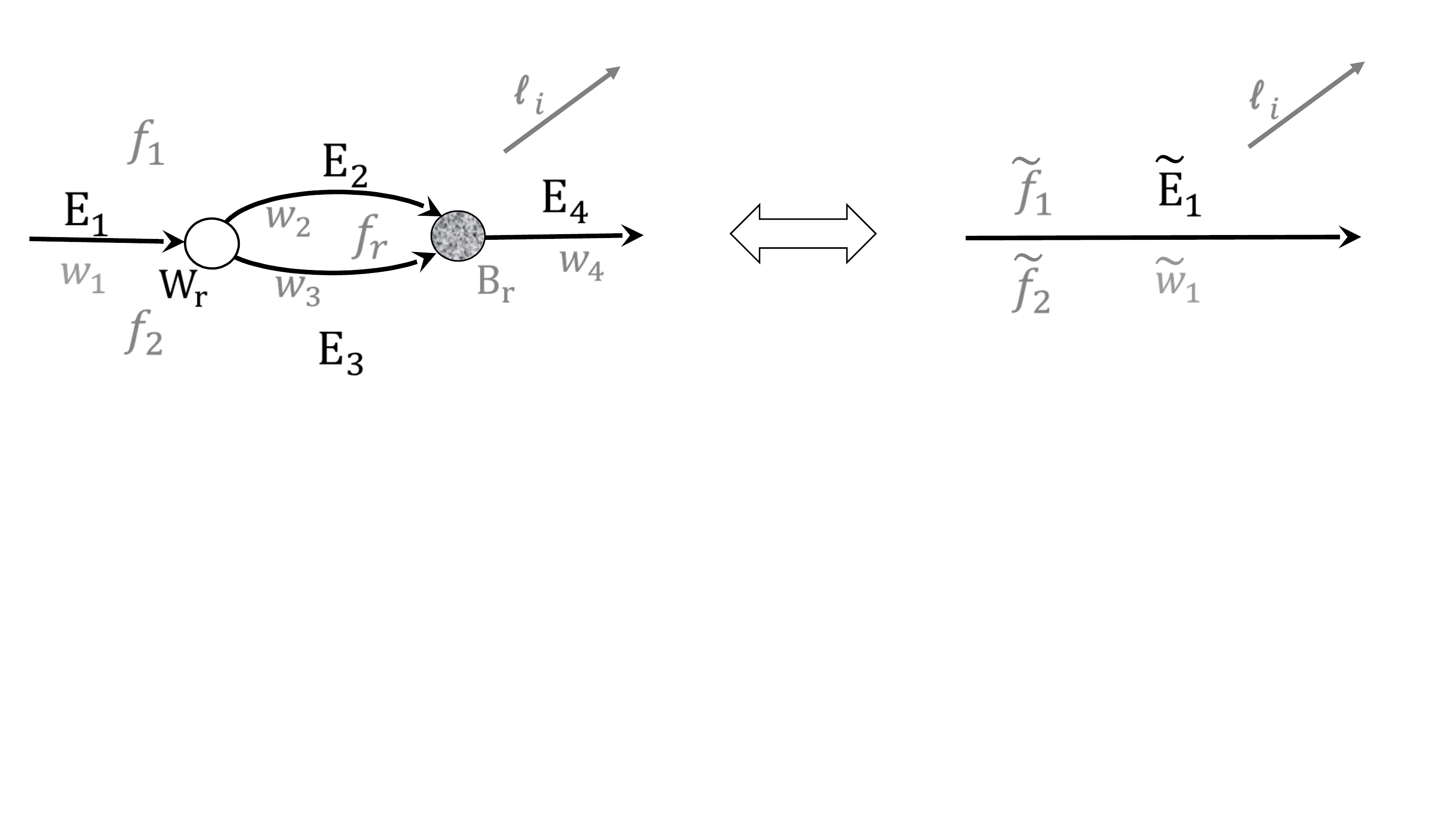}}
	\vspace{-3. truecm}
  \caption{\small{\sl The parallel edge reduction.}\label{fig:parall_red_poles}}
\end{figure}

{\bf (R2) The dipole reduction}
The dipole reduction eliminates an isolated component consisting of two vertices joined by an edge $e$ (see Figure \ref{fig:dip_leaf}[left]). The transformation leaves invariant the weight of the face containing such component. Since the edge vector at $e$ is $E_e=0$, this transformation acts trivially on the vector system.

\begin{figure}
\includegraphics[width=0.48\textwidth]{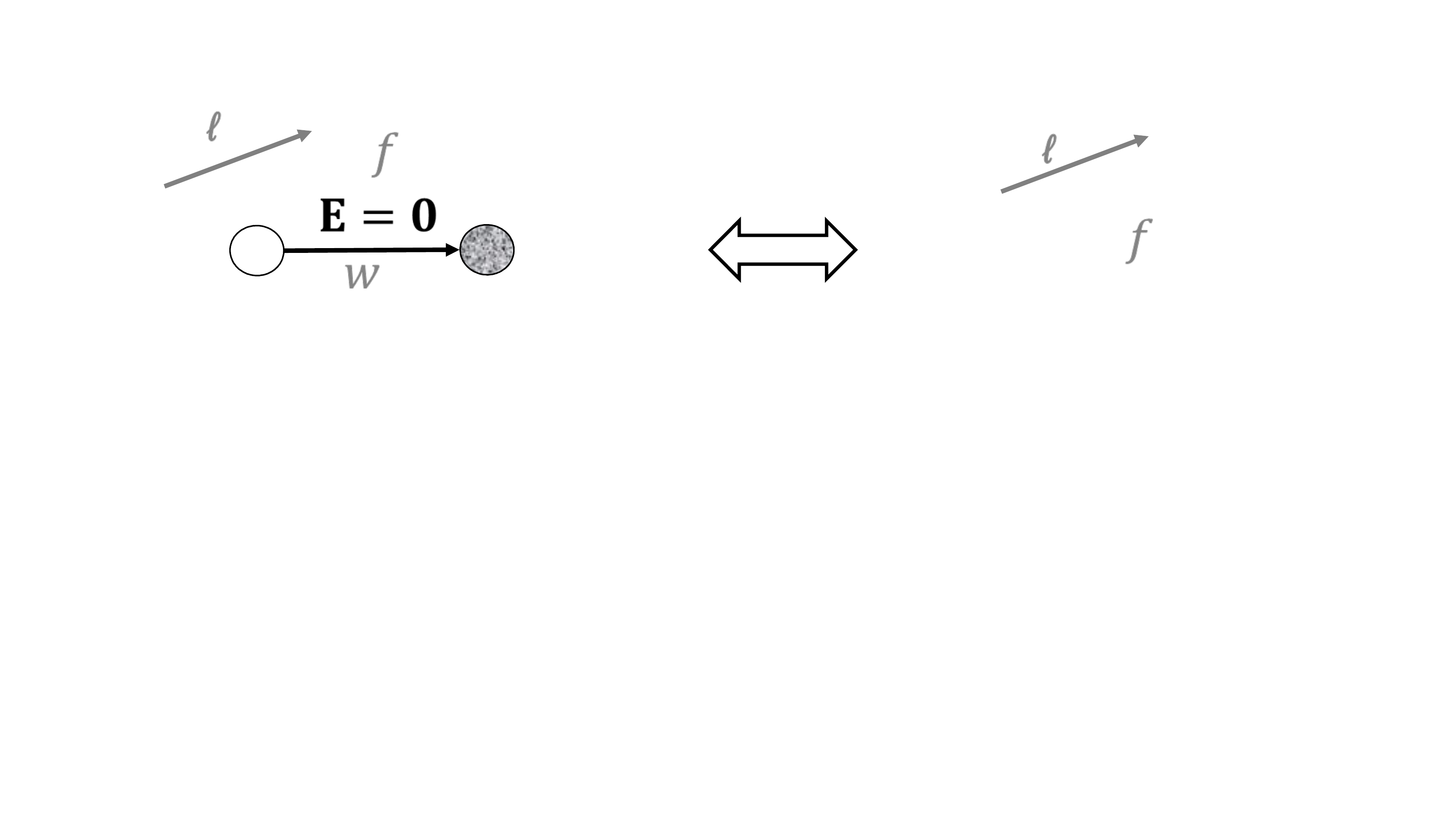}
\hfill
\includegraphics[width=0.48\textwidth]{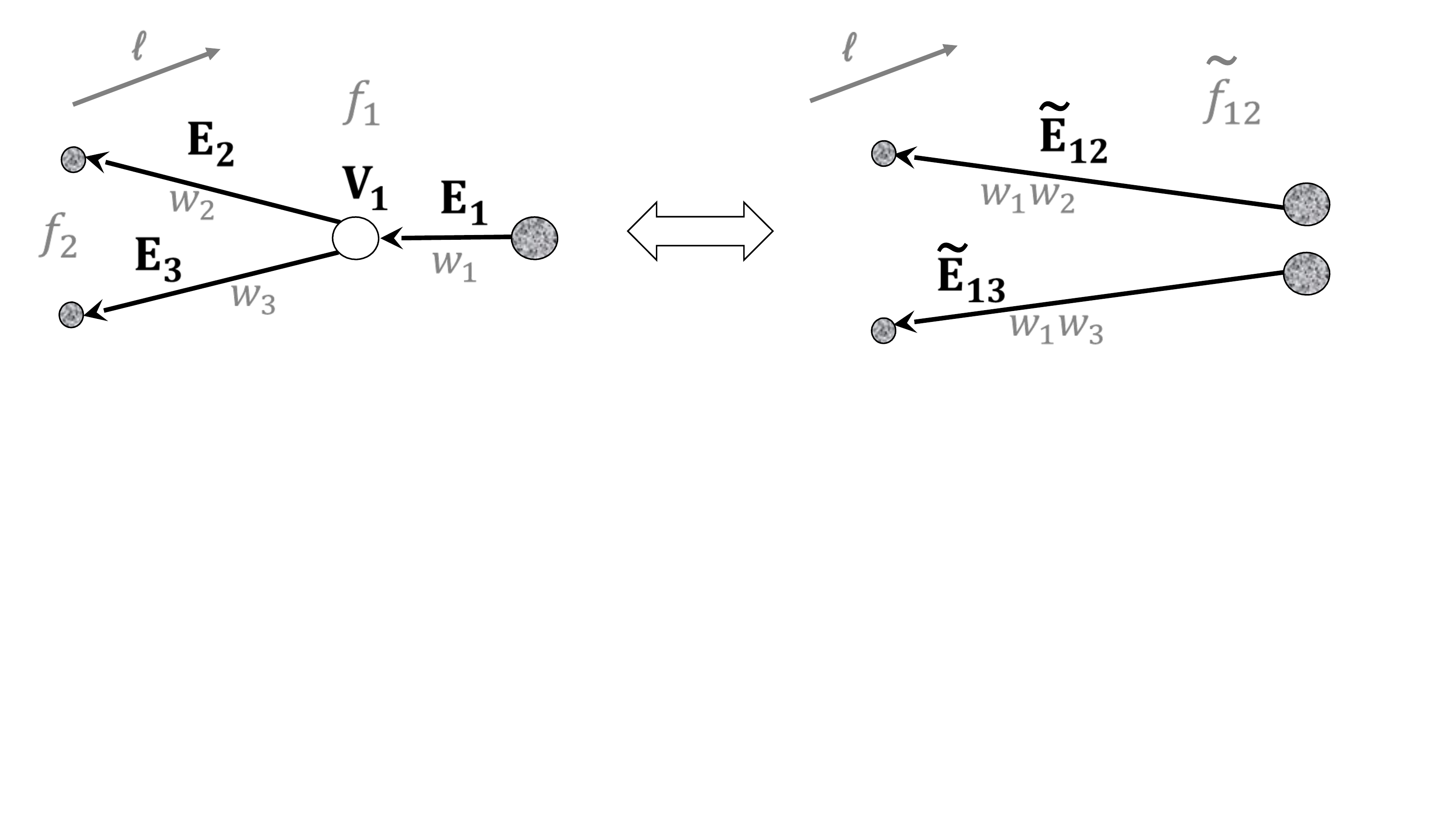}
\vspace{-2.5 truecm}
\caption{\small{\sl Left: the dipole reduction. Right: the leaf reduction.}}
\label{fig:dip_leaf}
\end{figure}

\smallskip
{\bf (R3) The leaf reduction}
The leaf reduction occurs when a network contains a vertex $u$ incident to a single edge $e_1$ ending at a trivalent vertex (see Figure \ref{fig:dip_leaf}[right]): in this case it is possible to remove $u$ and $e_1$, disconnect $e_2$ and $e_3$, assign the color of $u$ at all newly created vertices of the edges $e_{12}$ and $e_{13}$. In the leaf reduction (R3) the only non-trivial case corresponds to the situation where the faces $f_1$, $f_2$ are distinct in the initial configuration. We assume that $e_1$ is short enough, and it does not intersect the gauge rays. If we have two faces of weights $f_1$ and $f_2$ in the initial configuration, then we merge them into a single face of weight ${\tilde f}_{12} =f_1f_2$; otherwise ${\tilde f}_{12}=f_1$ and the effect of the transformation is to create new isolated components. We also assume that the newly created vertices are close enough to $V_1$, therefore the windings are not affected. Then $E_1 = {\tilde E}_{12} + {\tilde E}_{13}$ and 
${\tilde E}_{12} = w_{1} E_2$, ${\tilde E}_{13} =  w_{1} E_3$.

\section{Existence of null edge vectors on reducible networks}\label{sec:null_vectors}

Edge vectors associated to the boundary source edges are not null due to Postnikov's results if the boundary source is not isolated. On the contrary, a component of a vector associated to an internal edge 
can be equal to zero even if the corresponding boundary sink can be reached from that edge (see Example \ref{example:null}). More in general, suppose that $E_e=0$ where $e$ is an edge ending at the vertex $V$. Then, if $V$ is black, all other edges at $V$ carry null vectors; the same occurs if $V$ is bivalent white. If $V$ is trivalent white, then the other edges at $V$ carry proportional vectors.

\begin{figure}
  \centering
	{\includegraphics[width=0.4\textwidth]{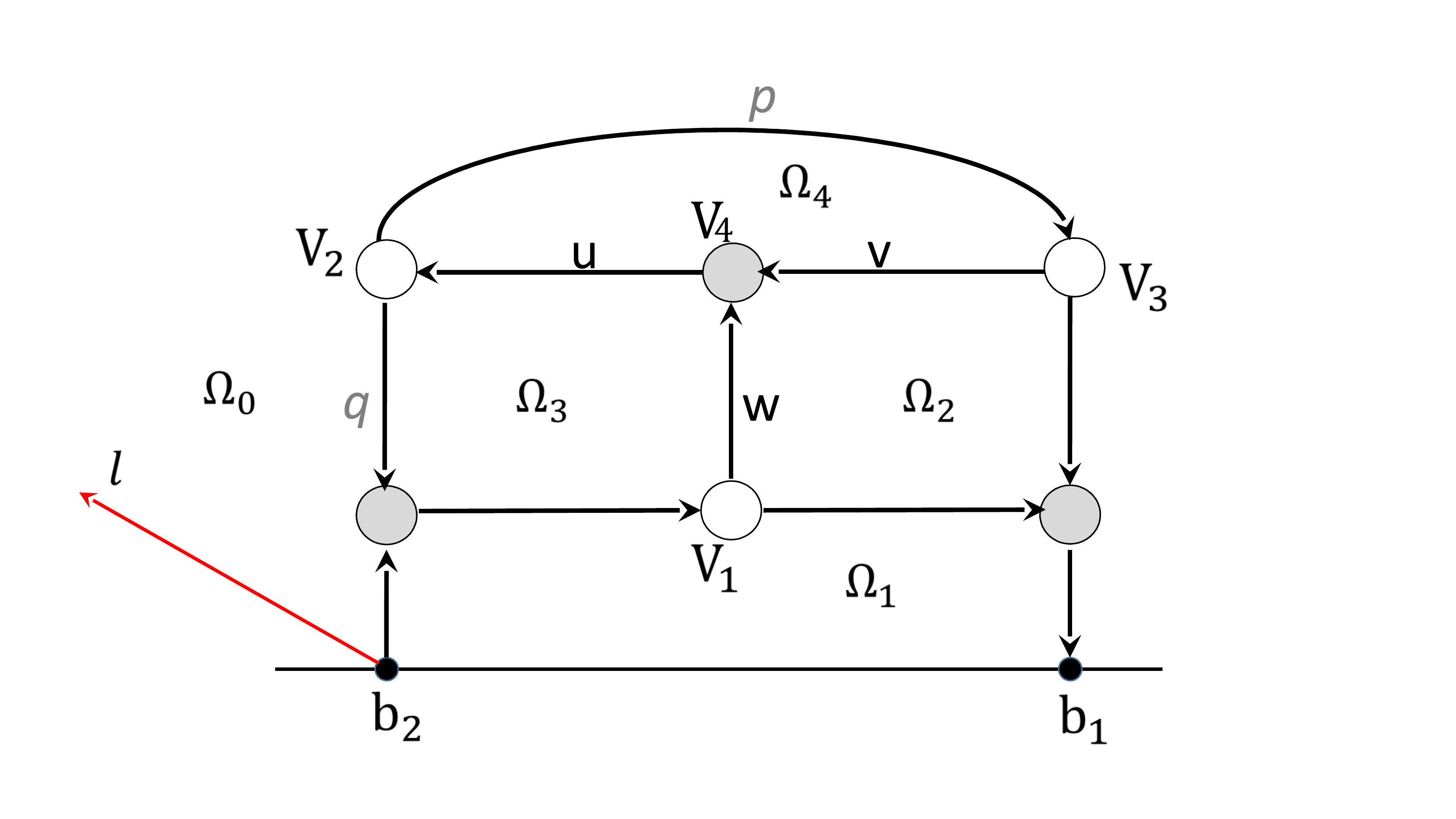}}
	\hspace{.5 truecm}
	{\includegraphics[width=0.4\textwidth]{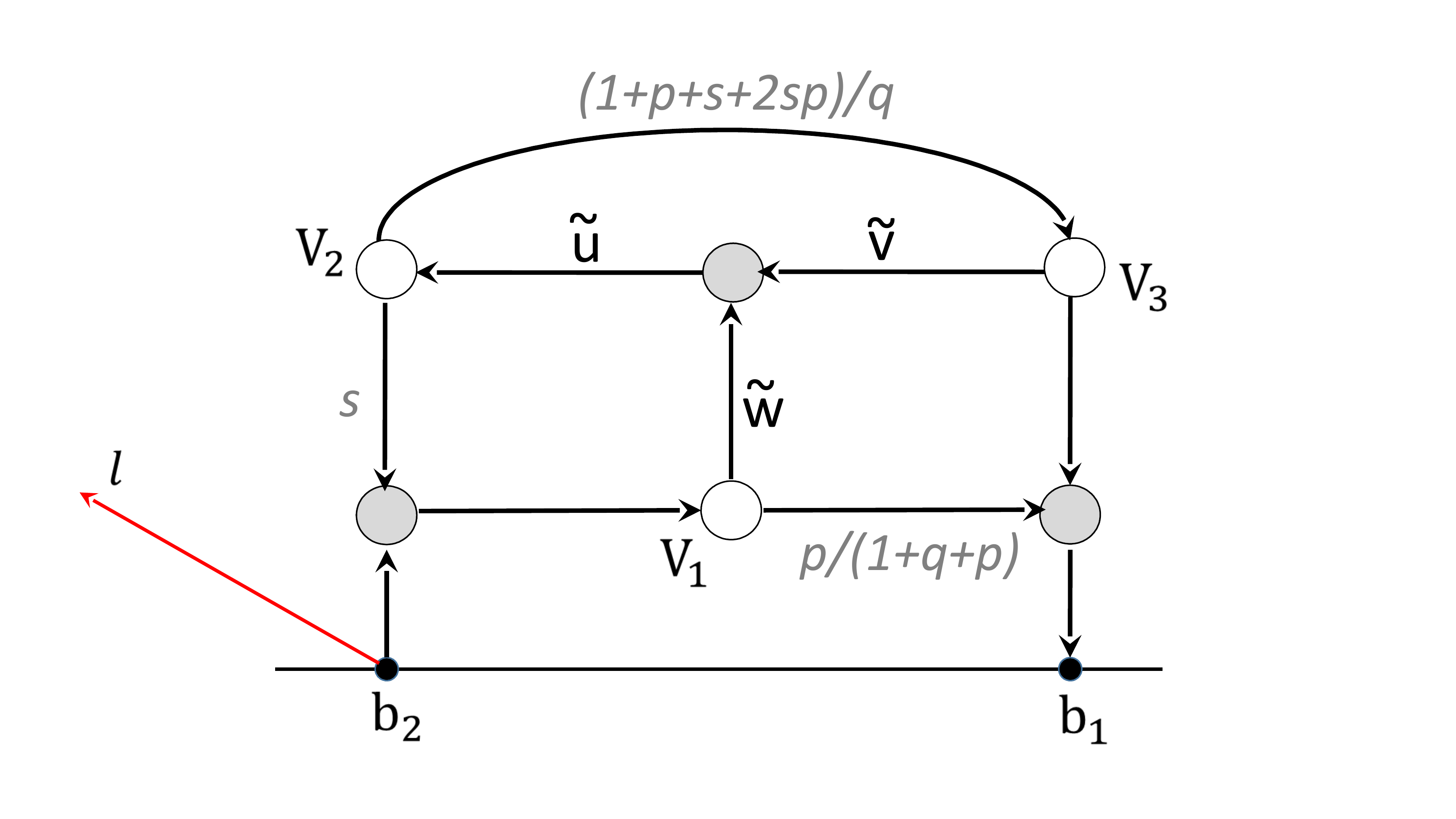}}
  \caption{\small{\sl The appearance of null vectors on reducible networks [left] and their elimination using the 
	gauge freedom for unreduced graphs of Remark \ref{rem:gauge_freedom} [right].}\label{fig:zero-vector}}
\end{figure}

Of course, if the network possesses either isolated boundary sources or components isolated from the boundary, null-vectors are unavoidable. In \cite{AG3}, we provide a recursive construction of edge vectors for canonically oriented Le--networks and obtain as a by-product that null edge vectors are forbidden if the Le--network represents a point in an irreducible positroid cell. The latter property indeed is shared by systems of edge vectors on acyclically oriented networks as a consequence of the following Theorem:

\begin{theorem}\label{thm:null_acyclic}\textbf{Edge vectors on acyclically oriented networks}
Let   $({\mathcal N}, {\mathcal O}, \mathfrak{l})$ be an acyclically oriented plabic network, which possesses neither internal sources nor sinks, representing a point in an irreducible positroid cell where $\mathcal O=\mathcal O(I)$. Then all edge vectors $E_e$ (both the internal and the boundary ones) are not-null. Moreover, in such case (\ref{eq:tal_formula}) in Theorem \ref{theo:null} simplifies to
\begin{equation}\label{eq:edge_parity}
E_{e}= \displaystyle\sum\limits_{j\in \bar I} \Big(\sum\limits_{F\in {\mathcal F}_{e,b_j}(\mathcal G)} \big(-1\big)^{\mbox{wind}(F)+\mbox{int}(F)}\ w(F)\Big) E_j
=\sum\limits_{j\in \bar I} \Big((-1)^{\sigma(e, b_j)}\sum\limits_{F\in {\mathcal F}_{e,b_j}(\mathcal G)} \ w(F)\Big) E_j,
\end{equation}
where the sum runs on all directed paths $F$ starting at $e$ and ending at $b_j$ and $\sigma(e, b_j)$ is the same for all $F$s. 
In particular (\ref{eq:edge_parity}) holds for all edges of Le-networks representing points in irreducible positroid cells.
\end{theorem}

\begin{proof}
  Let $e\in {\mathcal N}$ be an internal edge and let $b_j$ be a boundary sink such that there exists a directed path $P$ starting at $e$ and ending at $b_j$.  In order to prove (\ref{eq:edge_parity}), we need to show that $\mbox{wind}(F)+\mbox{int}(F)$ has the same parity for all directed paths from $e$ to $b_j$. Acyclicity and the absence of internal sources or sinks implies that there exists a directed path $P_0$ starting at a boundary source $b_i$ to $e$. Due to acyclicity, $P_0$ has no common edges with $P$ except $e$. Moreover, any other directed path ${\tilde P}$ from $e$ to $b_j$ may have a finite number of edges in common with $P$ and has no edge in common with $P_0$ except $e$. 

Therefore, using Corollary \ref{cor:bound_source}, we have:
\[
\begin{array}{c}
\mbox{int} (P_0)+ \mbox{int}( P) =\mbox{int} (P_0) + \mbox{int} ({\tilde P}) \quad (\!\!\!\!\!\!\mod 2),\\
\mbox{wind} (P_0) +\mbox{wind} (P) =\mbox{wind} (P_0\cup P) = \mbox{wind} (P_0\cup {\tilde P})
= \mbox{wind} (P_0) + \mbox{wind} ({\tilde P}) \quad (\!\!\!\!\!\!\mod 2),
\end{array}
\]
and the statement follows.
\end{proof}

The absence of null edge vectors for a given network is independent on its orientation and on the choice of ray direction, of vertex gauge and of weight gauge, because of the transformation rules of edge vectors established in Sections \ref{sec:gauge_ray}, \ref{sec:orient} and \ref{sec:different_gauge}. 
We summarize all the above properties of edge vectors in the following Proposition:

\begin{proposition}\label{prop:null_vectors}\textbf{Null edge vectors and changes of orientation, ray direction, weight and vertex gauges in $\mathcal N$.}
Let $({\mathcal N}, {\mathcal O}, \mathfrak{l})$ be a perfectly oriented plabic network with gauge direction $\mathfrak l$.
Let $E_e$ be its edge vector system with respect to the canonical basis at the boundary sink vertices. Then, $E_e\not =0$ on  $({\mathcal N}, {\mathcal O}, \mathfrak{l})$ if and only if $E_e\not =0$ 
on  $({\mathcal N}^{\prime}, {\mathcal O}^{\prime}, \mathfrak{l}^{\prime})$, where ${\mathcal N}^{\prime}$ is obtained from ${\mathcal N}$ changing either its orientation or the gauge ray direction or the weight gauge or the vertex gauge. 
\end{proposition}

The above statement follows from the fact that any change in the gauge freedoms - ray direction, weight gauge, vertex gauge - or in the network orientation acts by non--zero multiplicative constant on the edge vector, provided we use Lemma \ref{lemma:path} to represent the edge vectors when changing of base. 
This property suggests the fact that each graph possesses a unique system of relations up to gauge equivalence, and we shall prove that this is indeed the case in the continuation of this paper \cite{AG7}.
In the next Corollary we explicitly discuss the special case of acyclically orientable networks.

\begin{corollary}\label{cor:null_changes}\textbf{Characterization of edge vectors on acyclically orientable networks}
Let $(\mathcal N, \mathcal O, \mathfrak l)$ be an acyclically orientable plabic network representing a point in irreducible positroid cell. Then, the edge vector components are subtraction--free rational in the weights for any choice of orientation and gauge ray direction and any given change of gauge or orientation on the network acts on the right hand side of (\ref{eq:edge_parity}) with a non-zero multiplicative factor which just depends on the edge. 
\end{corollary}

\begin{proof}
If $\mathcal N$ is acyclically oriented, the statement follows from Theorem \ref{thm:null_acyclic}.
The unique case which is not completely trivial is the change of orientation along a directed path from $i_0$ to $j_0$. Then, according to the proof of Lemma~\ref{lemma:path}, $\hat E_e$ differs from $\tilde E_e$ by a non-zero multiplicative factor (see Equation (\ref{eq:hat_E_P}), and $\tilde E_e$ is just a linear combination with the same coefficients as  $E_e$ with respect to a different set of linearly independent vectors at the boundary sinks.
\end{proof}

In Section~\ref{sec:moves_reduc}, we discussed the effect of Postnikov moves and reductions on the transformation rules of edge vectors on equivalent networks. In particular, moves (M1), (M2)-flip and (M3) preserve both the plabic class and the acyclicity of the network.

\begin{corollary}\label{cor:Le_net}\textbf{Absence of null vectors for plabic networks equivalent to the Le-network.}
Let the positroid cell $\S$ be irreducible and let the plabic network  $({\mathcal N},{\mathcal O},\mathfrak l)$ represent a point in  $\S$ and be equivalent to the Le-network via a finite sequence of moves (M1), (M3) and flip moves (M2). Then ${\mathcal N}$ does not possess null edge vectors.
\end{corollary}

Null-vectors may just appear in reducible not acyclically orientable networks as in the example of Figure \ref{fig:zero-vector}. We plan to discuss thoroughly the mechanism of creation of null edge vectors in a future publication.
Edges carrying null vectors are contained in connected maximal subgraphs such that every edge belonging to one such subgraph carries a null vector and all edges belonging to its complement and having a vertex in common with it carry non zero vectors. For instance in the case of Figure \ref{fig:zero-vector}[left] there is one such subgraph and it consists of the edges $u,v,w$ and the vertices $V_1$, $V_2$, $V_3$ and $V_4$.
We conjecture that we may always choose the weights on reducible networks representing a given point so that all edge vectors are not null using the extra freedom in fixing the edge weights in reducible networks.

\begin{conjecture}\textbf{Elimination of null vectors on reducible plabic networks}\label{conj:null}
Let $({\mathcal N}, \mathcal O, \mathfrak l)$ be a reducible plabic network representing a given point $[A]\in\S \subset \GTNN$ for some irreducible positroid cell and such that it possesses a finite number of edges $e_1,\dots, e_s$ carrying null vectors,
$E_{e_l} = 0$, $l\in [s]$, and such that through each such edge there exists a path from some boundary source to some boundary sink. Then using the gauge freedom for unreduced graphs of Remark \ref{rem:gauge_freedom}, we may always change the weights on $\mathcal N$ so that the resulting network $({\tilde {\mathcal N}}, \mathcal O, \mathfrak l)$ still represents $[A]$ and the edge vectors ${\tilde E}_e \not =0$, for all $e\in {\tilde {\mathcal N}}$.
\end{conjecture}

The example in Figure \ref{fig:zero-vector} satisfies the conjecture. Both networks represent the same point 
$[ 2p/(1+p+q),1] \in Gr^{\mbox{\tiny TP}}(1,2)$, but for the second network all edge vectors are different from zero. Indeed for any choice of $s>0$, using (\ref{eq:tal_formula}), we get:
$$
E_{\tilde w}=-E_{\tilde u}=-E_{\tilde v}=\left(\frac{1+p}{1+p+q},0\right).
$$

\appendix

\section{Consistency of the system $\hat E_e$ at internal vertices}\label{app:orient}

In this Section we complete the proof of Lemmas~\ref{lemma:path} and \ref{lemma:cycle}.

Throughout this Appendix we use the same notations as in Section \ref{sec:orient}. In particular $\mathcal P_0$ is the simple path changing orientation and directed from the boundary source $b_{i_0}$ to the boundary sink $b_{j_0}$ in the initial orientation, and $\mathcal Q_0$ is the simple cycle changing orientation. 

We have to check that the system of vectors ${\hat E}_e$ defined in (\ref{eq:hat_E_P}) satisfy the linear relations at all internal vertices $V$ in the new orientation. We have to distinguish two cases:
\begin{enumerate}
\item $V$ does not belong to the path $\mathcal P_0$ (cycle $\mathcal Q_0$ respectively);
\item $V$ belongs to the path $\mathcal P_0$ (cycle $\mathcal Q_0$ respectively). 
\end{enumerate}

 Denote $\mbox{int}(e)$ and $\widehat{\mbox{int}}(e)$ the intersection number for $e$ before and after the change of orientation respectively.

Let us prove Lemmas~\ref{lemma:path},~\ref{lemma:cycle} in the first case. In this case all vectors incident to $V$ preserve the orientation, but the intersections may change. It is easy to check that ${\hat E}_e$ defined in (\ref{eq:hat_E_P}) solve the linear system at $V$ for the transformed network iff for each pair $f,g$ where $f$ is an incoming edge and $g$ is an outcoming one at $V$, one has 
\begin{equation}
\label{eq:int_vert_eq}
\mbox{int}(f)-\widehat{\mbox{int}}(f) = \gamma(f) - \gamma(g)\ \ \  (\!\!\!\!\!\mod 2).
\end{equation}
Indeed, if the orientation changes along a cycle $\mathcal Q_0$, then $\mbox{int}(f)=\widehat{\mbox{int}}(f)$ and the starting and the ending point of $f$ belong to the same region as $V$, therefore $\gamma(f) =\gamma(g)$; If the orientation changes along a path $\mathcal P_0$, $\mbox{int}(f)-\widehat{\mbox{int}}(f)$ is equal to the number of intersection of $f$ with ${\mathfrak{l}}_{i_0}$, ${\mathfrak{l}}_{j_0}$ $(\!\!\!\!\!\mod 2)$. If this number is even, the starting and the ending points of $f$ lie in regions with the same indexes, and  $\gamma(f)=\gamma(g)$. Otherwise,  $\gamma(f)=1-\gamma(g)$. 
Therefore  (\ref{eq:int_vert_eq}) is fulfilled, and the proof for vertices not belonging to $\mathcal P_0$ ($\mathcal Q_0$ respectively) is completed.

Let us prove Lemmas~\ref{lemma:path},~\ref{lemma:cycle} for $V$  belonging to $\mathcal P_0$ or $\mathcal Q_0$.  
\begin{figure}
  \centering
  {\includegraphics[width=0.49\textwidth]{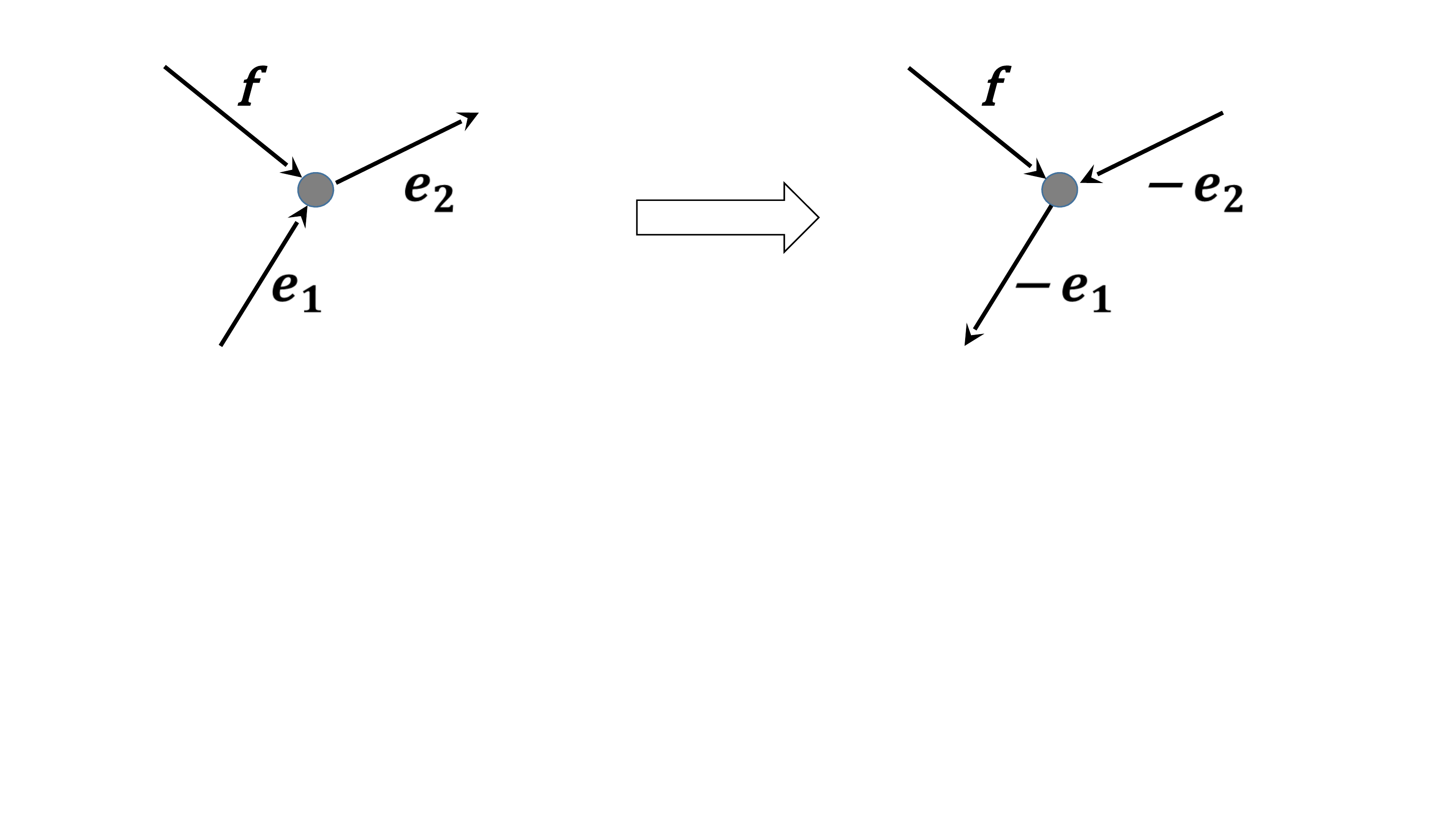}}
	\hfill
	{\includegraphics[width=0.49\textwidth]{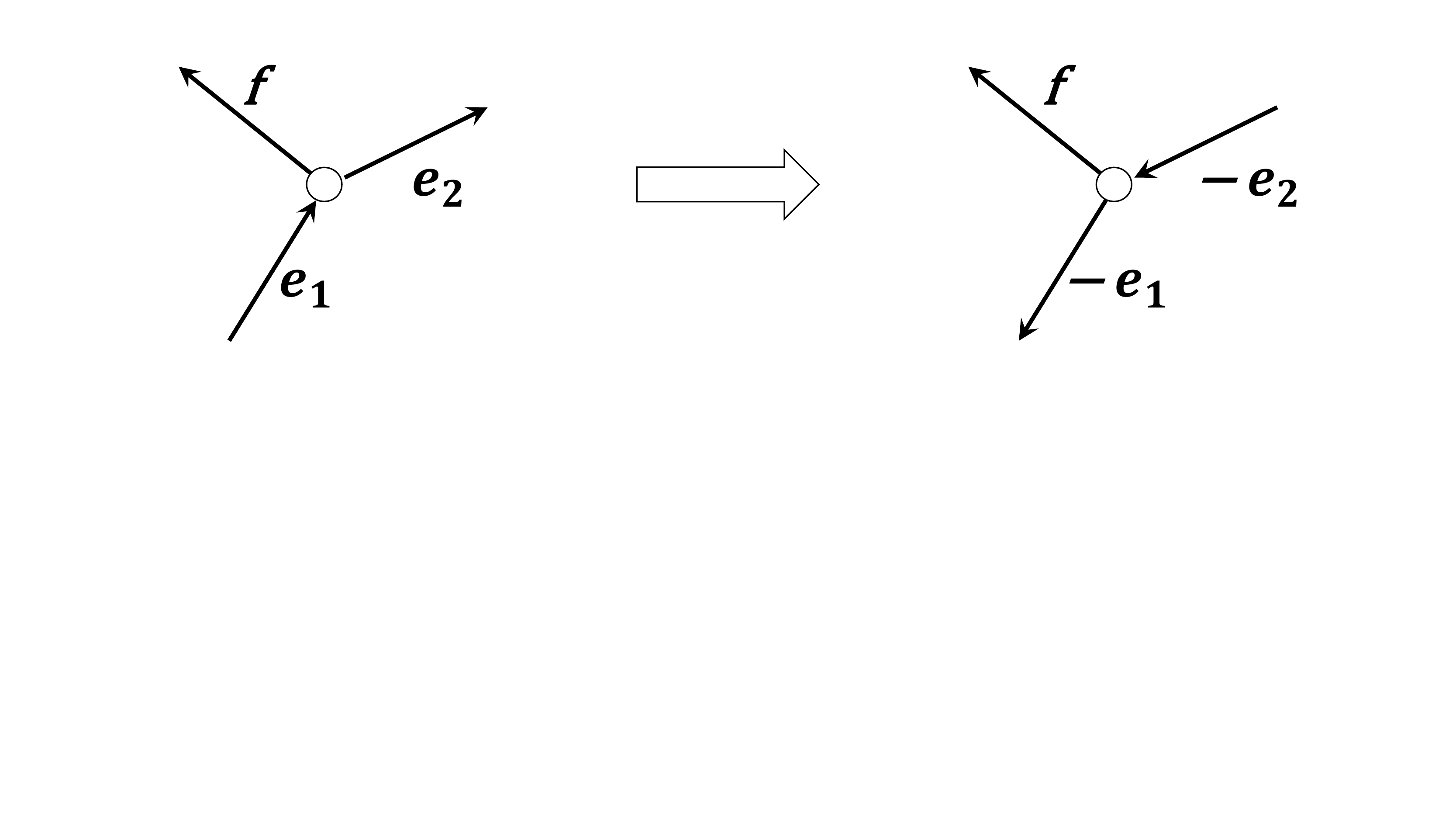}}
	\vspace{-2.3 truecm}
  \caption{Configurations at black [left] and white [right] vertices when $e_1,e_2$ belong to $\mathcal P_0$ or to $\mathcal Q_0$.}
	\label{fig:config_triv}
\end{figure}

Equations at a white $V_w$ before and after the change of variables are:
\begin{align}
\label{eq:path_w_vert_eq1}
 w_1^{-1} \tilde E_{e_1} &= (-1)^{\mbox{int}(e_1)+ \mbox{wind}(e_1,e_2)} \tilde E_{e_2} +
                   (-1)^{\mbox{int}(e_1)+ \mbox{wind}(e_1,f)} \tilde E_{f}, \\
\label{eq:path_w_vert_eq2}  
 w_2 \hat  E_{-e_2} &= (-1)^{\widehat{\mbox{int}}(e_2)+ \mbox{wind}(-e_2,-e_1)} \hat E_{-e_1}+
   (-1)^{\widehat{\mbox{int}}(e_2)+ \mbox{wind}(-e_2,f)} \hat E_{f}.                
\end{align}
Equations at a black $V_w$ before and after the change of variables are:
\begin{align}
\label{eq:path_b_vert_eq1}
 &(-1)^{\mbox{int}(e_1)+ \mbox{wind}(e_1,e_2)}  w_1^{-1} \tilde E_{e_1} &=&  &(-1)^{\mbox{int}(f)+ \mbox{wind}(f,e_2)}  w_f^{-1} \tilde E_{f}&  &=&   &\tilde E_{e_2},\\
\label{eq:path_b_vert_eq2}  
 &(-1)^{\widehat{\mbox{int}}(e_2)+ \mbox{wind}(-e_2,-e_1)}  w_2 \hat  E_{-e_2} &=&    &(-1)^{\widehat{\mbox{int}}(f)+ \mbox{wind}(f,-e_1)}  w_f^{-1} \tilde E_{f}&    &=&     &\hat E_{-e_1}.                
\end{align}
Taking into account that
\begin{equation}\label{eq:hat_E_P_bis}
{\hat E}_e = \left\{ \begin{array}{ll}
 (-1)^{\gamma(e)} {\tilde E}_e, & \mbox{ if } e\not \in \mathcal P_0 \ \ \mbox{or} \ \  \mathcal Q_0 , \mbox{ with } \gamma(e) \mbox{ as in (\ref{eq:eps_not_path})},\\
\displaystyle \frac{(-1)^{\gamma(e)}}{w_e} {\tilde E}_e, & \mbox{ if } e\in \mathcal P_0 \ \ \mbox{or} \ \  \mathcal Q_0 , \mbox{ with } \gamma(e) \mbox{ as in (\ref{eq:eps_on_path})},
\end{array}\right.
\end{equation}
where  $\tilde E =E$ if orientation changes along  $\mathcal Q_0$, the statement of Lemmas~\ref{lemma:path},~\ref{lemma:cycle} is equivalent to the following identities:
\begin{align}
  \label{eq:path_b_vert_eq3}
  \mbox{int}(e_1)+ \mbox{wind}(e_1,e_2)+ \widehat{\mbox{int}}(e_2)+ \mbox{wind}(-e_2,-e_1) + \gamma(e_1) + \gamma(e_2) =0  \ \ (\!\!\!\!\!\mod 2), \\
   \label{eq:path_b_vert_eq4}    
  \mbox{int}(e_1)+ \mbox{wind}(e_1,f) + \mbox{wind}(-e_2,f) + \mbox{wind}(-e_2,-e_1) +\gamma(e_1) + \gamma(f) = 1  \ \ (\!\!\!\!\!\mod 2), \\
 \mbox{int}(e_1)+ \mbox{wind}(e_1,e_2) +  \mbox{int}(f)+ \mbox{wind}(f,e_2) + \widehat{\mbox{int}}(f)+ \mbox{wind}(f,-e_1) +\gamma(e_1) + \gamma(f) = 0  \ \ (\!\!\!\!\!\mod 2).
 \label{eq:path_w_vert_eq5}               
\end{align}
Using the definition of the index $\gamma()$ in (\ref{eq:eps_not_path}) and  (\ref{eq:eps_on_path}), the left-hand side of Equation (\ref{eq:path_b_vert_eq3}) can be rewritten as:
\begin{equation}
\mbox{int}(e_1)+ \mbox{wind}(e_1,e_2)+ \widehat{\mbox{int}}(e_2)+ \mbox{wind}(-e_2,-e_1) + \gamma(e_1) + \gamma(e_2) =
\end{equation}
$$
= [ \widehat{\mbox{int}}(e_2) +  \mbox{int}(e_2) + \gamma_1(e_1) + \gamma_1(e_2)]  + [ \mbox{wind}(e_1,e_2) + \mbox{wind}(-e_2,-e_1)+  \gamma_2(e_1) + \gamma_2(e_2) ].
$$
From (\ref{eq:int_vert_eq}) it follows that the first parentheses equals to 0 $(\!\!\!\!\!\mod 2)$, therefore it is sufficient to check that 
\begin{equation}
\label{eq:a10}  
[ \mbox{wind}(e_1,e_2) + \mbox{wind}(-e_2,-e_1)+  \gamma_2(e_1) + \gamma_2(e_2) ] = 0  \ \ (\!\!\!\!\!\mod 2).
\end{equation}
Indeed, if $e_1$ and $e_2$ belong to the same half-plane with respect to $\mathfrak l$, then $ \gamma_2(e_1) + \gamma_2(e_2)=0 \ \ (\!\!\!\!\!\mod 2)$ and  $\mbox{wind}(e_1,e_2) = \mbox{wind}(-e_2,-e_1)=0$. If they belong to opposite half-planes, then  $ \gamma_2(e_1) + \gamma_2(e_2)=1$, one of the windings  $\mbox{wind}(e_1,e_2)$, $\mbox{wind}(-e_2,-e_1)=0$ equals to $\pm1$ and the other is zero (see Fig~\ref{fig:a10}). It proves (\ref{eq:a10})
\begin{figure}
  \centering
  \includegraphics[width=0.3\textwidth]{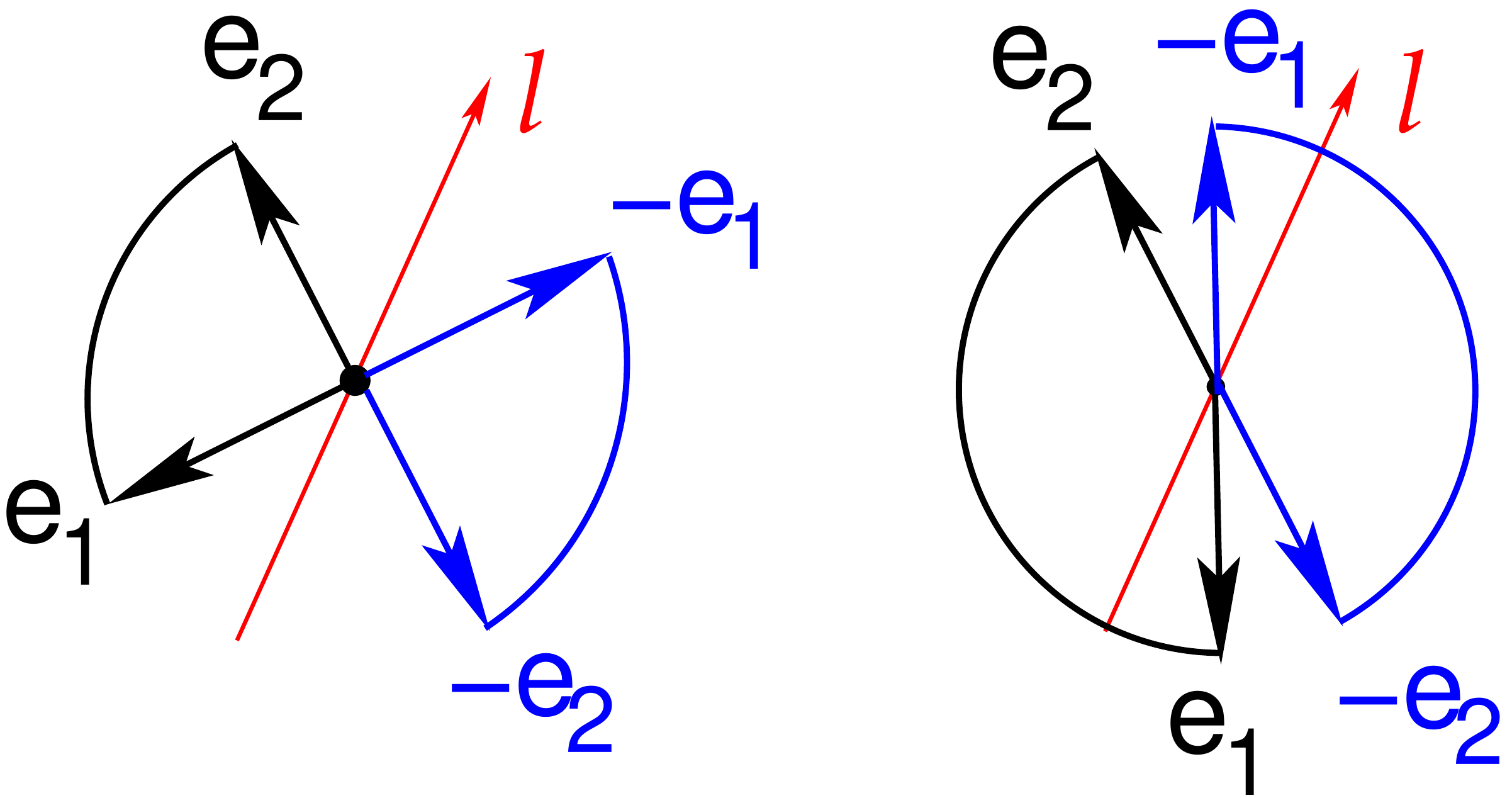}
  \caption{On the left $e_1$ and $e_2$ belong to the same half-plane, $\gamma_2(e_1)= \gamma_2(e_2)=1$,  $\gamma_2(-e_1)= \gamma_2(-e_2)=0$,  $\mbox{wind}(e_1,e_2)=\mbox{wind}(-e_2,-e_1)=0$. On the right $e_1$ and $e_2$ belong to opposite half-planes, $\gamma_2(e_1)= \gamma_2(-e_2)=0$,  $\gamma_2(-e_1)= \gamma_2(e_2)=1$, $\mbox{wind}(e_1,e_2)=0$, $\mbox{wind}(-e_2,-e_1)=1$.}
	\label{fig:a10}
\end{figure}

Similarly, the left-hand sides of Equations (\ref{eq:path_b_vert_eq4}), (\ref{eq:path_w_vert_eq5})  can be rewritten as:
\begin{equation}
\mbox{int}(e_1)+ \mbox{wind}(e_1,f) + \mbox{wind}(-e_2,f) + \mbox{wind}(-e_2,-e_1) +\gamma(e_1) + \gamma(f) = 
\end{equation}
$$
= [ \mbox{wind}(e_1,f) + \mbox{wind}(-e_2,f) + \mbox{wind}(-e_2,-e_1) +\gamma_2(e_1)] +  [\gamma_1(e_1) +\gamma(f) ],
$$
\begin{equation}
\mbox{int}(e_1)+ \mbox{wind}(e_1,e_2) +  \mbox{int}(f)+ \mbox{wind}(f,e_2) + \widehat{\mbox{int}}(f)+ \mbox{wind}(f,-e_1) +\gamma(e_1) + \gamma(f)=
\end{equation}
$$
= [\mbox{wind}(e_1,e_2) + \mbox{wind}(f,e_2) +  \mbox{wind}(f,-e_1) +\gamma_2 (e_1)] +
[\mbox{int}(f) + \widehat{\mbox{int}}(f) + \gamma_1(e_1) + \gamma(f)].
$$

To complete the proof, we use the cyclic order:

\begin{definition}\textbf{Cyclic order.}\label{def:cyclic_order} Generic triples of vectors in the plane have natural cyclic order. We write $[f,g,h]=0$ if the triple $f$, $g$, $h$ is ordered counterclockwise, and $[f,g,h]=1$ if the triple $f$, $g$, $h$ is ordered clockwise (see Fig~\ref{fig:cyclic_order}).
\end{definition}  
\begin{figure}
  \centering
  \includegraphics[width=0.25\textwidth]{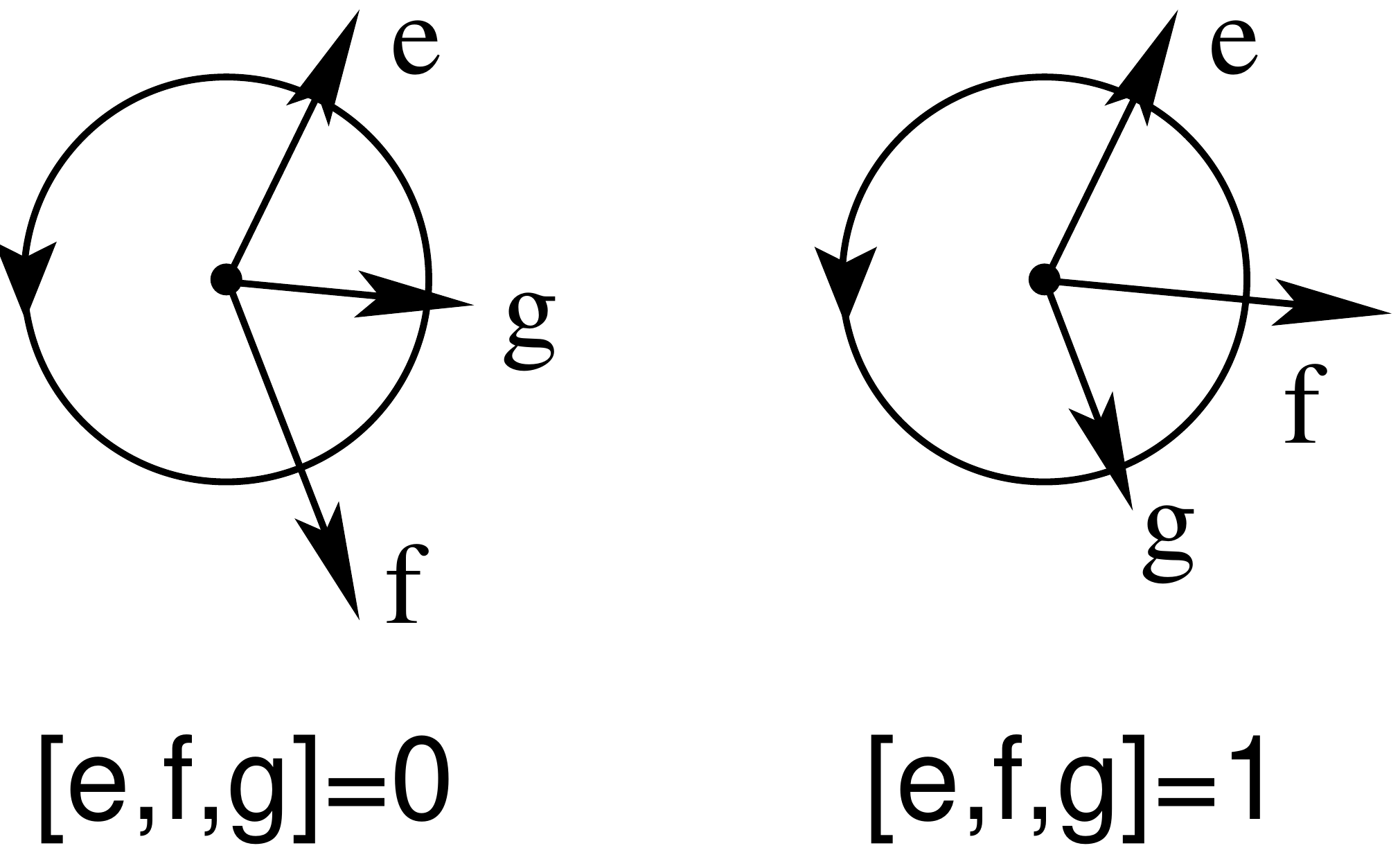}
  \caption{Cyclic order on triples of vectors. By definition $[e,f,g]=[f,g,e]=[g,e,f]=1-[e,g,f]=1-[g,f,e]=1-[f,e,g]$.}
	\label{fig:cyclic_order}
\end{figure}

Then the proof of  Lemmas~\ref{lemma:path},~\ref{lemma:cycle} in the second case immediately follows from Lemma~\ref{lemma:equiv_rel}.

\begin{lemma}
  \label{lemma:equiv_rel}
  Let $V$ belong to $\mathcal P_0$ or $\mathcal Q_0$, $e_1,e_2,f$ be the edges at $V$, where $e_1$, $e_2$ belong to  $\mathcal P_0$ ($\mathcal Q_0$),  and $e_1$ ($e_2$) denotes an incoming (outgoing) edge respectively in the initial configuration (see Figure \ref{fig:config_triv}). 
\begin{enumerate}
\item If $V$ is black, then:
  \begin{align}
 \label{eq:black1}   
\widehat{\mbox{int}}(f) + \mbox{int}(f) + \gamma(f) +\gamma_1(e_1) &= [e_1,-e_2,f] \quad
                                                                         (\!\!\!\!\!\!\mod 2),\\
 \label{eq:black2}      
 \mbox{wind}(e_1,e_2) + \mbox{wind}(f,e_2) +  \mbox{wind}(f,-e_1) +\gamma_2 (e_1) &=  [e_1,-e_2,f] \quad
(\!\!\!\!\!\!\mod 2). 
\end{align}
\item If $V$ is white, then:
  \begin{align}
   \label{eq:white1}    
\gamma(f) +\gamma_1(e_1) &= [e_1,-e_2,-f] \quad
                               (\!\!\!\!\!\!\mod 2),\\
  \label{eq:white2}     
  \mbox{wind}(e_1,f) + \mbox{wind}(-e_2,f) + \mbox{wind}(-e_2,-e_1) +\gamma_2(e_1)&= 1-[e_1,-e_2,-f] \quad
(\!\!\!\!\!\!\mod 2). 
\end{align}
\end{enumerate}
\end{lemma}

\begin{proof}
  In the proof all identities hold mod 2.
  
  To prove (\ref{eq:black1}), let us remark that $\widehat{\mbox{int}}(f) + \mbox{int}(f)=0 $ if and only if
  both the starting and the ending points of $f$ lie in regions equally marked. Moreover
  $\gamma(f) +\gamma_1(e_1)=0$ if and only if the starting point of $f$ lies in area with the same marking as the area at the left of the ending point of $e_1$. Therefore  $\widehat{\mbox{int}}(f) + \mbox{int}(f)+\gamma(f) +\gamma_1(e_1)=0$ if vector $f$ lies to the left of the directed chain $e_1,e_2$, otherwise this expression equals 1. But vector $f$ lies to the left (right) of the directed chain $e_1,e_2$ if and only if $[e_1,-e_2,f]=0$ (1 respectively).

  To prove (\ref{eq:white1}), let us remark that $\gamma(f) +\gamma_1(e_1)=0$ if and only if the starting point of $f$ lies in area at the left of the ending point of $e_1$.  In this case vector $f$ lies to the left (right) of the directed chain $e_1,e_2$ if and only if $[e_1,-e_2,-f]=0$ (1 respectively).

It is evident that  (\ref{eq:black2}), (\ref{eq:white2})  are true for configurations presented at Figure~\ref{fig:bw_winding}:
  \begin{figure}
  \centering
  \includegraphics[width=0.8\textwidth]{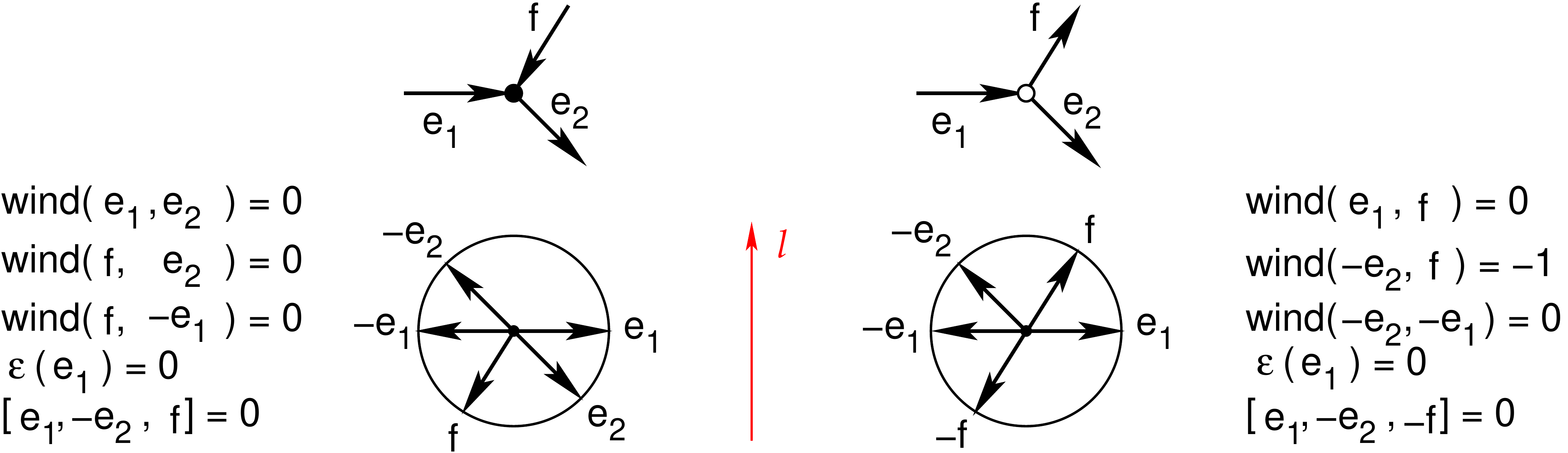}
  \caption{Check of  (\ref{eq:black2}), (\ref{eq:white2}) for one configuration.}
	\label{fig:bw_winding}
\end{figure}

Lemma~\ref{lem:rotation} implies that they hold true for all configurations and gauge ray directions:
  \begin{enumerate}
  \item The right-hand side of  (\ref{eq:black2}), (\ref{eq:white2}) does not depend on $\mathfrak l$. The terms in the left-hand side change when $\mathfrak l$ passes $f$, $e_2$, $e_1$, $-e_1$ ( $f$, $-e_2$, $e_1$, $-e_1$) if $V$ is black (white) respectively. If $\mathfrak l$ passes either $f$ or $e_2$ ($-e_2$),  then two winding terms change by 1 contemporarily. If $\mathfrak l$ passes either $e_1$ or $-e_1$, exactly one winding term and $\gamma_1(e_1)$ change. In all cases equality remains true;
  \item If we keep $\mathfrak l$ and all vectors except $f$ fixed, both the right-hand side and the left-hand side change by 1 when $f$ ($-f$) passes $e_1$, $-e_2$, and remain fixed in all other cases, therefore the equality remains true;
  \item  If we keep $\mathfrak l$ and all vectors except $e_2$ fixed, both the right-hand side and the left-hand side change by 1 when $e_2$ passes $-e_1$, $-f$ ($f$), and remain fixed in all other cases, therefore the equality remains true.
   \end{enumerate} 
\end{proof}

\section*{Acknowledgments}  

The authors would like to express their gratitude to T. Lam for pointing our attention to Reference \cite{AGPR}.


\begin{thebibliography}{00}

\bibitem{A3} S. Abenda,   ``Kasteleyn theorem, geometric signatures and KP-II divisors on planar bipartite networks in the disk'', \textit{Math. Phys. Anal. Geom.} \textbf{24}, (2021), Art. \#35, 64 pp.
  
\bibitem{AG1} Abenda, S., and P.G. Grinevich, ``Rational degenerations of $M$-curves, totally positive Grassmannians and KP--solitons.'' \textit{Commun. Math. Phys.} \textbf{361}, no. 3 (2018): 1029--1081.

\bibitem{AG2} Abenda, S., and P.G. Grinevich, ``Real soliton lattices of the Kadomtsev-Petviashvili II equation and desingularization of spectral curves corresponding to $Gr^{\mbox{\tiny TP}}(2,4)$.''  \textit{Proc. Steklov Inst. Math.} \textbf{302}, no. 1 (2018): 1--15. 

\bibitem{AG3} Abenda, S., and P.G. Grinevich, ``Reducible $M$-curves for Le-networks in the totally-nonnegative Grassmannian and KP--II multiline solitons.'' \textit{Sel. Math. New Ser.} \textbf{25}, no. 3 (2019) 25:43. https://doi.org/10.1007/s00029-019-0488-5

\bibitem{AG6} Abenda, S., and P.G. Grinevich, \textit{Real regular KP divisors on $\mathtt M$--curves and totally non-negative Grassmannians}, ArXiv:2002.04865.

\bibitem{AG7} Abenda, S., and P.G. Grinevich, \textit{Geometric nature of relations on plabic graphs and totally non-negative Grassmannians}, ArXiv:2111.05782.

\bibitem{AGPR} Affolter, N., M. Glick, P. Pylyavskyy, and S. Ramassamy, ``Vector--relation configurations and plabic graphs'', \textit{S\'eminaire Lotharingien de Combinatoire} \textbf{84B} (2020), Art. \#91, 12 pp. 
  
\bibitem{AGP1} Arkani--Hamed, N., J.L. Bourjaily, F. Cachazo, A.B. Goncharov, A. Postnikov, and J. Trnka, \textit{Scattering Amplitudes and the Positive Grassmannian.}, arXiv:1212.5605.

\bibitem{AGP2} Arkani--Hamed, N., J.L. Bourjaily, F. Cachazo, A.B. Goncharov, A. Postnikov, and J. Trnka, \textit{Grassmannian geometry of scattering amplitudes.} Cambridge University Press, Cambridge, 2016.

\bibitem{ADM} Atiyah, M., M. Dunajski, and L.J. Mason, ``Twistor theory at fifty: from contour integrals to twistor strings.'' \textit{Proc. R. Soc. A.}  473 (2017):  20170530, 33 pp.

\bibitem{BFGW} Bourjaily, J.L., S. Franco, D. Galloni, and C. Wen, ``Stratifying on--shell cluster varieties: the geometry of non--planar on--shell diagrams.'' \textit{J. High Energy Phys.} (2016), no. 10, 003, front matter+30 pp.

\bibitem{BG} Buchstaber, V., and A. Glutsyuk, \textit{Total positivity, Grassmannian and modified Bessel functions.} 
arXiv:1708.02154.

\bibitem{CK} Chakravarty, S., and Y. Kodama, ``Soliton solutions of the KP equation and application to shallow water waves.'' \textit{Stud. Appl. Math.} 123 (2009): 83--151.

\bibitem{CW} Corteel, S., and L.K. Williams, ``Tableaux combinatorics for the asymmetric exclusion process.'' \textit{Adv. in Appl. Math.} 39, no. 3 (2007): 293--310.

\bibitem{FG1} Fock, V.V., and A. B. Goncharov, ``Cluster ${\mathcal X}$--Varieties, Amalgamation and
Poisson-Lie Groups.'', in Algebraic Geometry and Number Theory, dedicated to Drinfeld's 50th birthday, pp. 27--68, \textit{Progr. Math.} 253, Birkhauser, Boston, 2006.

\bibitem{Fom} Fomin, S., ``Loop--erased walks and total positivity.'' \textit{Trans. of the AMS} 353, no. 9 (2001): 3563--3583.

\bibitem{FPS} Fomin, S., P. Pylyavskyy, and E. Shustin, \textit{Morsifications and mutations.} arXiv:1711.10598 (2017).

\bibitem{FZ1} Fomin, S., and A. Zelevinsky, ``Double Bruhat cells and total positivity.'' \textit{J. Amer. Math. Soc.} 12 (1999): 335--380.

\bibitem{FZ2} Fomin S., and A. Zelevinsky, ``Cluster algebras I: Foundations.'' \textit{J. Am. Math. Soc.} 15 (2002): 497--529.

\bibitem{FGPW} Franco S., Galloni D., Penante, B., Wen, Congkao W., ``Non-planar on-shell diagrams.'' \textit{J. High Energy Phys.}, \textbf{2015}, no. 6,  Art. \#199, 71 pp. 

\bibitem{GKL} Galashin, P., S.N. Karp,  and T. Lam,  ``The totally nonnegative Grassmannian is a ball.'' \textit{S\'eminaire Lotharingien de Combinatoire} 80B (2018): Article \#23, 12 pp.

\bibitem{GK} Gantmacher, F.R., and M.G. Krein, ``Sur les matrices oscillatoires.'' \textit{C.R. Acad. Sci. Paris}
201 (1935): 577--579.

\bibitem{GK2} Gantmacher, F.R., and M.G. Krein, \textit{Oscillation Matrices and Kernels and Small Vibrations of Mechanical Systems.} (Russian), Gostekhizdat, Moscow-
Leningrad, (1941), second edition (1950), English edition from AMS Chelsea Publ. (2002).

\bibitem{GSV1} Gekhtman, M.,  M. Shapiro, and A. Vainshtein, ``Poisson geometry of directed networks in a
disk'', \textit{Selecta Math.} 15 (2009): 61--103

\bibitem{GSV}  Gekhtman, M., M. Shapiro, and A. Vainshtein, \textit{Cluster algebras and Poisson geometry.} Mathematical Surveys and Monographs, 167. American Mathematical 
Society, Providence, RI, (2010), xvi+246 pp.

\bibitem{GSV2} Gekhtman M., M. Shapiro, and A. Vainshtein, ``Poisson Geometry of Directed Networks
in an Annulus.'' \textit{J. of the Europ. Math. Soc.} 14 (2012): 541--570.

\bibitem{GGMS} Gel'fand, I.M., R.M. Goresky, R.D. MacPherson, and V.V. Serganova, ``Combinatorial geometries, convex polyhedra, and Schubert cells.'' \textit{Adv. in Math.}  63, no. 3 (1987): 301--316.

\bibitem{GS} Gel'fand, I.M., and V.V. Serganova, ``Combinatorial geometries and torus strata on homogeneous compact manifolds.'' \textit{Russian Mathematical Surveys} 42, no. 2 (1987): 133--168.

\bibitem{Kap} Kaplan, J., ``Unraveling ${\mathcal L}_{n;k}$ Grassmannian Kinematics.'' \textit{J. High Energy Phys.} 2010, no. 3, 025, (2010) 34 pp. 

\bibitem{Kar} Karlin, S.,  \textit{Total Positivity, Vol. 1.} Stanford, 1968. 

\bibitem{Kas1} P.W. Kasteleyn, {\em The statistics of dimers on a lattice.I. The number of dimer arrangements on a quadratics lattice}, Physica {\bf 27} (1961), 1209-1225.

\bibitem{KW1} Kodama, Y. and L.K. Williams, ``The Deodhar decomposition of the Grassmannian and the regularity of KP solitons.'' \textit{Adv. Math.} 244 (2013): 979--1032.

\bibitem{KW2} Kodama, Y. and L.K. Williams, ``KP solitons and total positivity for the Grassmannian.'' \textit{Invent. Math.} 198 (2014) 637--699.


\bibitem{Lam1} Lam, T., ``Dimers, webs, and positroids.'', \textit{J. Lond. Math. Soc.} (2) 92, no. 3 (2015): 633--656.

\bibitem{Lam2} Lam, T., \textit{Totally nonnegative Grassmannian and Grassmann polytopes.}, Current developments in mathematics 2014, 51--152, Int. Press, Somerville, MA, 2016.

\bibitem{Law} Lawler, G., \textit{Intersections of random walks.} Birkh\"auser, 1991.

\bibitem{Lus1} Lusztig, G.,  ``Total positivity in reductive groups.'' \textit{Lie Theory and Geometry: in honor
of B. Kostant}, Progress in Mathematics 123, Birkh\"auser, Boston, 1994, 531--568.

\bibitem{Lus2} Lusztig, G., ``Total positivity in partial flag manifolds.'' \textit{Representation Theory} 2 (1998),
  70--78.

\bibitem{Mach} Machacek, John. ``Boundary measurement matrices for directed networks on surfaces.'' \textit{Adv. in Appl. Math.} \textbf{93} (2018), 69--92.


\bibitem{MS} Mason, L., and  D. Skinner, ``Dual Superconformal Invariance, Momentum Twistors
and Grassmannians.'' \textit{J. High Energy Phys.} 2009, no. 11, 045, (2009), 39 pp.

\bibitem{OPS} Oh, S., A. Postnikov, and D.E. Speyer, ``Weak separation and plabic graphs.'' \textit{Proc. Lond. Math. Soc.} (3) 110, no. 3 (2015): 3, 721--754.

\bibitem{Pos} Postnikov, A., \textit{Total positivity, Grassmannians, and networks.}, arXiv:math/0609764 [math.CO].

\bibitem{PSW} Postnikov, A., D. Speyer, and L. Williams, ``Matching polytopes, toric geometry, and the totally non-negative Grassmannian.'' \textit{J. Algebraic Combin.} 30, no. 2 (2009): 173--191.

\bibitem{Rie} Rietsch, K.,  ``An algebraic cell decomposition of the nonnegative part of a flag variety.''
\textit{Journal of Algebra} 213, no. 1 (1999): 144--154.

\bibitem{RW} Rietsch, K., and L. Williams, ``The totally nonnegative part of G/P is a CW complex-'', \textit{Transform. Groups} 13, no- 3--4 (2008): 839--853.

\bibitem{Sch} Schoenberg, I., ``\"Uber variationsvermindende lineare Transformationen.'' \textit{Math. Zeit.} 32, (1930): 321--328.

\bibitem{Sc} Scott J.S., ``Grassmannians and cluster algebras.'' \textit{Proc. London Math. Soc.} 92 (2006): 345--380.


\bibitem{Sp} Speyer, D.E. ``Variations on a theme of Kasteleyn, with application to the totally nonnegative Grassmannian.'' \textit{Electron. J. Combin.} 23, no. 2 (2016) Paper 2.24, 7 pp.

\bibitem{Tal2} Talaska, K.,  ``A Formula for Pl\"ucker Coordinates Associated with a Planar Network.'' \textit{IMRN} 2008, (2008),  Article ID rnn081, 19 pages.


\end{thebibliography}
\end{document}